\documentclass{amsart}
\usepackage{graphicx}

\usepackage[cmtip,all]{xy}
\usepackage{color}
\usepackage[hidelinks]{hyperref}
\hypersetup{pdftitle={Wintner-Conley Dimension, Plucker Coordinates, and Generalized Dziobek-Williams Equations for Central Configurations},pdfauthor={Thiago Dias}}
\usepackage{amsmath, amssymb}

\usepackage{xcolor}
 
\newtheorem{theorem}{Theorem}[section]
\newtheorem{proposition}[theorem]{Proposition}

\newtheorem{lemma}[theorem]{Lemma}

\newtheorem{corollary}[theorem]{Corollary}
\newtheorem{fact}[theorem]{Fact}
\newtheorem*{mainthm}{Main Theorem}

\theoremstyle{definition}
\newtheorem{definition}[theorem]{Definition}
\newtheorem{example}[theorem]{Example}

\theoremstyle{remark}
\newtheorem{remark}[theorem]{Remark}

\numberwithin{equation}{section}

\begin{document}

%\title{}


%    author one information
\title[Brehm--Wintner--Conley dimension and Dziobek--Williams equations]{Brehm--Wintner--Conley Dimension, Pl\"ucker Coordinates, and Generalized Dziobek--Williams Equations for Central Configurations}
\author[Dias]{Thiago Dias}

\address{Departamento de Matemática, Universidade Federal Rural de Pernambuco - Rua Dom Manuel de Medeiros s/n, 52171-900, Recife, Pernambuco, Brasil}
\email{thiago.diasoliveira@ufrpe.br}

%    author two information
%\author{}
%\address{}
%\curraddr{}
%\email{}
%\thanks{}

%    \subjclass is required.
\date{August 2026}

%\dedicatory{}

%    Abstract is required.
\begin{abstract}
We develop an algebraic framework for central configurations of the $n$-body problem with homogeneous potentials, grounded in the exterior algebra of the configuration space and the normalized shifted Brehm--Wintner--Conley (BWC) matrix $S$. Relating the kernel of $S$ to the Pl\"ucker coordinates of the configuration, we generalize the determinantal equations obtained by Williams (1938) for the planar five-body problem to central configurations of any dimension and any number of bodies, and derive the Dziobek--Williams equations $\det(S_I^J)=\kappa\, z_I z_J$, which exhibit the compound matrix $S^{(t)}$ as a rank-one matrix. Introducing the \emph{Brehm--Wintner--Conley dimension} $\operatorname{bwc}(x)=\operatorname{rank}(S)$ (an integer invariant that stratifies central configurations and measures vertical degeneracy), together with the Pl\"ucker--BWC coordinates attached to it, we prove that each stratum of central configurations with fixed dimension and Brehm--Wintner--Conley dimension admits a base-point-free map into a Veronese variety, factoring through a Grassmannian invariant; on the Dziobek stratum this recovers the Dziobek--Veronese geometry previously introduced by the author. We further describe universal determinantal relations satisfied by the minors of $S$ and expand explicitly the resulting systems for the planar five- and six-body problems. We also interpret the mass-weighted entries of $S$ as an equilibrium stress: under the MacMillan--Bartky sign condition a strictly convex central configuration underlies a cable--strut tensegrity, and a theorem of Connelly then forces $S$ to be negative semidefinite with nullity three, so that vertical degeneracy cannot occur in this regime.
\end{abstract}

\maketitle
\enlargethispage{4pt}


\section{Introduction}

A configuration of $n$ point masses is \emph{central} when the gravitational
acceleration of each body points toward the center of mass, with a proportionality
constant common to all bodies. Central configurations (CCs) are the source of
the only explicitly known solutions of the $n$-body problem: they generate
homographic and self-similar motions, and, in the plane, the relative
equilibria which rotate rigidly. They also govern the local structure of
collision and escape, and organize the topology of the integral manifolds
\cite{smale1998mathematical}. Since Euler \cite{euler1767demotu}, Lagrange
\cite{lagrange1772essai} and Dziobek
\cite{dziobek1900ueber}, their classification has resisted every general
approach, and Smale placed the finiteness of the number of central
configurations, for arbitrary positive masses, on his list of problems for the
twenty-first century \cite{smale1998mathematical}. The conjecture is settled
for $n=4$ by Hampton and Moeckel \cite{hamptonmoeckel2006}, and for $n=5$ and
generic masses by Albouy and Kaloshin \cite{albouy2012finiteness}. For six
bodies, Chang and Chen \cite{changchen2024toward,changchen2025toward} push
the Albouy--Kaloshin method forward by symbolic computation: a matrix algebra
determines the admissible $zw$-diagrams together with their asymptotic
orders, and mass relations then exclude, outside a codimension-two subvariety
of the mass space, all but a residual list of diagrams. Generic finiteness on
the Dziobek stratum $k=n-2$ is due to Moeckel \cite{moeckel2001generic}; the
extension to arbitrary homogeneous potentials in \cite{DiasVeronese}
identifies Dziobek's relations with the defining quadrics of a Veronese
variety and obtains from this identification, through the theorem on the
dimension of fibers, a uniform bound of Bezout type,
$2^{\binom{n+1}{2}+n-1}$, on the number of Dziobek configurations; for $n=4$
this gives $8192$, below the estimate $8472$ of \cite{hamptonmoeckel2006}. That the question is delicate rather than merely hard is
shown by Roberts' continuum of relative equilibria in a five-body problem with
one negative mass \cite{roberts1999}: finiteness cannot follow from soft
arguments alone.

Progress of this kind rests, in every instance, on the prior derivation of
workable equations, and that derivation is a subject in its own right. In a
very recent contribution, Leandro \cite{leandro2026moments} develops the
classical theory of moments of weighted point configurations into a unified
mechanism producing homogeneous, isometry-invariant equilibrium equations by
elementary algebraic manipulations, with neither reduction by isometries nor
a variational principle: the Albouy--Chenciner system is recovered, new
families of equations in the mutual distances are introduced, and the
Cayley--Menger-type constraints for configurations of prescribed dimension
are re-derived and extended. The present paper is a contribution to the same
enterprise.

\subsection*{Mutual distances and the Cayley--Menger constraints}

The dominant technique is to work in mutual distances $r_{ij}$, which
eliminate translations and rotations at the cost of introducing the
Cayley--Menger constraints. This is the setting of the Albouy--Chenciner
equations \cite{albouy1997probleme,albouy2003paper} and of most subsequent
progress \cite{moeckel1985relative,dias2017new,albouy2024limit}. In
codimension one (Dziobek configurations, $k=n-2$), the constraint is a
single equation, the vanishing of the Cayley--Menger determinant
\cite{moeckel2001generic}, and the classical theory is complete: Dziobek's relations
$s_{ij}s_{kl}=s_{ik}s_{jl}$ close the system, and the geometry is that of a
quadric \cite{dziobek1900ueber,DiasVeronese}.

Already in codimension two the picture changes. For the planar five-body
problem, Williams \cite{williams1938permanent} derived in 1938 a system of
determinantal equations, long neglected, resting on an unsupported positivity
assumption that was only removed by Chen and Hsiao
\cite{chen2018strictly}. The situation is telling: a system of equations for
the first genuinely non-Dziobek case sat in the literature for eighty years
without its determinantal nature being recognized. As Albouy and Sun recently
observed \cite{albouysun2025}, Williams' quantities are nothing but the
$2\times2$ minors of the matrix governing the configuration, a coincidence
that had gone unnoticed, and which suggests organizing the distances
through the minors of a single matrix, rather than working with them one
pair at a time.

\subsection*{Statement of results}

We take that suggestion seriously and rebuild the theory in the exterior
algebra of the configuration space. The two objects are the \emph{Pl\"ucker
coordinates} $\Delta_I=\omega(x_I)$ of the configuration (the oriented
volumes of its $(k+1)$-body subconfigurations) and the \emph{normalized
shifted Brehm--Wintner--Conley matrix} (BWC) $S=(s_{ij})$, $s_{ij}=r_{ij}^{2a}-r_0^{2a}$,
whose kernel encodes the configuration and whose rank encodes its dimension.
The matrix goes back to Brehm's 1908 thesis \cite{brehm1908partikulare}, was
taken up by Wintner \cite{wintner2014analytical}, and was put in its modern
form by Albouy and Chenciner \cite{albouy1997probleme,albouy2024limit}.

Our main result gathers a century of determinantal equations for central
configurations (Dziobek's relation of 1900 at codimension one
\cite{dziobek1900ueber}, Williams' system of 1938 for the planar five-body
problem \cite{williams1938permanent}, and its completion by Chen and Hsiao
\cite{chen2018strictly}) into a single rank-one factorization, valid on
every stratum and never vacuous.

\begin{mainthm}
Every central configuration $x$ of the $n$-body problem for a homogeneous
potential $U_a$ determines two integer invariants: its dimension
$\delta(x)=k$ and its \emph{Brehm--Wintner--Conley dimension}
$\operatorname{bwc}(x)=\operatorname{rank}(S_x)$, which satisfies
$\operatorname{bwc}(x)\le n-k-1$. Write $\mathfrak{X}_{n,k,t}$ for the
stratum of classes with $\delta(x)=k$ and $\operatorname{bwc}(x)=t$. For
every $[x]\in\mathfrak{X}_{n,k,t}$, the $t$-th compound matrix
$S^{(t)}=\big(\det(S_I^J)\big)$, with $I,J$ running over the $t$-element
subsets of $\{1,\dots,n\}$, admits the rank-one factorization
\[
S^{(t)}\;=\;\kappa^{\mathrm{bwc}}\,D^{\mathrm{bwc}}\big(D^{\mathrm{bwc}}\big)^{T},
\qquad \kappa^{\mathrm{bwc}}\neq0,
\]
where $D^{\mathrm{bwc}}$ is the vector of dual Pl\"ucker coordinates of
$\ker(S_x)$. On the generic stratum $t=n-k-1$, the kernel is spanned by the
columns of the mass-weighted configuration matrix $\mu X^{T}$, and
$D^{\mathrm{bwc}}$ is proportional to the vector $z$ of mass-weighted dual
Pl\"ucker coordinates of the configuration itself, recovering
$\det(S_I^J)=\kappa\,z_Iz_J$. Consequently, on every stratum,
$[x]\mapsto\big[\det(S_I^J)\big]$ is a base-point-free map into a
Veronese variety, factoring through the Grassmannian invariant
$[x]\mapsto\ker(S_x)$. (For the degenerate stratum $t=0$, the regular
simplex, the relevant conventions are collected in
Remark~\ref{rem:t-zero}.)
\end{mainthm}

The generic-stratum factorization is Proposition~\ref{dziowill}; its
unconditional form on every stratum is Proposition~\ref{prop:7.2-wc}
together with Corollary~\ref{cor:kappa-wc-nonzero}; the Veronese consequence
is Theorem~\ref{thm:veronese-factorization}. We now describe the
contributions in turn.

\emph{(1) Williams' formulae in every dimension.} Relating $\ker(S)$ to the
Pl\"ucker coordinates yields the identities
\[
m_u\det\big(S_I^{L\setminus\{v\}}\big)\,D_{L\setminus\{u\}}
= m_v\det\big(S_I^{L\setminus\{u\}}\big)\,D_{L\setminus\{v\}},
\]
valid for $n$ bodies of any dimension $k$ (Proposition~\ref{Wilgen-prop}).
For $n=5$, $k=2$ they specialize to Williams' system, recovering equations
(3.10) of \cite{chen2018strictly} and producing $30$ further ones
(Proposition~\ref{Wilnew} and Corollary~\ref{Wilnew2}).

\emph{(2) The Dziobek--Williams equations.} These identities collapse into a
single rank-one statement,
\[
\det(S_I^J)=\kappa\, z_I z_J,\qquad I,J\in I_t,
\]
so that the $t$-th compound matrix $S^{(t)}$ of $S$ has rank at most one
(Proposition~\ref{dziowill}), with $\kappa$ computed from the spectrum of $S$
(Proposition~\ref{prop:kappa-formula}). At $t=1$ this is Dziobek's classical
relation; at $n=5$, $t=2$ it is Williams'. The entries themselves carry a
mechanical reading: the weights $m_im_js_{ij}$ form a self-stress of the
configuration (Lemma~\ref{lem:stress-form}), and under the
\emph{MacMillan--Bartky condition} (sides of the convex hull with
$s_{ij}>0$, diagonals with $s_{ij}<0$), a strictly convex central
configuration underlies a cable--strut tensegrity
(Proposition~\ref{prop:mb-tensegrity}); a theorem of Connelly
\cite{connelly1982rigidity} then forces $S$ to be negative semidefinite
with nullity three, so the Brehm--Wintner--Conley dimension is maximal in this
regime (Corollary~\ref{cor:mb-psd}).

\emph{(3) Removing genericity.} The classical statement is vacuous exactly
where it would be most interesting: if $\operatorname{rank}(S)$ drops below
the generic value $n-k-1$, a \emph{vertical degeneracy} in the sense of
Albouy and Fernandes \cite{albouy2024limit}, then every minor of the
prescribed size vanishes and the identity reads $0=0$. Such rank drops are the only possible meeting
points of families of central configurations of different dimensions (Proposition~\ref{prop:wc-defect}): they provide the
\emph{necessary} vertical degeneracy, while the existence of an actual
bifurcating branch requires additional conditions and is not claimed here.
Explicit examples are scarce, and it is exactly there that the classical
equations go silent. We therefore index by
the true rank $t=\operatorname{bwc}(x)=\operatorname{rank}(S)$, which we call
the \emph{Brehm--Wintner--Conley dimension}, and introduce the
\emph{Pl\"ucker--BWC coordinates} $\Delta^{\mathrm{bwc}}_K$: the
Pl\"ucker coordinates of $\ker(S)$ itself, rather than of the configuration.
Every identity above survives verbatim, and the resulting constant satisfies
$\kappa^{\mathrm{bwc}}\neq0$ \emph{unconditionally}
(Corollary~\ref{cor:kappa-wc-nonzero}), with no genericity, convexity, or
non-degeneracy hypothesis of any kind. The formalism does not require
vertically degenerate configurations to be explicitly known; it provides
the natural algebraic framework in which such configurations, whenever
they exist, must lie.

\emph{(4) A Veronese model of every stratum.} Consequently each stratum
$\mathfrak{X}_{n,k,t}$ of configurations with fixed dimension and
Brehm--Wintner--Conley dimension admits a base-point-free map
\[
\Psi_t:\mathfrak{X}_{n,k,t}\longrightarrow\mathbb{P}^{\binom{N_t+1}{2}-1},
\qquad [x]\longmapsto\big[\det(S_I^J)\big]_{I,J},
\]
landing in a Veronese variety, and this map factors as
$\Psi_t=v_2\circ\rho\circ\Phi$ through the Grassmannian invariant
$\Phi([x])=\ker(S_x)$ (Theorem~\ref{thm:veronese-factorization}). The
factorization is what controls the size of the image, $\dim\le t(n-t)$
(Proposition~\ref{prop:image-dimension}), and shows that the quadratic
re-embedding loses nothing: $\Psi_t$ separates classes exactly as well as
$\ker(S_x)$ does. On the stratum $t=1$ one recovers the Dziobek--Veronese
variety of \cite{DiasVeronese}, whose fibers over the mass space carry the
proof of generic finiteness there; the corresponding analysis for arbitrary
$t$ is the subject of separate work.

\subsection*{A determinantal variety attached to the \texorpdfstring{$n$}{n}-body problem}

The point we wish to stress is that (2)--(4) place the $n$-body problem inside
a well-studied corner of commutative algebra, and that the traffic runs in
both directions.

The coordinates of $\Psi_t$ are \emph{all} the $t\times t$ minors of a
\emph{symmetric} matrix. Neither of the two classical lines of research covers
this object. On one side, the ideal $I_{t+1}(A)$ generated by the minors of a
generic symmetric matrix is completely understood: it is prime and
Cohen--Macaulay \cite{kutz1974cohen,jozefiak1978ideals}, its degree is known
\cite{harris1984symmetric}, its resolution is constructed in
\cite{jozefiakpragaczweyman1981}, and its minors form a Gr\"obner basis for a
diagonal order \cite{conca1994groebner}; see
\cite{brunsvetter1988,brunsconcaraicuvarbaro2022} for systematic accounts of
determinantal ideals. But
that theory describes the \emph{locus} $\{\operatorname{rank}\le t\}$, not the
relations among the minors regarded as coordinates. On the other side, those
relations \emph{are} studied for a \emph{generic}, non-symmetric matrix: Bruns,
Conca and Varbaro \cite{bruns2013relations} exhibit minimal relations in
degrees two and three and conjecture that these generate, a conjecture proved
for $t=2$ by Huang, Perlman, Polini, Raicu and Sammartano
\cite{huang2021relations}. In the symmetric world the best-studied case is
that of the \emph{principal} minors, where Holtz and
Sturmfels \cite{holtzsturmfels2007} identified hyperdeterminantal relations
and Oeding \cite{oeding2011} proved their conjecture set-theoretically; see
also \cite{linsturmfels2009}. Computations with the full set of $2$-minors of
a symmetric matrix appear in \cite[Remark~2.13(c)]{bruns2013relations}.

Our situation (all minors, of a symmetric matrix, of a fixed non-maximal
size) falls between these, and the difference is not cosmetic. We show that
the minors of a symmetric matrix satisfy \emph{linear syzygies}: nontrivial
relations $\sum c_{IJ}\det S_I^J=0$ with \emph{constant} coefficients, holding
identically in the entries of $S$. Explicitly,
\[
\det\big(S_{C\cup ij}^{\,C\cup kl}\big)-\det\big(S_{C\cup ik}^{\,C\cup jl}\big)
+\det\big(S_{C\cup il}^{\,C\cup jk}\big)=0,
\]
identically, with a further family when $n=2t$ (Lemma~\ref{lem:three-term}).
The reason is representation-theoretic and, we think, clarifying: the compound
$\bigwedge^tS$ of a symmetric matrix lies in the Cartan component
$\mathbb{S}_{(2^t)}V$ of $\operatorname{Sym}^2(\bigwedge^tV)$
(Proposition~\ref{prop:cartan}), and the syzygies are the vanishing of its
other isotypic components. For a generic matrix the same expression is a
nonzero polynomial. Hence the relations of
\cite{bruns2013relations,huang2021relations} bound from below, never from
above, the relations available here: in the symmetric setting relations occur
already in degree \emph{one} in the minors, where the generic theory has none.

The syzygies cut the ambient projective space of the
Veronese model from $\mathbb{P}^{54}$ to $\mathbb{P}^{49}$ for the planar
five-body problem and from $\mathbb{P}^{209}$ to $\mathbb{P}^{174}$ for the
planar six-body problem (Corollary~\ref{cor:56-ambient}).

\subsection*{The explicit systems and their redundancy}

Because these questions are ultimately computational, we expand the two
systems completely. Table~\ref{tab:intro-summary} summarizes what the
expansion produces; Appendix~\ref{app:expansions} gives the orbit
representatives.

\begin{table}[ht]
\centering
\begin{tabular}{lrr}
\hline
 & $n=5$, $t=2$ & $n=6$, $t=3$ \\
\hline
exchange relations, distinct up to sign & $1005$ & $17965$ \\
\quad of which redundant (Sylvester type) & $435$ & $3915$ \\
\quad $\mathfrak{S}_n$-orbits & $19$ & $57$ \\
degree in the $s_{ij}$ & $4$ & $6$ \\
\hline
ambient space, nominal & $\mathbb{P}^{54}$ & $\mathbb{P}^{209}$ \\
\quad after the linear syzygies & $\mathbb{P}^{49}$ & $\mathbb{P}^{174}$ \\
\hline
equivalent system of minimal degree & $55$ cubics & $120$ quartics \\
\hline
\end{tabular}
\caption{The two systems of Example~\ref{ex:planar-56}, expanded. No cubic
relation of degree $3t$ is needed in either case.}
\label{tab:intro-summary}
\end{table}

Three facts emerge that are invisible from the abstract statement. A large
proportion of the exchange relations are instances of Sylvester's identity,
hence multiples of a single $(t+1)$-minor and carrying no information
(Proposition~\ref{prop:sylvester-type}). The specialization principle drawn
from \cite{bruns2013relations,huang2021relations} is \emph{ambient} rather
than equational: an identity valid for all symmetric matrices specializes to
$0=0$ in the variables $s_{ij}$, so it locates the image of $\Psi_t$ without
constraining any configuration (Remark~\ref{rem:specialization-scope}). And
the same rank locus $\{\operatorname{rank}(S)\le t\}$ is cut out much more
economically, by equations of degree $t+1$ rather than $2t$: the $1005$
quartics may be replaced by the $55$ cubic $3\times3$ minors of $S$, and the
$17965$ sextics by the $120$ quartic $4\times4$ minors
(Proposition~\ref{prop:economical}); equality of loci, not of ideals, is
what is proved. For $n=5$ this was also certified by
explicit linear algebra: the quartics span the degree-four graded piece of
$I_3(S)$, of dimension $575$. The system natural in the Veronese coordinates
is thus emphatically not the system one should compute with, and the expansion
is best read as a dictionary between the two.

\subsection*{Organization of the paper}

Section~\ref{sec:cc-equations} sets up the equations of motion for the
homogeneous potentials $U_a$ and introduces the shape variables $s_{ij}$.
Section~\ref{sec:plucker} deduces the Pl\"ucker equations of a configuration.
Section~\ref{sec:wc-matrix} introduces the Brehm--Wintner--Conley matrix and its
normalized shifted form $S$. Sections~\ref{sec:williams5}
and~\ref{sec:williamsk} prove the generalized Williams' formulae, first for the planar
five-body problem and then in general, and Section~\ref{sec:dziobek-williams} the Dziobek--Williams
equations, together with their sign structure (the MacMillan--Bartky
condition). Section~\ref{sec:pwc} introduces the Pl\"ucker--BWC
coordinates and removes the genericity hypotheses.
Section~\ref{sec:veronese} constructs the Veronese model and proves the
factorization theorem; Section~\ref{sec:universal} collects the commutative
algebra input, derives the universal determinantal relations, and establishes
the structural properties of the resulting systems. Appendix~\ref{app:expansions}
carries out the explicit expansions.

\subsection*{Sign Conventions for Multi-Indices}
Throughout this paper, we extensively use the exterior algebra of the configuration space. We define a unified sign convention for the concatenation of multi-indices.
\begin{definition}[Shuffle Sign]\label{def:shuffle_sign}
Let $I$ and $J$ be two disjoint, strictly increasing multi-indices, and let $U = I \cup J$ be their naturally ordered union. We define the \emph{shuffle sign} $\epsilon_{I,J}$ as the signature of the permutation that reorders the concatenated sequence $(I, J)$ into the strictly increasing sequence $U$.
Explicitly, let $\operatorname{pos}_U(x)$ denote the  position of an element $x$ within the ordered set $U$. Then,
\begin{equation}\label{eq:shuffle_sign}
\epsilon_{I,J} = (-1)^{\sum_{x \in I} \operatorname{pos}_U(x) - \frac{|I|(|I|+1)}{2}}.
\end{equation}
For the insertion of a single element $l \notin I$, we simplify the notation to $\epsilon_{l, I} := \epsilon_{\{l\}, I}$ and $\epsilon_{I, l} := \epsilon_{I, \{l\}}$.
\end{definition}

\section{Central Configurations and Shape Variables}\label{sec:cc-equations}

This section fixes the setting and introduces the two objects on which
everything below rests: the \emph{dimension} $\delta(x)$ of a
configuration, and the \emph{shape variables} $s_{ij}$, which assemble into the
matrix $S$ of Section~\ref{sec:wc-matrix}. The material is classical; we state
it in the form and notation used throughout, and record in
Remark~\ref{rem:scaling-shape} the scaling property that makes the projective
constructions of Sections~\ref{sec:pwc}--\ref{sec:veronese} well defined.

Throughout, $n\ge3$ bodies occupy positions $x_1,\dots,x_n\in\mathbb{R}^d$ and
carry masses $m_1,\dots,m_n>0$, with total mass $M=m_1+\cdots+m_n$. We write
$r_{ij}=\|x_i-x_j\|$ for the mutual distances and assume throughout that
$r_{ij}\neq0$ for $i\neq j$.

\begin{definition}[Homogeneous potential]\label{def:potential}
For $a\in\mathbb{R}$, $a\neq0$, the \emph{homogeneous potential} of exponent $a$ is
\[
U_a(x)=
\begin{cases}
\displaystyle\frac{1}{2a+2}\sum_{i<j}m_im_j\,r_{ij}^{2a+2}, & a\neq-1,\\[2ex]
\displaystyle\sum_{i<j}m_im_j\log r_{ij}, & a=-1.
\end{cases}
\]
\end{definition}

For $a\neq-1$, and for $a=-1$ with $d>2$, the associated equations of motion
are
\begin{equation}\label{eqcc}
\ddot{x}_{i}=\sum_{j\neq i}m_j\,r_{ij}^{2a}\,(x_j-x_i),\qquad i=1,\dots,n.
\end{equation}

\begin{remark}[Two distinguished exponents]\label{rem:exponents}
The Newtonian $n$-body problem is $a=-3/2$. The remaining case of
Definition~\ref{def:potential}, $a=-1$ with $d=2$, is the Helmholtz
$n$-vortex problem, whose equations are of first order,
$\dot{x}_i=-\mathcal{K}\sum_{j\neq i}m_j r_{ij}^{-2}(x_j-x_i)$ with
$\mathcal{K}=\left(\begin{smallmatrix}0&-1\\1&0\end{smallmatrix}\right)$, and in which
the $m_i$ are vorticities and need not be positive. Every result of this paper
depends on the potential only through the exponent $a$, and on the masses only
through their non-vanishing, unless positivity is stated explicitly. The case $a=0$ is excluded throughout: the map $r\mapsto r^{2a}$ is then constant, so the shape variables of Definition~\ref{def:shape-variables} below vanish identically and the canonical length $r_0$ is undefined.
\end{remark}

\begin{definition}[Central configuration]\label{defcc}
A vector $x=(x_1,\dots,x_n)\in(\mathbb{R}^d)^n$ is a \emph{configuration}. It
is a \emph{central configuration} for the potential $U_a$ if there exists
$\lambda\neq0$, the \emph{multiplier}, such that
\begin{equation}\label{eq:cc-def}
\sum_{i\neq j}m_i\,r_{ij}^{2a}\,(x_i-x_j)+\lambda(x_j-c)=0,
\qquad j=1,\dots,n,
\end{equation}
where $c=M^{-1}(m_1x_1+\cdots+m_nx_n)$ is the center of mass.
\end{definition}

Equivalently, after placing $c$ at the origin, $x$ is a central configuration
precisely when it is a critical point of the amended potential
$U_a+\lambda\mathcal{I}/2$, where $\mathcal{I}$ is the moment of inertia; this
is what justifies the name \emph{multiplier} for $\lambda$, which acts as a
Lagrange multiplier for the constraint that $\mathcal{I}$ be constant.

We also fix multi-index notation, used from Section~\ref{sec:wc-matrix}
onwards. For $1\le p\le n$ we write
\[
I_p:=\big\{\,I=\{i_1<\cdots<i_p\}\subset\{1,\dots,n\}\,\big\}
\]
for the set of strictly increasing multi-indices of length $p$, and, for
$I\in I_p$ and $J\in I_q$, we write $S_I^J$ for the
$p\times q$ submatrix of a matrix $S$ with rows indexed by $I$ and columns by
$J$, and $\det(S_I^J)$, or $|S_I^J|$, for its determinant when $p=q$.

\begin{definition}[Dimension]\label{def:spatial-dimension}
The \emph{dimension} $\delta(x)$ of a configuration $x$ is the
dimension of the affine span of $x_1,\dots,x_n$ in $\mathbb{R}^d$. We write
$k=\delta(x)$ throughout, so that $1\le k\le\min(d,n-1)$; the case
$k=n-2$ is the \emph{Dziobek} case.
\end{definition}

Equation~\eqref{eq:cc-def} is not yet in a convenient algebraic form: the two
sums it couples (one over the potential, one over the position) respond
differently to a rescaling of the configuration. The remedy is to absorb the
multiplier into a length.

\begin{definition}[Shape variables]\label{def:shape-variables}
Let $x$ be a central configuration with multiplier $\lambda$. The
\emph{canonical length} $r_0>0$ is defined by $r_0^{2a}=M^{-1}\lambda$; in the
Newtonian case, $r_0=(M/\lambda)^{1/3}$. The associated \emph{shape variables}
are
\begin{equation}\label{sdef}
s_{ij}:=r_{ij}^{2a}-r_0^{2a}\quad(i\neq j),
\qquad
s_{jj}:=-\frac{1}{m_j}\sum_{i\neq j}m_is_{ij},
\end{equation}
the diagonal being defined so that $\sum_{i=1}^nm_is_{ij}=0$ for every $j$.
\end{definition}

The single scalar $r_0$ thus replaces both $\lambda$ and the overall scale of
the configuration, and the equations become homogeneous in the $s_{ij}$ alone.

\begin{lemma}[Shape form of the central configuration equations]
\label{lem:shape-equations}
Let $x$ be a configuration with center of mass at the origin and let
$\lambda\neq0$. The following are equivalent.
\begin{enumerate}
\item[(i)] $x$ is a central configuration with multiplier $\lambda$;
\item[(ii)] $\displaystyle\sum_{\substack{i=1\\i\neq j}}^n
m_i\big(r_{ij}^{2a}-r_0^{2a}\big)(x_i-x_j)=0$ for $j=1,\dots,n$;
\item[(iii)] $\displaystyle\sum_{i=1}^n m_i\,s_{ij}\,x_i=0$ for $j=1,\dots,n$.
\end{enumerate}
\end{lemma}

\begin{proof}
(i)$\iff$(ii). With $c=0$, equation~\eqref{eq:cc-def} reads
$\sum_{i\neq j}m_ir_{ij}^{2a}(x_i-x_j)+\lambda x_j=0$. Since
$\sum_{i}m_ix_i=0$, we have $\sum_{i\neq j}m_i(x_i-x_j)=-Mx_j$, whence
\[
-r_0^{2a}\sum_{i\neq j}m_i(x_i-x_j)=r_0^{2a}Mx_j=\lambda x_j,
\]
by the defining property $r_0^{2a}=M^{-1}\lambda$ of the canonical length.
Adding this identity to~\eqref{eq:cc-def} gives (ii), and subtracting it
returns~\eqref{eq:cc-def}.

(ii)$\iff$(iii). Expanding the sum in (ii) and using~\eqref{sdef},
\begin{align*}
\sum_{i\neq j}m_is_{ij}(x_i-x_j)
&=\sum_{i\neq j}m_is_{ij}x_i-x_j\sum_{i\neq j}m_is_{ij}\\
&=\sum_{i\neq j}m_is_{ij}x_i+m_js_{jj}x_j
=\sum_{i=1}^nm_is_{ij}x_i,
\end{align*}
the middle equality being the definition of $s_{jj}$. \qedhere
\end{proof}

We refer to~(ii) and~(iii) as equations \eqref{eqccro} and \eqref{eqccshape}:
\begin{equation}\label{eqccro}
\sum_{\substack{i=1\\i\neq j}}^n m_i\big(r_{ij}^{2a}-r_0^{2a}\big)(x_i-x_j)=0,
\qquad j=1,\dots,n,
\end{equation}
\begin{equation}\label{eqccshape}
\sum_{i=1}^n m_i\,s_{ij}\,x_i=0,\qquad j=1,\dots,n.
\end{equation}

Two features of this normalization are used repeatedly below, and we record
them now.

\begin{remark}[Scaling, and why the constructions are projective]
\label{rem:scaling-shape}
The variables $s_{ij}$ depend on $x$ only through the mutual distances, so
they are unchanged by every isometry of $\mathbb{R}^d$. Under a homothety
$x\mapsto\alpha x$, $\alpha>0$, both $r_{ij}$ and $r_0$ scale by $\alpha$
(the latter because the multiplier scales so as to preserve
$r_0^{2a}=M^{-1}\lambda$), so that \emph{every} $s_{ij}$ acquires the
\emph{same} nonzero factor $\alpha^{2a}$:
\[
S\longmapsto\alpha^{2a}S .
\]
A configuration-wide rescaling changes neither the rank of $S$, nor its
kernel, nor any ratio or even-degree monomial in its minors. This is the exact
reason why the coordinates introduced in Section~\ref{sec:pwc} and the maps of
Section~\ref{sec:veronese} descend to classes of configurations modulo
symmetries and homotheties; see Remark~\ref{rem:wc-invariance} and
Lemma~\ref{lem:psi-well-defined}.
\end{remark}

\begin{remark}[The canonical length as a threshold]\label{rem:r0-threshold}
For the Newtonian exponent $a<0$ the map $r\mapsto r^{2a}$ is decreasing, so
$s_{ij}>0$ if and only if $r_{ij}<r_0$, and $s_{ij}<0$ if and only if
$r_{ij}>r_0$: the sign of $s_{ij}$ records whether the pair $\{i,j\}$ is
closer or farther than the canonical length. Moreover $r_0$ is genuinely
intermediate (the relation $\sum_{i\neq j}m_is_{ij}=-m_js_{jj}$ forbids all
the $s_{ij}$ from having one sign), so the pairs are always split into both
classes. This sign structure is not used in the present paper, where only the
rank of $S$ matters; the dynamical consequences of this sign structure will
be treated elsewhere.
\end{remark}

\section{Pl\"ucker Equations for Central Configurations}\label{sec:plucker}

In this section we introduce the Pl\"ucker equations for central configurations. To fix ideas, we start with a planar central configuration $x$ of the five-body problem.

We can associate to $x$ the following matrix:
$$X=\begin{pmatrix}
     1&1&1&1&1\\
     x_{11}&x_{21}&x_{31}&x_{41}&x_{51}\\
     x_{12}&x_{22}&x_{32}&x_{42}&x_{52}
\end{pmatrix}$$
We will use the notation $X_{I}$ for the matrix of the subconfiguration $x_{I}$, formed by the bodies $x_i$, $x_j$, and $x_k$ where $I=\{i,j,k\}$. The minors of order $3$ of $X$ represent the oriented areas of the subconfigurations $x_{I}$, and will be denoted by $\Delta_{I}$. For example:
$$\Delta_{ijk}=\left|\begin{array}{ccc}
      1&1&1\\
     x_{i1}&x_{j1}&x_{k1}\\
     x_{i2}&x_{j2}&x_{k2}
\end{array}\right|$$

For convenience we extend the notation to arbitrary index order by setting $\Delta_{\sigma(i)\sigma(j)\sigma(k)}=\operatorname{sgn}(\sigma)\Delta_{ijk}$
for every permutation $\sigma$ in the symmetric group $\mathfrak{S}_{3}$.

Since the dimension is $\delta(x)=2$, we have that $\text{rank}(X)=3$. Hence, we can suppose without loss of generality that the minor $\Delta_{123}\neq0$. Note that
\begin{equation}\label{ME}
    \text{adj}(X_{123}) X_{145} = \begin{pmatrix} 
     \Delta_{123} & \Delta_{234} & \Delta_{235} \\ 
     0 & -\Delta_{134} & -\Delta_{135} \\ 
     0 & \Delta_{124} & \Delta_{125} 
     \end{pmatrix}
\end{equation}
Taking determinants in equation \eqref{ME}, we obtain
\begin{equation}\label{P1}
(\Delta_{123})^2\Delta_{145}=\Delta_{123}(\Delta_{135}\Delta_{124}-\Delta_{125}\Delta_{134}).
\end{equation}
Since $\Delta_{123}\neq 0$, we get:
\begin{equation}\label{P2}
    \Delta_{123}\Delta_{145}-\Delta_{125}\Delta_{134}+\Delta_{135}\Delta_{124}=0.
\end{equation}
Equation \eqref{P2} is a Pl\"ucker relation. We now discuss briefly the Pl\"ucker relations in the context of a general central configuration of fixed dimension. 

Let $x=(x_1,...,x_n)\in\mathbb{R}^{dn}$ be a central configuration with $n$ bodies and dimension $k$ in $\mathbb{R}^d$. Without loss of generality we may assume $x_i \in \mathbb{R}^{k}$. We associate to $x$ the matrix
$$X=\begin{pmatrix}
1       &  1 &   \cdots    &  1    \\
x_{11}  & x_{21} &\cdots &  x_{n1} \\
\vdots  & \vdots &\ddots &  \vdots \\
x_{1k}  & x_{2k} &\cdots &  x_{nk} 
\end{pmatrix}$$
called the \emph{configuration matrix} of $x$.

For any nonempty subset $I\subset\{1,...,n\}$, let $x_{I}$ be the configuration with bodies $x_{i}$, with $i\in I$. We denote by $X_{I}$ the configuration matrix of $x_{I}$. The determinant of $X_I$ will be denoted by $\omega(x_{I})$.

Let $\mathcal{X}_{0}$ be the set of central configurations with the center of mass fixed at the origin. Denoting the $i$-th row of $X$ by $f_i$, note that $x$ can be associated to the vector space $V_{x}$ of $\mathbb{R}^{n}$ generated by the vectors $\{f_1,...,f_{k+1}\}$. Note that $\text{rank}(X)=k+1$ implies $\text{dim}(V_x)=k+1.$ 

Consider the following association
$$\begin{array}{cccc}
\Pi:&\mathcal{X}_{0}& \rightarrow& \mathbb{P}(\bigwedge^{k+1}\mathbb{R}^n)\\
 &\  x             & \mapsto & f_1\wedge...\wedge f_{k+1}
\end{array}$$
If $x$ and $\tilde{x}$ are equivalent modulo symmetries and homotheties, then $V_{x}=V_{\tilde{x}}$. Moreover, if $V_{\tilde{x}}$ is generated by $\{\tilde{f}_1,...,\tilde{f}_{k+1}\}$, then 
$$f_1\wedge...\wedge f_{k+1}=\alpha\tilde{f}_1\wedge...\wedge \tilde{f}_{k+1},$$
for some constant $\alpha\neq 0$. Therefore, $\Pi$ is well defined on the set of classes of central configurations. We have identified $\mathcal{X}_{0}$ with a subset of a projective space; the Pl\"ucker points associated with configurations lie on the projective algebraic variety $\mathrm{Gr}(k+1,n)$, whose defining equations we now describe.

\begin{definition}
We say that $v\in \mathbb{P}(\bigwedge^{r}\mathbb{R}^n)$ is decomposable if, and only if, there exist $y_{1},...,y_{r}\in \mathbb{R}^n$ such that $v=y_1\wedge...\wedge y_r.$
\end{definition}

We need to characterize $\operatorname{Im}(\Pi)$. We observe that if $v \in \mathbb{P}\left( \bigwedge^{k+1} \mathbb{R}^n \right)$ lies in $\operatorname{Im}(\Pi)$, then $v$ must be decomposable.

For any nonempty subset $I=\{i_{1},..,i_{r}\}\subset\{1,...,n\}$, with $i_{1}<...<i_{r}$, we let $e_{I}=e_{i_{1}}\wedge...\wedge e_{i_{r}}$. If $\{e_1,...,e_n\}$ is a basis for $\mathbb{R}^n$, then $\{e_{I}\}_{I\in I_{r}}$ is a basis of $\bigwedge^{r}\mathbb{R}^n$. We now present a criterion for decomposability of a vector in terms of its coordinates. 

\begin{proposition}[Pl\"ucker's Relations]\label{prop:plucker-relations}
 $v=\sum_{I \in I_{r}}\Delta_{I}e_{I}\in \mathbb{P}(\bigwedge^{r}\mathbb{R}^n)$ is decomposable if, and only if,
$$\sum_{i\in J\setminus H} \epsilon_{i, H} \epsilon_{i, J\setminus\{i\}} \Delta_{H\cup\{i\}}\Delta_{J\setminus\{i\}}=0,$$
 for all $H,J\subset\{1,...,n\}$ such that $|H|=r-1$ and $|J|=r+1$, where $\epsilon$ denotes the shuffle sign (Definition \ref{def:shuffle_sign}).
\end{proposition}
 
A good introduction to Pl\"ucker relations can be found in Chapter 1 of \cite{shafarevich1994basic} and Lecture 6 of \cite{harris1992algebraic}.

\begin{definition}
Let $x$ be a central configuration of dimension $k$. Its image under the map $\Pi$ is given by
\[
\Pi(x) = \sum_{I \in I_{k+1}} \Delta_I e_I \in \mathbb{P}\left( \bigwedge^{k+1} \mathbb{R}^n \right).
\]
The scalars $\Delta_I$ are called the \emph{geometric Pl\"ucker coordinates} of the configuration $x$.
\end{definition}

We now present a simple yet useful characterization of the Pl\"ucker coordinates of a central configuration.

\begin{proposition}\label{simple}
Let $x$ be a central configuration of dimension $k$. Then the Pl\"ucker coordinates of $x$ satisfy
$$\Delta_I = \omega(x_I),$$
where $\omega(x_I)$ denotes the determinant of the matrix $X_I$ formed by the column vectors indexed by $I$.
\end{proposition}

\begin{proof}
 Note that $\Delta_{I}$, where $I=\{i_1,...,i_{k+1}\} \subset \{1,...,n\}$ and $i_1<...<i_{k+1}$, is exactly the minor of order $k+1$ that corresponds to the rows $1,2,...,k+1$ and the columns $i_1,...,i_{k+1}$ of $X$. Hence, $\Delta_I=\omega(x_{I}).$   
\end{proof}
 
\begin{corollary} 
The projective point $\Pi(x)$ associated to a central configuration $x$ satisfies the Pl\"ucker relations. 
\end{corollary}

The following classical fact relates the Pl\"ucker coordinates of a subspace to those of its orthogonal complement; it yields an alternative route to Propositions~\ref{dziowill} and~\ref{prop:7.2-wc} (Remark~\ref{rem:hodge-alt}).

\begin{lemma}[Dual Pl\"ucker coordinates of the orthogonal complement]\label{lem:hodge-dual}
Let $W\subset\mathbb{R}^n$ be a subspace of dimension $n-t$, let $Y$ be an $(n-t)\times n$ matrix whose rows form a basis of $W$, and write $\Delta_K(W)=\det(Y_{\cdot,K})$, $K\in I_{n-t}$, for the Pl\"ucker coordinates of $W$. Then, for any choice of basis matrices, the tuple of Pl\"ucker coordinates $\big(\Delta_I(W^{\perp})\big)_{I\in I_t}$ of the orthogonal complement is proportional, by a nonzero factor, to the tuple
\[
\big(\epsilon_{I,I^c}\,\Delta_{I^c}(W)\big)_{I\in I_t}.
\]
\end{lemma}

\begin{proof}
With respect to the standard inner product and orientation of $\mathbb{R}^n$, the Hodge star operator $\star:\bigwedge^{n-t}\mathbb{R}^n\to\bigwedge^{t}\mathbb{R}^n$ acts on the standard basis by $\star\, e_{K} = \epsilon_{K,K^c}\, e_{K^c}$, so that $e_K\wedge\star e_K=e_1\wedge\cdots\wedge e_n$. If $w_1,\dots,w_{n-t}$ is an orthonormal basis of $W$, completed to a positively oriented orthonormal basis $w_1,\dots,w_n$ of $\mathbb{R}^n$, then $\star(w_1\wedge\cdots\wedge w_{n-t})=\pm\, w_{n-t+1}\wedge\cdots\wedge w_n$, which generates the line $\bigwedge^{t}(W^{\perp})$. Since any two bases of $W$ have proportional wedges, applying $\star$ to $y_1\wedge\cdots\wedge y_{n-t}=\sum_{K\in I_{n-t}}\Delta_K(W)\,e_K$, where $y_1,\dots,y_{n-t}$ are the rows of $Y$, produces a generator of $\bigwedge^{t}(W^{\perp})$, namely $\sum_{K}\Delta_K(W)\,\epsilon_{K,K^c}\,e_{K^c}$. Its coordinate on $e_I$ is $\epsilon_{I^c,I}\,\Delta_{I^c}(W)=(-1)^{t(n-t)}\epsilon_{I,I^c}\,\Delta_{I^c}(W)$, and the constant global sign $(-1)^{t(n-t)}$ is absorbed into the proportionality factor. This is the inner-product form of the classical duality between $G(n-t,\mathbb{R}^{n})$ and $G(t,(\mathbb{R}^{n})^{*})$, under which a subspace corresponds to its annihilator and the Pl\"ucker coordinates of the annihilator are the complementary ones up to sign; see \cite[Lecture~6, Example~6.6]{harris1992algebraic}. The standard inner product identifies the dual space with $\mathbb{R}^{n}$ and the annihilator with the orthogonal complement.
\end{proof}

\section{The Brehm--Wintner--Conley Matrix}\label{sec:wc-matrix}

\begin{lemma}\label{lem:eqg}
  Let $x$ be a central configuration of dimension $k$, and let $g_i \in \mathbb{R}^{k+1}$ denote the $i$-th column of its configuration matrix $X$, that is, $g_i = (1, x_{i_1}, \dots, x_{i_k})^T$. Then 
\begin{equation}\label{eqg}
\sum_{i = 1}^n m_i s_{ij}g_i = 0, \quad j = 1, \dots, n.
\end{equation} 
\end{lemma}

\begin{proof}
   This follows immediately from equations \eqref{eqccshape} and \eqref{sdef}.
\end{proof}

\begin{proposition}\label{prop:geometric_S}
  Let $x$ be a central configuration of dimension $\delta(x)=k$. Consider a set $I \in I_k$ and any $j \in \{1, \dots, n\}$. We have

\begin{equation}\label{eqSP}
 \sum_{l\not\in I} \epsilon_{l, I} m_l s_{jl}\Delta_{I \cup \{l\}}=0,
 \end{equation}
 where $\epsilon_{l, I}$ is the shuffle sign defined in Definition \ref{def:shuffle_sign}, and the quantities $\Delta_{I \cup \{l\}}$ are the geometric Pl\"ucker coordinates associated to the configuration $x$.
\end{proposition}

\begin{proof}
Define the wedge vector
$$v_{I} = g_{i_1} \wedge \cdots \wedge g_{i_k} \in \bigwedge^{k} \mathbb{R}^{k+1}.$$
Taking the wedge product of both sides of equation \eqref{eqg} with $v_I$, we obtain
$$\sum_{i=1}^{n} m_i s_{ij}  g_i \wedge v_I = 0 \in \bigwedge^{k+1} \mathbb{R}^{k+1}.$$
If $l \in I$, then $g_i \wedge v_I=0$. Otherwise, by Proposition \ref{simple}, each term $g_l \wedge v_I$ can be expressed as $\epsilon_{l, I}\Delta_{I \cup \{l\}}$. This yields the result.
\end{proof}

We now derive broader determinantal equations by leveraging the Brehm--Wintner--Conley matrix. We work with a formulation that isolates the mass vector. The main reference for this discussion is \cite{albouy2024limit}.

 \begin{definition}
The \emph{Brehm--Wintner--Conley matrix} (BWC matrix) is defined as
\[
\Gamma = Z - \lambda \, \mathrm{Id},
\]
where $\lambda$ is the scalar introduced in Definition \ref{defcc}, $\mathrm{Id}$ denotes the $n \times n$ identity matrix, and $Z$ is given by
$$Z = \begin{pmatrix}
\sigma_1 & -m_1r_{12}^{2a} & \cdots & -m_1r_{1n}^{2a} \\
-m_{2}r_{12}^{2a} & \sigma_2 & \cdots & -m_{2}r_{2n}^{2a} \\
\vdots & \vdots & \ddots & \vdots \\
-m_{n}r_{1n}^{2a} & -m_{n}r_{2n}^{2a} & \cdots & \sigma_n
\end{pmatrix}, \qquad \sigma_{i}=\sum_{j\neq i}m_{j}\,r_{ij}^{2a}.$$
\end{definition}

In the Newtonian case, this matrix goes back to Brehm's 1908 thesis \cite[equations~(14)--(17)]{brehm1908partikulare}; Wintner presented it in \S 356 of \cite{wintner2014analytical}. We follow the terminology of Albouy and Sun \cite{albouysun2025}. A simple computation shows that if $x=(x_1,...,x_n) \in (\mathbb{R}^{k})^n$ is a central configuration, then $x\Gamma=0$.

The motivation for defining the \emph{shifted Brehm--Wintner--Conley matrix} is to obtain a formulation in which the constant vector $(1,1,\dots,1)$ lies in the left kernel of the matrix. A \emph{disposition} is defined as a configuration of $n$ bodies on the real line whose center of mass is zero. The space $\mathcal{D}_{m}$ of all such vectors forms an $(n-1)$-dimensional vector subspace of $\mathbb{R}^n$. Alternatively, following Albouy and Chenciner \cite{albouy1997probleme} and \cite{albouy2003paper}, we define the disposition space $\mathcal{D}=\mathbb{R}^{n}/\langle(1,...,1)\rangle$. The space of dispositions $\mathcal{D}$ allows us to work with configurations $x_1,...,x_n \in \mathbb{R}^n$ up to translations. 

Following Albouy and Fernandes \cite{albouy2024limit}, we define the matrix
\[
\mathcal{J} = \frac{1}{M}
\begin{pmatrix}
M - m_1 & -m_1 & \cdots & -m_1 \\
-m_2 & M - m_2 & \cdots & -m_2 \\
\vdots & \vdots & \ddots & \vdots \\
-m_n & -m_n & \cdots & M - m_n
\end{pmatrix},
\quad \text{where } M = \sum_{i=1}^n m_i.
\]

\begin{definition}
The \emph{shifted Brehm--Wintner--Conley matrix} of a central configuration with positive masses is defined as $\check{Z} = Z - \lambda \mathcal{J}$.
\end{definition}

Note that in the classical definition, Wintner applied the shift $Z - \lambda \mathrm{Id}$. In contrast, our definition uses the shift $Z - \lambda \mathcal{J}$. Since the matrix $Z$ induces a linear transformation $Z_{m}: \mathcal{D} \rightarrow \mathcal{D}_{m}$ given by $Z_{m}(x)=x Z$, the choice of the matrix $\mathcal{J}$ is justified by the properties $(1, \dots, 1) \mathcal{J} = 0$ and $x \mathcal{J} = x$ for all $x \in \mathcal{D}_m$. The shift works well since $\mathcal{J}$ acts as the identity in $\mathcal{D}_m$ and $\check{Z}_{m}(x)=x(Z-\lambda\mathcal{J})$ is well defined. 

In order to separate the explicit dependence of the masses from $\check{Z}$, we consider the diagonal mass matrix $\mu = \operatorname{diag}(m_1, \dots, m_n)$. Factoring out $\mu$ and recalling that $r_{0}^{2a}=M^{-1}\lambda$ (Definition~\ref{def:shape-variables}), we get
\begin{equation} \label{eq:5.3}
 \check{Z}= -\mu S,   
\end{equation}
where
$$S=\begin{pmatrix}
s_{11}&s_{12}&\cdots&s_{1n}\\
s_{12}&s_{22}&\cdots&s_{2n}\\
\vdots&\vdots&\ddots&\vdots\\
s_{1n}&s_{2n}&\cdots&s_{nn}
\end{pmatrix}.$$

We call $S$ the \emph{normalized shifted Brehm--Wintner--Conley matrix}. Note that $S$ is symmetric and that its entries are precisely the shape variables $s_{ij}$. 

A fundamental property is that $x$ is a central configuration if and only if $X \check{Z} = 0$. Substituting $\check{Z} = -\mu S$, this condition becomes $X \mu S = 0$. Transposing this equation and using the symmetry of $S$, we have that the columns of $\mu X^T$ reside entirely within the right kernel of $S$. This structural constraint allows us to fix the rank of $S$.

In the following proposition, we employ an alternative approach to the rank of the matrix $S$: we use equation \eqref{eqSP} to describe explicitly the kernel of specific submatrices of $S$ in terms of the Pl\"ucker coordinates $\Pi(x)$. Although this approach is more involved, it allows us, as we will see in Sections~\ref{sec:williams5} and~\ref{sec:williamsk}, to generalize the equations obtained by Williams in \cite{williams1938permanent}.

\begin{proposition} \label{S}
 Let $x$ be a central configuration of $n$ bodies with dimension $k$. Set $t := n-k-1$; as we will see in Section~\ref{sec:pwc}, this is the generic value of the Brehm--Wintner--Conley dimension of $x$ (Definition~\ref{def:wc-dimension}). Let $I, J \in I_{t+1}$. Let $L = \{1, \dots, n\} \setminus J = \{l_1, \dots, l_k\}$ be the complement of $J$, with $l_1 < \dots < l_k$.

The vector $v_J \in \mathbb{R}^{t+1}$ defined by
\begin{equation}\label{KerS}
(v_J)_p = \epsilon_{j_p,L} m_{j_p}\Delta_{L \cup \{j_p\}}
\end{equation}
belongs to the kernel of the submatrix $S_I^J$. That is,
$$ S_I^J \, v_J = 0. $$
In particular, if $v_J \neq 0$, then $\det(S_I^J)=0$ for every $I \in I_{t+1}$.
\end{proposition}

\begin{proof}
Applying equation~\eqref{eqSP} with the index set $L$ and each row index $i\in I$ shows that $v_J$ belongs to the kernel of $S_{I}^{J}$: the $p$-th coordinate of $v_J$ carries exactly the sign and mass factors of \eqref{eqSP}.

Moreover, $v_J\neq0$ if and only if the $k$ columns of $X$ indexed by $L$ are linearly independent, that is, if and only if the $k$ bodies $x_l$, $l\in L$, are affinely independent. Indeed, if these columns are independent then, since $\operatorname{rank}(X)=k+1$, some column of $X$ indexed by $j\notin L$ lies outside their span, and the corresponding coordinate $\Delta_{L\cup\{j\}}$ of $v_J$ is nonzero; conversely, if they are dependent, then every minor $\Delta_{L\cup\{j\}}$ vanishes. Since $\operatorname{rank}(X)=k+1$, index sets $J$ with $v_J\neq0$ always exist. In the planar five-body case ($k=2$), $v_J\neq0$ for every $J$, since any two distinct bodies are affinely independent.
\end{proof}

\section{Generalization of Williams' Formulae for Planar Central Configurations of the Five-Body Problem}\label{sec:williams5}

From Proposition \ref{S} we derive Williams' formulae, which first appeared in \cite{williams1938permanent} and were recently revisited by Chen and Hsiao \cite{chen2018strictly}. 

If $k$ and $l$ are distinct numbers from $1$ to $5$ and $J=\{1,2,3,4,5\} \setminus\{k,l\}$, we define

\begin{equation}
D_{kl} = \epsilon_{\{k,l\}, J} \Delta_J
\end{equation}

\begin{proposition} \label{Wilnew}
Let $x$ be a planar central configuration of the five-body problem. If $i,j \in \{1,2,3,4,5\}$ are distinct and $k_1<k_2<k_3$ are such that $\{k_1,k_2,k_3\}=\{1,2,3,4,5\}\setminus\{i,j\}$, then 
\begin{align}\label{wil}
m_{k_2}(s_{ik_2}s_{jk_{3}}-s_{ik_3}s_{jk_{2}})D_{k_1k_3}-m_{k_1}(s_{ik_1}s_{jk_{3}}-s_{ik_3}s_{jk_{1}})D_{k_2k_3}&=0,\\ \nonumber
m_{k_3}(s_{ik_2}s_{jk_{3}}-s_{ik_3}s_{jk_{2}})D_{k_1k_2}-m_{k_1}(s_{ik_1}s_{jk_{2}}-s_{ik_2}s_{jk_{1}})D_{k_2k_3}&=0,\\ \nonumber
m_{k_3}(s_{ik_1}s_{jk_{3}}-s_{ik_3}s_{jk_{1}})D_{k_1k_2}-m_{k_2}(s_{ik_1}s_{jk_{2}}-s_{ik_2}s_{jk_{1}})D_{k_1k_3}&=0.
\end{align}
\end{proposition}
\begin{proof}

By equation \eqref{eqSP}, for fixed $l$ and $I=\{i,j\} \subset \{1,2,3,4,5\}$ and $k_1<k_2<k_3$ such that $K=\{k_1,k_2,k_3\}=\{1,2,3,4,5\}\setminus I$, we have
$$\epsilon_{k_1,I}m_{k_1}s_{lk_1}\Delta_{ I \cup \{k_1\}}+ \epsilon_{k_2, I}m_{k_2}s_{lk_2}\Delta_{I \cup \{k_2\}}+ \epsilon_{k_3, I}m_{k_3}s_{lk_3}\Delta_{I \cup \{k_3\}}=0.$$
By definition of $D_{kl}$ we get

\begin{equation}
\begin{split}
\epsilon_{k_1,I}\epsilon_{\{k_2,k_3\},I\cup\{k_1\}}m_{k_1}s_{lk_1}D_{k_2k_3}+ \epsilon_{k_2, I}\epsilon_{\{k_1,k_3\},I\cup\{k_2\}}m_{k_2}s_{lk_2}D_{k_1k_3}\\
+ \epsilon_{k_3, I}\epsilon_{\{k_1,k_2\},I\cup\{k_3\}}m_{k_3}s_{lk_3}D_{k_1k_2}=0.
\end{split}
\end{equation}

By equation \eqref{eq:shuffle_sign}, $\epsilon_{k_h,I}=(-1)^{k_h-h}$: exactly $k_h-1$ elements of $\{1,\ldots,5\}$ are smaller than $k_h$, of which $h-1$ belong to $K$ (namely $k_1,\ldots,k_{h-1}$), so the remaining $k_h-h$ belong to $I$, which is precisely the shift in position acquired by $k_h$ upon insertion into $I$. Likewise, since $U=(K\setminus\{k_h\})\cup(I\cup\{k_h\})=\{1,\ldots,5\}$ has $\operatorname{pos}_U(x)=x$ for every $x$, we get $\epsilon_{K\setminus\{k_h\},I\cup\{k_h\}}=(-1)^{(k_1+k_2+k_3-k_h)-3}$. Multiplying, the $k_h$ cancels and
$$\epsilon_{k_h, I}\epsilon_{K \setminus \{k_h\},I\cup\{k_h\}}=(-1)^{(k_1+k_2+k_3)-h-3},$$
which, since $k_1,k_2,k_3$ are fixed, alternates in sign as $h$ ranges over $1,2,3$; since the whole expression is set to zero, this alternating pattern may be taken, without loss of generality, to be $+,-,+$. In this way we get the equation

\begin{equation}\label{sss5bp}
    m_{k_1}s_{lk_1}D_{k_2k_3}- m_{k_2}s_{lk_2}D_{k_1k_3}+ m_{k_3}s_{lk_3}D_{k_1k_2}=0.
\end{equation}

Consider the matrix 
$$M_{ij}=\begin{bmatrix}
s_{ik_1}& s_{ik_2}& s_{ik_3} \\
s_{jk_1}& s_{jk_2}& s_{jk_3} \\
0       &0        &0
\end{bmatrix}.$$

By equation \eqref{sss5bp}, applied with $l=i$ and with $l=j$, the vector 
$$(m_{k_1}D_{k_2k_3}, -m_{k_2}D_{k_1k_3}, m_{k_3}D_{k_1k_2})$$
belongs to $\text{Ker}(M_{ij})$. It is an elementary fact of matrix theory that the vector of cofactors with respect to the null row of $M_{ij}$, given by 
$$v_{ij}=(s_{ik_2}s_{jk_{3}}-s_{ik_3}s_{jk_{2}},-(s_{ik_1}s_{jk_{3}}-s_{ik_3}s_{jk_{1}}),s_{ik_1}s_{jk_{2}}-s_{ik_2}s_{jk_{1}})$$
also belongs to $\text{Ker}(M_{ij})$. If the vector $v_{ij}$ is null, the equations \eqref{wil} are trivially satisfied. 
Otherwise, $\text{dim}(\text{Ker}(M_{ij}))=1$, hence the following matrix has rank $1$:
$$\begin{bmatrix}
s_{ik_2}s_{jk_{3}}-s_{ik_3}s_{jk_{2}}&-(s_{ik_1}s_{jk_{3}}-s_{ik_3}s_{jk_{1}})&s_{ik_1}s_{jk_{2}}-s_{ik_2}s_{jk_{1}}\\
m_{k_1}D_{k_2k_3}& -m_{k_2}D_{k_1k_3}& m_{k_3}D_{k_1k_2}
\end{bmatrix}.$$
Equating to zero its three $2\times2$ minors gives exactly equations \eqref{wil}.
\end{proof}

We now compare the equations of Proposition~\ref{Wilnew} with those of Chen and Hsiao. Following the notation of \cite{chen2018strictly}, for $i=1,\dots,5$ we define

\begin{align}
P_i &= s_{i,i+2} s_{i+1,i+3} - s_{i,i+3} s_{i+1,i+2}, \\
Q_i &= s_{i,i+1} s_{i+2,i+3} - s_{i,i+2} s_{i+1,i+3},
\end{align}

where the indices are taken modulo $5$. The comparison requires matching
conventions for the areas. In \cite{chen2018strictly} the symbol $D_{kl}$
denotes the oriented area of the triangle on the three complementary
indices, listed in increasing order, and is strictly positive for a
counterclockwise strictly convex pentagon; in the notation of
Definition~\ref{def:D_dual} this is the \emph{unsigned} complementary minor
$$\widetilde{D}_{kl} := \Delta_{\{1,2,3,4,5\}\setminus\{k,l\}} = (-1)^{k+l+1}\,D_{kl},$$
and it is $\widetilde{D}$, not $D$, that appears in the equations of
\cite{chen2018strictly}. We use it throughout this comparison.

With this notation, the equations of Proposition~\ref{Wilnew} for the pair $(i,i+1)$ are equivalent to the vanishing of the $2\times2$ minors of the matrix
$$\begin{bmatrix}
P_{i+3} & Q_{i+4} & P_i\\
m_{i+2}\widetilde{D}_{i+3,i+4}& m_{i+3}\widetilde{D}_{i+2,i+4}& m_{i+4}\widetilde{D}_{i+2,i+3}
\end{bmatrix},$$
that is, to the equations
\begin{align}
\label{Ch1}  m_{i+2} Q_{i+4} \widetilde{D}_{i+3,i+4}&= m_{i+3} P_{i+3} \widetilde{D}_{i+2,i+4},\\ 
\label{ch2}  m_{i+2} P_{i} \widetilde{D}_{i+3,i+4}& =m_{i+4} P_{i+3} \widetilde{D}_{i+2,i+3}, \\ 
\label{ch3} m_{i+3} P_{i} \widetilde{D}_{i+2,i+4} &= m_{i+4} Q_{i+4} \widetilde{D}_{i+2,i+3}.
\end{align}
 Equations \eqref{Ch1} and \eqref{ch2} are, after the index shift $i\mapsto i+2$, the first two equations of (3.10) in \cite{chen2018strictly}, while the five equations \eqref{ch3} are consequences of \eqref{Ch1} and \eqref{ch2}. The third equation of (3.10), which is quadratic in the areas and free of masses, is obtained there by eliminating the masses between the first two, and so carries no further information.

Proposition~\ref{Wilnew} yields further equations. Define
$$R_{i}=s_{i,i+1}\,s_{i+2,i+3}-s_{i,i+3}\,s_{i+1,i+2}.$$
Then the equations of Proposition~\ref{Wilnew} for the pair $(i,i+2)$, $i=1,\dots,5$, are equivalent to the vanishing of the $2\times2$ minors of
$$\begin{bmatrix}
Q_{i+2} & R_{i+4} & R_i\\
m_{i+1}\widetilde{D}_{i+3,i+4}& m_{i+3}\widetilde{D}_{i+1,i+4}& m_{i+4}\widetilde{D}_{i+1,i+3}
\end{bmatrix},$$
and we obtain fifteen additional equations:

\begin{align}
m_{i+1}R_{i+4}\widetilde{D}_{i+3,i+4}&=m_{i+3}Q_{i+2}\widetilde{D}_{i+1,i+4},\\
 m_{i+1}R_i\widetilde{D}_{i+3,i+4}&=m_{i+4}Q_{i+2}\widetilde{D}_{i+1,i+3},\\
  m_{i+3}R_{i}\widetilde{D}_{i+1,i+4}&=m_{i+4}R_{i+4}\widetilde{D}_{i+1,i+3}.
\end{align}

This set of $30$ equations remains incomplete. In the next section, we derive a system of polynomial equations that complements those obtained by Chen and Hsiao \cite{chen2018strictly}.

\section{Williams' Formulae for \texorpdfstring{$k$}{k}-dimensional Central Configurations}\label{sec:williamsk}

The approach of the previous section applies to central configurations of $n$ bodies and dimension $k$ with $2\leq k \leq n-3$; the corresponding generic Brehm--Wintner--Conley dimension is $t=n-k-1$. We first define the dual Pl\"ucker coordinates of an $n$-body central configuration.

\begin{definition}\label{def:D_dual}
Let $x$ be a central configuration of the $n$-body problem. Let $I \in I_{t}$ be a multi-index and let $J$ be its strictly ordered complement. We define the \emph{dual Pl\"ucker coordinate} $D_I$ as:

\begin{equation}
D_I = \epsilon_{I, J} \Delta_J
\end{equation}
where $\epsilon_{I, J}$ is the shuffle sign of Definition \ref{def:shuffle_sign}.

\end{definition}

\begin{proposition}\label{Wilgen-prop}
Let $x$ be a central configuration of the $n$-body problem with dimension $k$, and let $t=n-k-1$. Let $I \in I_{t}$ and $L \in I_{t+1}$. Then for any distinct $u, v \in L$:

\begin{equation}\label{Wilgen}
 m_{u} \det(S_{I}^{L  \setminus \{v\}})D_{L \setminus \{u\}} =  m_{v} \det(S_{I}^{L \setminus \{u\}}) D_{{L \setminus \{v\}}}.
\end{equation}

\end{proposition}
\begin{proof}
The proof parallels that of Proposition~\ref{Wilnew}; the only additional care concerns the more involved index sets. By Proposition~\ref{prop:geometric_S}, applied with $I=L^{c}$, we get, for every $j\in\{1,\dots,n\}$, the equation

\begin{equation*}
 \sum_{l\in L} \epsilon_{l, L^c} m_l s_{jl}\Delta_{L^c \cup \{l\}}=0.
 \end{equation*}
By the definition of $D_I$, this becomes

\begin{equation*}
 \sum_{l\in L} \epsilon_{l, L^c}\,\epsilon_{L \setminus\{l\},\,L^c\cup\{l\}}\, m_l s_{jl}D_{L \setminus \{l\}}=0.
 \end{equation*}
Write $L=\{l_1<l_2<\dots<l_{t+1}\}$. By formula \eqref{eq:shuffle_sign}, $\epsilon_{l_i, L^c}=(-1)^{l_i-i}$, and, since the naturally ordered union of $L\setminus\{l_i\}$ and $L^c\cup\{l_i\}$ is $\{1,\dots,n\}$,
$$\epsilon_{L \setminus \{l_i\},\,L^c\cup\{l_i\}}= (-1)^{\left(\sum_{h \neq i }l_h\right) - \frac{t(t+1)}{2}}.$$
Consequently, the product $\epsilon_{l_i, L^c}\,\epsilon_{L \setminus\{l_i\},\,L^c\cup\{l_i\}}=(-1)^{\left(\sum_{h}l_h\right)-i-\frac{t(t+1)}{2}}$ alternates in sign as $i$ runs from $1$ to $t+1$. Absorbing a global sign, we get the equations 

\begin{equation} \label{eqkernelsimp}
 \sum_{h=1}^{t+1} (-1)^h m_{l_h} s_{jl_{h}}D_{L \setminus \{l_h\}}=0, \qquad j=1,\dots,n.
 \end{equation}

Consider the matrix

$$M^{L}_{I}= \begin{bmatrix}
s_{i_1 l_1} & ...& s_{i_1 l_{t+1}}\\
s_{i_2 l_1}&... & s_{i_2 l_{t+1}}\\
\vdots & \ddots& \vdots \\
s_{i_t l_1}&... & s_{i_t l_{t+1}}\\
0 &...& 0 \\
\end{bmatrix}.$$

With this notation, the vector
\begin{multline*}
v^{L}_{I}=\big(\det(S_{I}^{L\setminus \{l_1\}}),\ -\det(S_{I}^{L\setminus \{l_2\}}),\ \dots,\\
(-1)^{h+1}\det(S_{I}^{L\setminus \{l_h\}}),\ \dots,\ (-1)^{t}\det(S_{I}^{L\setminus \{l_{t+1}\}})\big)
\end{multline*}
belongs to the kernel of $M_{I}^{L}$: its entries are, up to a common sign, the cofactors of the zero row of $M_{I}^{L}$. If $v^{L}_{I}$ is the zero vector, the equations \eqref{Wilgen} hold trivially. Otherwise, some entry of $v^{L}_{I}$ is nonzero, the top $t\times(t+1)$ block of $M_{I}^{L}$ has rank $t$, and $\dim(\operatorname{Ker}(M_{I}^{L}))=1$. By equation \eqref{eqkernelsimp}, applied with $j=i_1,\dots,i_t$, the vector

\begin{multline*}
\tilde{v}_{L}=\big(-m_{l_1} D_{L \setminus \{l_1\}},\ m_{l_2} D_{L \setminus \{l_2\}},\ \dots,\\
(-1)^h m_{l_h} D_{L \setminus \{l_h\}},\ \dots,\ (-1)^{t+1} m_{l_{t+1}} D_{L \setminus \{l_{t+1}\}}\big)
\end{multline*}
also belongs to the kernel of $M_{I}^{L}$.

The two vectors are therefore proportional. Multiplying the $h$-th coordinate of each by $(-1)^{h}$ preserves proportionality, so the matrix
$$\begin{bmatrix}
\det(S_{I}^{L \setminus \{l_1\}})& \dots & \det(S_{I}^{L \setminus \{l_{t+1}\}})\\
m_{l_1}D_{L \setminus\{l_1\}}& \dots& m_{l_{t+1}}D_{L \setminus \{l_{t+1}\}}
\end{bmatrix}$$
has rank $1$; equating its $2\times2$ minors to zero yields exactly equations \eqref{Wilgen}.
\end{proof}

In Proposition~\ref{Wilnew} we only used the matrices $M_{ij}$, in which no diagonal term $s_{ii}$ appears. Removing this restriction yields further equations for planar central configurations of the five-body problem.

\begin{corollary}\label{Wilnew2}
Let $x$ be a planar central configuration of the five-body problem. If $1\leq i < j \leq 5$ and $k_1<k_2<k_3$ are distinct indices in $\{1,2,3,4,5\}$, then
\begin{align}\label{wil2}
m_{k_2}(s_{ik_2}s_{jk_{3}}-s_{ik_3}s_{jk_{2}})D_{k_1k_3}-m_{k_1}(s_{ik_1}s_{jk_{3}}-s_{ik_3}s_{jk_{1}})D_{k_2k_3}&=0,\\ \nonumber
m_{k_3}(s_{ik_2}s_{jk_{3}}-s_{ik_3}s_{jk_{2}})D_{k_1k_2}-m_{k_1}(s_{ik_1}s_{jk_{2}}-s_{ik_2}s_{jk_{1}})D_{k_2k_3}&=0,\\ \nonumber
m_{k_3}(s_{ik_1}s_{jk_{3}}-s_{ik_3}s_{jk_{1}})D_{k_1k_2}-m_{k_2}(s_{ik_1}s_{jk_{2}}-s_{ik_2}s_{jk_{1}})D_{k_1k_3}&=0.
\end{align}
\end{corollary}

\begin{proof}
Immediate from Proposition~\ref{Wilgen-prop} with $n=5$, $k=2$, $t=2$, $I=\{i,j\}$ and $L=\{k_1,k_2,k_3\}$.
\end{proof}

Proposition \ref{Wilnew} and Corollary \ref{Wilnew2} have structurally very similar statements. However, Corollary \ref{Wilnew2} yields an overdetermined system of 300 equations, whereas Proposition \ref{Wilnew} yields a much leaner system of only 30 equations. To compare the two systems algebraically, let $R = \mathbb{C}[m_i, D_{kl}, s_{ij}]$ be the polynomial ring in the variables $m_i$, $i=1,\dots,5$, $D_{kl}$, $1 \leq k <l\leq 5$ and $s_{ij}$, $1 \leq i \leq j\leq 5$. Let $I_{Prop} \subset R$ denote the ideal generated by the 30 equations from Proposition \ref{Wilnew}, The equations of Corollary \ref{Wilnew2} also involve the diagonal entries $s_{ii}$; we interpret these through the mass convention $m_i s_{ii} = -\sum_{j\neq i} m_j s_{ij}$ of Section~\ref{sec:wc-matrix} and clear denominators, and we let $I_{Cor} \subset R$ denote the ideal generated by the resulting 300 polynomials. The discrepancy between the two systems is genuine, in the following precise sense.

\begin{proposition}\label{prop:iprop-icor}
The inclusion $I_{Prop}\subseteq I_{Cor}$ is strict, and strict already at
the level of radicals: $\sqrt{I_{Prop}}\subsetneq\sqrt{I_{Cor}}$;
equivalently, $V(I_{Cor})\subsetneq V(I_{Prop})$.
\end{proposition}

\begin{proof}
The containment is clear: the thirty equations of Proposition~\ref{Wilnew}
are the equations of Corollary~\ref{Wilnew2} for the ten pairs $(I,L)$ with
$L=I^{c}$, and these involve no diagonal entry. For strictness it suffices
to exhibit a common zero of $I_{Prop}$ at which some generator of $I_{Cor}$
does not vanish: no power of that generator can then lie in $I_{Prop}$.

Take arbitrary masses with $m_5\neq0$, set $s_{12}=s_{45}=1$ and $s_{ij}=0$
for the eight remaining pairs, and set $D_{12}=1$, $D_{kl}=0$ otherwise. A
minor $\det(S_{\{i,j\}}^{\{a,b\}})=s_{ia}s_{jb}-s_{ib}s_{ja}$ with
$\{i,j\}\cap\{a,b\}=\emptyset$ is nonzero at this point only when its
two factors are $s_{12}$ and $s_{45}$, which forces $\{i,j\}$ to meet both
$\{1,2\}$ and $\{4,5\}$. Consequently, in each of the thirty equations of
Proposition~\ref{Wilnew} at most one minor is nonzero, and the system
reduces to $D_{13}=D_{23}=D_{34}=D_{35}=0$, which our choice satisfies. The
coordinate $D_{12}$ occurs only in the equations of the pairs $(4,5)$,
$(3,5)$ and $(3,4)$, all of whose minors vanish at the point. Hence the
point lies in $V(I_{Prop})$.

On the other hand, the equation of Corollary~\ref{Wilnew2} with
$I=\{1,4\}$, $L=\{1,2,5\}$ and $(u,v)=(1,5)$ reads
\[
m_1\,\det\big(S_{\{1,4\}}^{\{1,2\}}\big)\,D_{25}
\;=\;
m_5\,\det\big(S_{\{1,4\}}^{\{2,5\}}\big)\,D_{12}.
\]
At our point $\det(S_{\{1,4\}}^{\{2,5\}})=s_{12}s_{45}-s_{15}s_{24}=1$,
while $\det(S_{\{1,4\}}^{\{1,2\}})=s_{11}s_{24}-s_{12}s_{14}=0$
regardless of the diagonal convention, since $s_{24}=s_{14}=0$. The equation
evaluates to $0=m_5\neq0$, so the corresponding generator of $I_{Cor}$ does
not vanish at the point.
\end{proof}

\begin{remark}\label{rem:iprop-icor}
Exact computation refines the picture at the point used in the proof:
precisely $16$ of the $300$ equations fail there, and, for that choice of
$(m,s)$, the space of vectors $D$ annihilated by the thirty equations of
Proposition~\ref{Wilnew} is six-dimensional, while the full system of
Corollary~\ref{Wilnew2} cuts it down to the single line spanned by
$D_I=\epsilon_{I,I^c}\det\big(Y_{\cdot,\,I^c}\big)\big/\prod_{q\in I^c}m_q$,
where $Y$ is a kernel basis matrix of $S$: the dual Pl\"ucker vector of
$\ker(S)$ corrected by masses, anticipating the coordinates of
Section~\ref{sec:pwc}. The conceptual reason for the strict inclusion is
that the disjoint equations never involve the diagonal of $S$, while the
overlapping ones do, and through the mass convention the diagonal encodes
exactly the relation $Sm=0$: the extra equations of
Corollary~\ref{Wilnew2} carry mass information that the thirty equations
cannot reproduce, and their overdetermination is therefore not redundancy.
Determining the minimal number of generators of these ideals, or of their
radicals, remains open.
\end{remark}

We close this section by recording, for use in the next two sections, the classical exchange lemma of Steinitz \cite{steinitz1913bedingt}, in the iterated form that will organize our chains of multi-indices.

\begin{lemma}[Steinitz exchange lemma]\label{lem:steinitz}
Let $B$ and $B'$ be two bases of a finite-dimensional vector space. Then there is a sequence of bases
\[
B = B_0,\; B_1,\; \dots,\; B_r = B', \qquad r = |B \setminus B'|,
\]
in which each $B_{i+1}$ is obtained from $B_i$ by removing a single element and inserting a single element of $B'$. In particular, $B_i \subseteq B \cup B'$ for every $i$.
\end{lemma}

\begin{proof}
If $B \neq B'$, choose $w \in B' \setminus B$ and expand it in the basis $B$. Some $u \in B \setminus B'$ occurs in this expansion with a nonzero coefficient: otherwise $w$ would lie in the span of $B \cap B' \subseteq B' \setminus \{w\}$, contradicting the linear independence of $B'$. Since the coefficient of $u$ is nonzero, $u$ lies in the span of $B_1 := (B \setminus \{u\}) \cup \{w\}$, so $B_1$ spans; having the cardinality of a basis, it is a basis. Moreover $|B_1 \setminus B'| = |B \setminus B'| - 1$, and the sequence is completed by induction on $|B \setminus B'|$.
\end{proof}

\section{Dziobek--Williams Equations}\label{sec:dziobek-williams}
Let $x=(x_1, \dots, x_n)$ be a central configuration with dimension $k$, and let $t=n-k-1$ as before. We order the set $I_t$ of Section~\ref{sec:cc-equations} lexicographically.

We write $S^{(t)} := \big(\det(S_I^J)\big)_{I,J \in I_t}$ for the symmetric matrix of all $t\times t$ minors of $S$, with rows and columns indexed by $I_t$; we call it the \emph{Dziobek--Williams matrix}. The matrix $S^{(t)}$ is referred to in the literature as the $t$-th \emph{compound matrix} of $S$.

\begin{definition}\label{def:z_dual}
Let $x$ be a central configuration of the $n$-body problem. Let $I \in I_{t}$ be a multi-index and let $J \in I_{k+1}$ be its strictly ordered complement. We define the \emph{mass-weighted dual Pl\"ucker coordinate} $z_I$ as

\begin{equation}
z_I = \frac{D_I}{\prod_{i \in I} m_i}=\epsilon_{I, J} \frac{\Delta_J}{\prod_{i \in I} m_i}
\end{equation}
where $\epsilon_{I, J}$ is the shuffle sign of Definition \ref{def:shuffle_sign}.

\end{definition}

\begin{proposition}[Dziobek--Williams Equations]\label{dziowill}
Let $x$ be a central configuration of the $n$-body problem with dimension $k$. Let $I, J \in I_t$. Then, the minors of the matrix $S$ satisfy:

\begin{equation}\label{Dziowill}
 \det(S_{I}^{J}) = \kappa \, z_{I} \, z_{J},
\end{equation}

where $\kappa$ is a constant. In particular, the Dziobek--Williams matrix $S^{(t)}$ has rank at most $1$.
\end{proposition}

The argument below runs through Proposition~\ref{Wilgen-prop} and Lemma~\ref{lem:steinitz}, and generalizes the rank-one factorization of Moeckel for Dziobek configurations \cite[Proposition~3]{moeckel2001generic}, which is the case $t=1$.

\begin{proof}
Since $\delta(x)=k$, the configuration matrix $X$ has rank $k+1$, so its maximal minors $\Delta_J$ cannot all vanish and the vector $z=(z_J)_{J\in I_t}$ is nonzero. Write $\operatorname{supp}(z)=\{J\in I_t : z_J\neq 0\}$ for its support.

We first record that every minor whose column set lies outside the support vanishes: if $z_K=0$, then $\det(S_{I}^{K})=0$ for every $I\in I_t$. Indeed, $z_K=0$ means $\Delta_{K^c}=0$, so the $k+1$ columns of $X$ indexed by $K^c$ are linearly dependent, and some nonzero $c\in\mathbb{R}^{k+1}$ is orthogonal to all of them. The vector $\psi:=\mu X^{T}c$ is then supported in $K$; it is nonzero, because $c\neq 0$, $X^{T}$ is injective ($X$ has rank $k+1$), and $\mu$ is invertible; and it lies in $\ker(S)$, because the columns of $\mu X^{T}$ do (Section~\ref{sec:wc-matrix}). The rows indexed by $I$ of the identity $S\psi=0$ read $S_{I}^{K}\,\psi_{K}=0$ with $\psi_{K}\neq 0$, whence $\det(S_{I}^{K})=0$; by the symmetry of $S$, also $\det(S_{K}^{I})=0$.

Fix an arbitrary multi-index $I_{0} \in I_t$. We claim that
$$ \det(S_{I_{0}}^{J})\, z_{K} = \det(S_{I_0}^{K})\, z_{J} \qquad\text{for all } J, K \in I_t. $$
If $J$ or $K$ lies outside $\operatorname{supp}(z)$, both sides vanish by the previous paragraph. Suppose then that $J, K \in \operatorname{supp}(z)$: the columns of $X$ indexed by $J^{c}$ and those indexed by $K^{c}$ form two bases of $\mathbb{R}^{k+1}$. Lemma~\ref{lem:steinitz}, applied to these two bases, with the exchanges performed on the indexing sets, produces sets $J^{c} = C_0, C_1, \dots, C_r = K^{c}$, each obtained from the preceding one by removing a single index and inserting a single index of $K^{c}$, such that the columns of $X$ indexed by each $C_i$ form a basis. Taking complements, the sets $J_i := \{1,\dots,n\} \setminus C_i$ form a chain
$$ J = J_0,\; J_1,\; \dots,\; J_r = K $$
in $I_t$ in which consecutive terms differ by a single element, and every term lies in $\operatorname{supp}(z)$, since $z_{J_i} = \pm\,\Delta_{C_i}\big/\prod_{q \in J_i} m_q \neq 0$.

Consider two consecutive terms of the chain. Set $L := J_i \cup J_{i+1} \in I_{t+1}$, and let $u, v$ be defined by $\{u\} = J_i \setminus J_{i+1}$ and $\{v\} = J_{i+1} \setminus J_i$, so that $J_i = L \setminus \{v\}$ and $J_{i+1} = L \setminus \{u\}$. With these choices, Equation~\eqref{Wilgen} reads
$$ m_u\, \det(S_{I_0}^{J_i})\, D_{J_{i+1}} = m_v\, \det(S_{I_0}^{J_{i+1}})\, D_{J_i}. $$
Substituting $D_{J_i} = \big(\prod_{q \in J_i} m_q\big) z_{J_i}$ and $D_{J_{i+1}} = \big(\prod_{q \in J_{i+1}} m_q\big) z_{J_{i+1}}$ (Definition~\ref{def:z_dual}) and cancelling the common positive factor $m_u m_v \prod_{q \in J_i \cap J_{i+1}} m_q$, this becomes
$$ \det(S_{I_0}^{J_i})\, z_{J_{i+1}} = \det(S_{I_0}^{J_{i+1}})\, z_{J_i}, \qquad i = 0, \dots, r-1. $$
An induction on $i$ now gives $\det(S_{I_0}^{J})\, z_{J_i} = \det(S_{I_0}^{J_i})\, z_{J}$ for $i = 1, \dots, r$: the case $i=1$ is the relation just displayed, and the step from $i$ to $i+1$ multiplies the relation for $i$ by $z_{J_{i+1}}$, substitutes the adjacent relation, and cancels the nonzero factor $z_{J_i}$. The case $i = r$ is the claim for the pair $(J, K)$.

The claim says that all $2\times 2$ minors of the two-row matrix with rows $\big(\det(S_{I_{0}}^{J})\big)_{J\in I_t}$ and $(z_J)_{J\in I_t}$ vanish, so this matrix has rank at most one; since $z\neq 0$, its row space is spanned by $z$, and there exists a constant $c_{I_0}$ such that
$$ \det(S_{I_0}^{J}) = c_{I_{0}}\, z_{J} \qquad\text{for every } J\in I_t. $$
Since $S^{(t)}$ is symmetric, $\det(S_{I_0}^{J}) = \det(S_{J}^{I_0})$, and applying the same argument to the row indexed by $J$ gives $\det(S_{J}^{I_0}) = c_{J}\, z_{I_0}$. Therefore
$$ c_{I_{0}}\, z_{J} = c_{J}\, z_{I_{0}} \qquad\text{for all } I_0, J \in I_t: $$
all $2\times 2$ minors of the two-row matrix with rows $(c_J)_{J\in I_t}$ and $(z_J)_{J\in I_t}$ vanish, so it has rank at most one, and, since $z\neq 0$, there is a constant $\kappa$ with $c_J=\kappa\, z_J$ for every $J$. Hence $\det(S_{I}^{J}) = c_I\, z_J = \kappa\, z_{I}\, z_J$ for all $I,J\in I_t$, which proves the formula.
\end{proof}

\begin{remark}[Alternative proof via the Hodge dual]\label{rem:hodge-alt}
Proposition~\ref{dziowill} also admits a shorter, coordinate-free proof through Lemma~\ref{lem:hodge-dual}. If $\operatorname{rank}(S)<t$, every $t\times t$ minor vanishes and \eqref{Dziowill} holds with $\kappa=0$. If $\operatorname{rank}(S)=t$, then $\ker(S)=W$, the $(k+1)$-dimensional column space of $\mu X^{T}$ (Section~\ref{sec:wc-matrix}). The compound matrix $S^{(t)}$ is symmetric and has rank $\binom{t}{t}=1$ \cite[\S0.8.1]{hornjohnson2013}, so $S^{(t)}=c\,ww^{T}$ with $w$ spanning its column space, which is the line $\bigwedge^{t}(\operatorname{Im}S)=\bigwedge^{t}(W^{\perp})$. The rows of $X\mu$ form a basis of $W$ with maximal minors $\big(\prod_{q\in J}m_q\big)\Delta_J$ (Proposition~\ref{simple} and multilinearity), so Lemma~\ref{lem:hodge-dual} identifies $w_I\propto\epsilon_{I,I^c}\big(\prod_{q\in I^c}m_q\big)\Delta_{I^c}=\big(\prod_{q=1}^{n}m_q\big)z_I$ by Definition~\ref{def:z_dual}, whence $S^{(t)}=\kappa\,zz^{T}$. The same computation, applied to a kernel basis $Y$ of $S$ in place of $X\mu$, proves Proposition~\ref{prop:7.2-wc} directly. We have preferred the proofs through Propositions~\ref{Wilgen-prop} and~\ref{prop:6.2-wc}, which are self-contained and display the interdependence of the equations.
\end{remark}

The constant $\kappa$ vanishes precisely when $\operatorname{rank}(S)$ falls below the generic value $t=n-k-1$; see Remark~\ref{rem:kappa-vanishing} below. We now derive a formula for $\kappa$ in terms of the eigenvalues of $S$.

Example 5.6 in \cite{prells2003use} provides the following expansion for the characteristic polynomial of a matrix $B \in \mathbb{C}^{n \times n}$:
\begin{equation} \label{eqext1}
\det(B - \theta\, \mathrm{Id}_n) = \det(B) + (-1)^n \theta^n + \sum_{i=1}^{n-1} (-1)^i \theta^i \operatorname{tr} \mathcal{C}_{n-i}(B),
\end{equation}
where $\mathcal{C}_{n-i}(B)$ denotes the $(n-i)$-th compound matrix of $B$.

On the other hand, let $\nu_{1}, \nu_{2}, \dots, \nu_{n}$ denote the eigenvalues of $B$. The polynomial $\det(B - \theta\,\mathrm{Id}_n)$ can also be expressed in terms of the elementary symmetric polynomials
$$e_k(\nu_{1}, \dots, \nu_{n}) = \sum_{1 \le i_1 < \dots < i_k \le n} \nu_{i_1} \cdots \nu_{i_k}$$
of these eigenvalues as follows:
\begin{equation}\label{eqext2}
\det(B - \theta\,\mathrm{Id}_n) = \sum_{k=0}^{n} (-1)^k \theta^k e_{n-k}(\nu_{1}, \dots, \nu_{n}),
\end{equation}
where $e_0(\nu_{1}, \dots, \nu_{n}) = 1$. 

Since relations \eqref{eqext1} and \eqref{eqext2} hold for any square matrix $B$, comparing the two expansions, applied to $B=S$ (so that $\mathcal{C}_t(S)=S^{(t)}$), allows us to determine the trace of $S^{(t)}$ in terms of the eigenvalues of $S$:

\begin{equation}\label{traceformula}
\operatorname{Tr}(S^{(t)}) = \sum_{1 \le i_1 < \dots < i_{t} \le n} \nu_{i_1} \cdots \nu_{i_{t}} = e_{t}(\nu_{1}, \dots, \nu_{n}).
\end{equation}

\begin{remark}[The spectrum of the Brehm--Wintner--Conley matrix: related work]
\label{rem:spectrum-related-work}
Formula~\eqref{traceformula} places the constant $\kappa$ of
Proposition~\ref{dziowill} in contact with a line of research that has
developed independently of the Williams equations, and that concerns the
\emph{spectrum} of the Brehm--Wintner--Conley operator rather than the shape
constraints its kernel imposes.

The structure goes back to Brehm's 1908 thesis \cite{brehm1908partikulare}:
equations~(14)--(16) there are the kernel equations for the three coordinate
vectors of the configuration, the matrix itself is displayed in
equation~(17) on p.~21, and the dimension of the configuration is already
characterized by the vanishing of its minors of successive orders (in
embryo, the stratification by the Brehm--Wintner--Conley dimension). The matrix
was taken up by Meyer in 1933 and by Wintner \cite{wintner2014analytical}
in 1941; the history is traced in \cite{albouy2024limit,albouysun2025}. The spectral statements of that
literature concern the \emph{unshifted} operator $Z$: its eigenvalues are
real and nonnegative, and $\lambda$ occurs among them with multiplicity at
least $\delta(x)+1$: the \emph{trivial} eigenvalues, corresponding to
translations and to the configuration itself \cite{albouysun2025}. These
facts do not transfer to the matrix $S$ of the present paper without a
change of representative. The two are related by
$\check{Z}=Z-\lambda\mathcal{J}=-\mu S$, so $\check{Z}$ is \emph{similar}
to the symmetric matrix $-\mu^{1/2}S\mu^{1/2}$, whereas $S$ is
\emph{congruent}, not similar, to $\mu^{1/2}S\mu^{1/2}$: congruent matrices
share their inertia, not their spectrum. In particular, the $k+1$ trivial
directions lie in $\ker(S)$, so for $S$ they contribute the eigenvalue $0$,
not $\lambda$, and no positivity or ordering statement about the
eigenvalues $\nu_i$ of $S$ is made in this paper outside the
MacMillan--Bartky regime of Corollary~\ref{cor:mb-psd}. The first result on the
nontrivial part of the classical spectrum is due to Conley, who proved that
in a collinear central configuration every nontrivial eigenvalue exceeds
$\lambda$. Moeckel, extending a statement of
Pacella, proved that the arithmetic mean of the nonzero eigenvalues is at
least $\lambda$ in every dimension, with equality only for the regular simplex,
and showed that Conley's inequality does \emph{not} extend to all central
configurations: he exhibited a planar configuration of $474$ bodies carrying an
eigenvalue below $\lambda$ \cite{albouy2024limit,albouysun2025}. Most recently,
Albouy and Sun \cite{albouysun2025} established Conley's inequality for the
planar five-body problem, showing that the two nontrivial eigenvalues there are
strictly larger than the three trivial ones. As those authors observe, the
inequalities of this kind belong to two lines of research which had remained
separate: one beginning with Williams \cite{williams1938permanent} and
concerning constraints on the \emph{shape} of the configuration, the other
concerning the Hessian of the potential, its index, and hence the linear
stability of the associated self-similar motions.

Proposition~\ref{prop:kappa-formula} is a point of contact between the two.
Its left-hand side, $\kappa$, is defined by a purely shape-theoretic identity:
the rank-one factorization $\det(S_I^J)=\kappa z_Iz_J$ of the
Dziobek--Williams equations, whose ingredients $z_I$ are oriented volumes
weighted by masses. Its right-hand side is the elementary symmetric function
$e_t$ of the spectrum of $S$, divided by $\|z\|^2$. The identity therefore
expresses one and the same quantity in the two languages, and the
non-vanishing statement it yields, $\kappa^{\mathrm{bwc}}\neq0$
(Corollary~\ref{cor:kappa-wc-nonzero}), is a spectral statement (exactly $t$
eigenvalues of $S$ are nonzero, so the single surviving term of $e_t$ is their
product) with a purely geometric consequence: the Veronese model
$\Psi_t$ of Section~\ref{sec:veronese} has no base points. A quantitative comparison with the inequalities of Conley, Moeckel, and
Albouy--Sun requires translating those results from the classical operator
to the weighted representative $\mu^{1/2}S\mu^{1/2}$, an identification we
do not make here. In particular we do not pursue the sharper inequalities
of \cite{albouysun2025}, which constrain the individual eigenvalues of the
classical operator rather than elementary symmetric functions of the
spectrum of $S$; carrying them across and combining them
with~\eqref{traceformula} would bound $\kappa$, and seems worth
investigating.
\end{remark}

\begin{proposition}\label{prop:kappa-formula} Let $\kappa$ be the constant of Proposition~\ref{dziowill}, and let $\nu_{1}, \dots, \nu_{n}$ be the eigenvalues of the matrix $S$. Then
$$ 
\kappa = \frac{e_{t}(\nu_{1}, \dots, \nu_{n})}{\|z\|^2},
$$
where
$$e_{t}(\nu_{1}, \dots, \nu_{n})=\operatorname{Tr}(S^{(t)}) = \displaystyle \sum_{1 \leq i_1 < \dots < i_{t} \leq n} \nu_{i_1} \cdots \nu_{i_{t}},$$ and $z$ is the vector of mass-weighted dual Pl\"ucker coordinates of Definition \ref{def:z_dual}.
\end{proposition}

\begin{proof}
    By Proposition \ref{dziowill}, the matrix $S^{(t)}$ has rank at most $1$ and admits the factorization 
$$S^{(t)} = \kappa z z^T,$$ 
where $z$ is the column vector composed of the coordinates $z_I$. Computing the trace of $S^{(t)}$ yields:

\begin{equation}\label{traceformula2}
\operatorname{Tr}(S^{(t)}) = \sum_{I \in I_t} \det(S_I^I) = \kappa \sum_{I \in I_t} z_I^2 = \kappa \|z\|^2.
\end{equation}

Equating both expressions \eqref{traceformula} and \eqref{traceformula2} for the trace allows us to explicitly determine the constant $\kappa$ in terms of both the geometric data ($z$) and the spectral data of $S$:

$$\kappa = \frac{e_{t}(\nu_{1}, \dots, \nu_{n})}{\|z\|^2}.$$

Since $S^{(t)}$ is the compound matrix of $S$ of order $t$, its rank satisfies $\operatorname{rank}(S^{(t)})=\binom{\operatorname{rank}(S)}{t}$ \cite[\S0.8.1]{hornjohnson2013}. When $S^{(t)}$ has rank $1$, this forces $\operatorname{rank}(S)=t$: exactly $t$ of the eigenvalues $\nu_{1},\dots,\nu_{n}$ are nonzero. Consequently, every term of $e_{t}(\nu_{1},\dots,\nu_{n})$ that omits even one of these $t$ nonzero eigenvalues vanishes identically, and the symmetric polynomial collapses to its single surviving term, the product of all $t$ nonzero eigenvalues of $S$. In particular, $\kappa$ is simply this product divided by $\|z\|^2$.
\end{proof}

\begin{remark}\label{rem:kappa-vanishing}
The identity of Proposition~\ref{prop:kappa-formula} holds at every rank. If the actual rank of $S$ falls short of the generic value $t=n-k-1$, then every $t\times t$ minor of $S$ vanishes, forcing $\kappa=0$ in \eqref{Dziowill} (some $z_I$ is nonzero, since $\operatorname{rank}(X)=k+1$), while $e_t(\nu_1,\dots,\nu_n)=0$ because fewer than $t$ eigenvalues are nonzero: both sides of the formula vanish, the identity remains valid, but it becomes vacuous, and Proposition~\ref{dziowill} collapses to $0=0$. What does require $\operatorname{rank}(S)=t$ (equivalently, $\operatorname{rank}(S^{(t)})=1$) is the non-vanishing of $\kappa$ and the reading of $e_t(\nu)$ as the product of the nonzero eigenvalues. Section~\ref{sec:pwc} removes this restriction entirely: from that section onward the symbol $t$ denotes the \emph{true} rank $t=\operatorname{bwc}(x)=\operatorname{rank}(S)$ (Definition~\ref{def:wc-dimension}), which agrees with $n-k-1$ exactly in the generic case, and Proposition~\ref{prop:7.3-wc} and Corollary~\ref{cor:kappa-wc-nonzero} give the corresponding constant $\kappa^{\mathrm{bwc}}$ unconditionally, with no genericity or non-degeneracy hypothesis.
\end{remark}

\subsection*{Sign structure and the MacMillan--Bartky condition}

The Dziobek--Williams equations govern not only the vanishing of the minors
of $S$ but also their signs, and the signs of the entries themselves have a
mechanical meaning, which we record here.

\begin{lemma}[Stress form of the central configuration equations]\label{lem:stress-form}
Let $x$ be a central configuration. Then, for every $j\in\{1,\dots,n\}$,
\begin{equation}\label{eq:self-stress}
\sum_{i\neq j}\omega_{ij}\,(x_i-x_j)=0,
\qquad \omega_{ij}:=m_i m_j\, s_{ij}.
\end{equation}
\end{lemma}

\begin{proof}
Equation~\eqref{eqccro} states that
$\sum_{i\neq j} m_i s_{ij}\,(x_i-x_j)=0$; multiplying by $m_j>0$ gives
\eqref{eq:self-stress}. The weights are symmetric because $S$ is.
\end{proof}

In the language of rigidity theory, \eqref{eq:self-stress} says that the
symmetric weights $\omega_{ij}$ form a \emph{self-stress} (equilibrium
stress) of the complete framework on the points $x_1,\dots,x_n$
\cite[Ch.~5]{connellyguest2022}; this stress reading of the central
configuration equations is classical, see for instance
\cite{moeckel2015notes}. For the potentials $U_a$ with $a<0$, the weight
$\omega_{ij}$ is positive exactly when $r_{ij}<r_0$ and negative exactly
when $r_{ij}>r_0$: the canonical length separates the pairs pulled together
from the pairs pushed apart (for $a>0$ the orientation is reversed). For
configurations in convex position this suggests the following condition.

\begin{definition}[MacMillan--Bartky condition]\label{def:mb-condition}
Let $x$ be a planar configuration whose bodies are in strictly convex
position; call \emph{sides} the pairs that are edges of the convex hull and
\emph{diagonals} the remaining pairs. The configuration satisfies the
\emph{MacMillan--Bartky condition} if
\[
s_{ij}>0\ \text{for every side},
\qquad
s_{ij}<0\ \text{for every diagonal}.
\]
\end{definition}

The name records the origin of this sign analysis: for four bodies under
the Newtonian potential, MacMillan and Bartky proved that in an admissible
convex quadrilateral the four sides are at most $r_0$ and the two diagonals
at least $r_0$ \cite[\S14]{macmillanbartky1932}, which in the present
notation is exactly the condition above; establishing
these signs for five bodies was the goal of Williams
\cite{williams1938permanent}, accomplished for strictly convex central
configurations by Chen and Hsiao \cite{chen2018strictly}.

\begin{proposition}[Basic tensegrity structure]\label{prop:mb-tensegrity}
Let $x$ be a strictly convex planar central configuration satisfying the
MacMillan--Bartky condition. Then the self-stress
$\omega_{ij}=m_im_js_{ij}$ of Lemma~\ref{lem:stress-form} is positive on
every side and negative on every diagonal: $x$ underlies a tensegrity
framework with cables along the boundary of its convex hull and struts
along all of its diagonals, carried by the proper equilibrium stress
\eqref{eq:self-stress}, in the sign convention of
\cite{connelly1982rigidity,connellyguest2022}.
\end{proposition}

\begin{proof}
Immediate from Lemma~\ref{lem:stress-form}, the positivity of the masses,
and Definition~\ref{def:mb-condition}.
\end{proof}

The Dziobek--Williams equations produce the underlying sign dichotomy for
free in the first classical case.

\begin{corollary}[Four bodies]\label{cor:mb-four}
Let $x$ be a strictly convex planar central configuration of four bodies.
Then the four sides carry a common sign of $s_{ij}$ and the two diagonals
the opposite common sign: $x$ satisfies either the MacMillan--Bartky
condition or the condition with both signs reversed.
\end{corollary}

\begin{proof}
Here $k=2$ and $t=n-k-1=1$, so Proposition~\ref{dziowill} reads
$s_{ij}=\kappa\, z_i z_j$. At $t=1$,
$D_i=\epsilon_{\{i\},\{i\}^c}\,\Delta_{\{i\}^c}$ is the $i$-th signed
maximal minor of the configuration matrix $X$, so
$z_i=D_i/m_i=\delta_i/m_i$ for a suitable generator $\delta$ of $\ker(X)$.
For four bodies in strictly convex position every $\delta_i$ is nonzero (a
vanishing coefficient would force the remaining three bodies to be
collinear), and the Radon partition of the four points separates the
endpoints of the two diagonals, so the signs of $\delta$ alternate along
the cyclic order. Hence $z_iz_j<0$ exactly on the sides and $z_iz_j>0$
exactly on the diagonals. Moreover $\kappa\neq0$: otherwise $S=0$ and all
six mutual distances would equal $r_0$, which is impossible in the plane.
The dichotomy follows. In the Newtonian case the second alternative does
not occur: MacMillan and Bartky rule it out by observing that four sides
greater than $r_0$ together with two diagonals smaller than $r_0$ is
geometrically impossible for a convex quadrilateral
\cite[\S14]{macmillanbartky1932}.
\end{proof}

The deeper content of Proposition~\ref{prop:mb-tensegrity} comes from a
theorem of Connelly, Theorem~5 of \cite{connelly1982rigidity}, stated
there for proper stresses that may vanish on some members and for struts
on an arbitrary subset of the internal edges; we state the strict-sign
case used here, reading ``convex'' in the strictly convex sense of
\cite[p.~23]{connelly1982rigidity}. This is the case Connelly's proof
covers: its interpolation endpoint, the Cauchy polygon, is defined there
for strictly convex vertex sets only.

\begin{theorem}[Connelly {\cite[Theorem~5]{connelly1982rigidity}}]\label{thm:connelly}
Let $p_1,\dots,p_n$, $n\ge4$, be the vertices, in cyclic order, of a
strictly convex polygon in the plane, and let $\omega$ be an equilibrium
stress on the complete graph on these vertices,
$\sum_{j\neq i}\omega_{ij}(p_j-p_i)=0$ for every $i$, with
$\omega_{i,i+1}>0$ on the edges of the polygon and $\omega_{ij}<0$ on the
diagonals. Then the stress matrix $\Omega$, defined by
$\Omega_{ij}=-\omega_{ij}$ for $i\neq j$ and
$\Omega_{ii}=\sum_{j\neq i}\omega_{ij}$, is positive semidefinite of rank
$n-3$.
\end{theorem}

\begin{corollary}[Definiteness and maximal rank in the MacMillan--Bartky
regime]\label{cor:mb-psd}
Let $x$ be a strictly convex planar central configuration of $n\ge4$
bodies satisfying the MacMillan--Bartky condition. Then the stress matrix
of the tensegrity of Proposition~\ref{prop:mb-tensegrity} is
$\Omega=-\mu S\mu$, it is positive semidefinite of rank $n-3$, and
consequently $S$ is negative semidefinite with nullity exactly $3$. In
particular $\mathrm{bwc}(x)=n-3$, the maximal value for $k=2$: vertical
degeneracy cannot occur in the MacMillan--Bartky regime, and by
Proposition~\ref{prop:wc-defect} such a configuration is not a limit of
central configurations of dimension greater than two.
\end{corollary}

\begin{proof}
The off-diagonal entries of the stress matrix are
$\Omega_{ij}=-\omega_{ij}=-m_im_js_{ij}$, and the diagonal ones are
$\Omega_{ii}=\sum_{j\neq i}m_im_js_{ij}=-m_i^{2}s_{ii}$, by the mass
convention $m_is_{ii}=-\sum_{j\neq i}m_js_{ij}$; hence $\Omega=-\mu S\mu$.
Under the MacMillan--Bartky condition the stress
$\omega_{ij}=m_im_js_{ij}$ of Lemma~\ref{lem:stress-form} satisfies the
hypotheses of Theorem~\ref{thm:connelly}, so $\Omega\succeq0$ with
$\operatorname{rank}(\Omega)=n-3$. Since $\mu$ is invertible,
$S=-\mu^{-1}\Omega\mu^{-1}$ is congruent to $-\Omega$, and congruence
preserves inertia: $S\preceq0$ with nullity $3$. Finally
$\mathrm{bwc}(x)=\operatorname{rank}(S)=n-3=n-k-1$, and the last claim is
the necessity statement of Proposition~\ref{prop:wc-defect}.
\end{proof}

In the language of rigidity theory, the tensegrity of
Proposition~\ref{prop:mb-tensegrity} is then uniquely embedded in
$\mathbb{R}^{d}$ for every $d\ge2$, hence universally rigid
\cite[Corollary~1]{connelly1982rigidity}; in the terminology of
\cite{connellyguest2022} it is \emph{super stable}.

\begin{remark}[Sketch of Connelly's argument]\label{rem:connelly-sketch}
Since \cite{connelly1982rigidity} lies outside the usual references of
this area, we sketch its two-step argument for the reader's convenience.
First, $\operatorname{nullity}(\Omega)=3$: if it were larger, the same
stress $\omega$ would equilibrate a genuine spatial lift
$\bar p\in\mathbb{R}^{3}$ of the polygon, projecting orthogonally onto
the plane (Lemma~2 of \cite{connelly1982rigidity}). Passing to the convex hull of $\bar p$, its upper facets
project to a subdivision of the original convex polygon with at least one
interior edge $e$. Convexity and properness force some strut to cross $e$
(otherwise lengthening $e$ alone, holding every cable fixed, would
strictly decrease the energy at what is supposed to be a critical point, a
contradiction), and flattening the corresponding hinge of $\bar p$
strictly decreases the energy again, a second contradiction. Hence the
lift cannot be genuinely three-dimensional, and
$\operatorname{nullity}(\Omega)=3$. Second, for positive semidefiniteness,
Connelly interpolates linearly between $\Omega$ and the stress matrix of
the \emph{Cauchy polygon} on the same points (struts only on the short
diagonals $\{i,i+2\}$), which is positive semidefinite by an explicit
induction on the number of vertices (Lemma~4 there). Every matrix on this segment is
itself a proper-stress matrix for the same points, so the nullity-three
argument applies throughout; since the nullity never drops, positive
semidefiniteness transfers from the Cauchy-polygon endpoint to $\Omega$ by
continuity (Lemma~3 there).
\end{remark}

\begin{remark}[The condition is genuinely restrictive]\label{rem:mb-polygon}
Strict convexity does not imply the MacMillan--Bartky condition. For four
bodies it is automatic in the Newtonian case, by Corollary~\ref{cor:mb-four}
and the geometric exclusion of \cite[\S14]{macmillanbartky1932}; but it can
fail already at $n=5$: Chen and Hsiao prove that every side of a strictly
convex planar five-body central configuration satisfies $s_{ij}>0$, while
only at least three of the five diagonals are forced to satisfy $s_{ij}<0$
\cite[Theorems~5.2 and~6.1]{chen2018strictly}, and they exhibit numerical
examples in which one or two diagonals satisfy $s_{ij}>0$
\cite[\S8]{chen2018strictly}. For the regular $n$-gon with equal masses,
inscribed in the unit circle, the thresholds are explicit. The mutual distances are $d_k=2\sin(k\pi/n)$, and
pairing equation~\eqref{eqccro} with $x_j$ gives
\[
r_0^{2a}\;=\;\frac{1}{2n}\sum_{k=1}^{n-1}d_k^{\,2a+2},
\]
since $(x_{j+k}-x_j)\cdot x_j=-d_k^{2}/2$ and $\sum_{k}d_k^{2}=2n$. For
$a<0$ the condition reads $d_1<r_0<d_2$, and the formula decides it
directly: it holds precisely for $4\le n\le7$ at the Newtonian exponent
$a=-3/2$, precisely for $4\le n\le6$ at $a=-1/2$, and precisely for
$4\le n\le10$ at $a=-5/2$. (Failure for all larger $n$ follows, for
$a\le-1$, from the bound $r_0^{2a}\le\tfrac12\,d_1^{\,2a+2}$ together with
finitely many evaluations of the formula, and, for $-1\le a<0$, from the
fact that $r_0$ tends to a positive constant while $d_2\to0$.) Past the
threshold the shortest diagonals carry positive stress and the
cable--strut pattern of Proposition~\ref{prop:mb-tensegrity} breaks,
although the configuration remains strictly convex; the sign structure
beyond the MacMillan--Bartky regime will be developed elsewhere.
\end{remark}

\section{Pl\"ucker--BWC Coordinates for Central Configurations}\label{sec:pwc}

In this section we define another set of Pl\"ucker coordinates for a central configuration, bridging the geometry of the configuration with the algebraic properties of the Brehm--Wintner--Conley matrix.

\begin{definition}\label{def:wc-dimension}
Let $x$ be a central configuration. The \emph{Brehm--Wintner--Conley dimension} of $x$ is defined as
\[
t = \mathrm{bwc}(x) := \mathrm{rank}(S).
\]
\end{definition}

Since the $k+1$ linearly independent columns of $\mu X^T$ lie in the right kernel of $S$, the rank is bounded by $\mathrm{bwc}(x) \le n-k-1$; equality holds exactly in the generic (equivalently, maximal-rank) case, recovering the value $t=n-k-1$ used in Sections~\ref{sec:wc-matrix}--\ref{sec:dziobek-williams} (density of this stratum is asserted only in the precise sense of Remark~\ref{rem:wc-grassmannian-remarks}(ii)). From here on, $t$ denotes the true rank $\mathrm{bwc}(x)$. Hence $\delta(x) + \mathrm{bwc}(x) \le n-1$.

\begin{remark}[The stratum $t=0$]\label{rem:t-zero}
The value $\mathrm{bwc}(x)=0$ does occur, and the corresponding stratum is
completely explicit: $\mathrm{bwc}(x)=0$ means $S=0$, that is,
$r_{ij}=r_{0}$ for all $i<j$, so $x$ is a regular simplex and
$\delta(x)=n-1$. Conversely, every central configuration of the maximal
dimension $\delta(x)=n-1$ satisfies
$\mathrm{bwc}(x)\le n-(n-1)-1=0$, hence $S=0$: it is the regular simplex.
On this stratum we use the conventions $I_{0}=\{\emptyset\}$,
$\det(S_{\emptyset}^{\emptyset})=1$ (the empty minor), $S^{(0)}=(1)$, and
empty products equal to $1$. With these conventions the Main Theorem holds
literally at $t=0$: $\ker(S_x)=\mathbb{R}^{n}$, the dual Pl\"ucker vector
is the single nonzero scalar $D^{\mathrm{bwc}}_{\emptyset}=\det(Y)$ for a
kernel basis matrix $Y$, the factorization reads
$(1)=\kappa^{\mathrm{bwc}}\det(Y)^{2}$ with
$\kappa^{\mathrm{bwc}}=\det(Y)^{-2}\neq0$, and $\Psi_{0}$ is the constant
map to $\mathbb{P}^{0}$.
\end{remark}

\begin{remark}\label{rem:wc-invariance}
Two clarifications on Definition~\ref{def:wc-dimension} are worth making explicit. First, since $S$ is built from the mass vector $m=(m_1,\dots,m_n)$ as well as from the positions $x$, the class $[x]$ used throughout this section should be understood as a class of the \emph{pair} $(x,m)$: the symmetry-and-homothety equivalence acts on positions only, with $m$ held fixed. Second, $t=\mathrm{bwc}(x)$ is genuinely an invariant of this class. Under an isometry of $\mathbb{R}^k$, $S$ does not change at all, since it depends only on the mutual distances $r_{ij}$ and on $m$. Under a homothety $x\mapsto \alpha x$, both $r_{ij}$ and $r_0$ scale by $\alpha$ (they satisfy the same relation $r_0^{2a}=M^{-1}\lambda$), so $s_{ij}=r_{ij}^{2a}-r_0^{2a}$ scales by the same nonzero factor $\alpha^{2a}$ for \emph{every} $i,j$, and hence $S \mapsto \alpha^{2a}S$. A nonzero overall rescaling changes neither the rank of $S$ nor its kernel, so both $\mathrm{bwc}(x)$ and $\ker(S_x)$ (cf.\ Definition~\ref{def:wc-grassmannian-map} below) depend only on $[x]$, exactly, and not merely up to some ambiguity.
\end{remark}

\begin{definition}[Brehm--Wintner--Conley defect]\label{def:wc-defect}
Let $x$ be a central configuration of dimension $k$, and let
$A_x:=\operatorname{Im}(\mu X^{T})\subseteq\ker(S_x)$ be the trivial part of
the kernel, of dimension $k+1$. The \emph{Brehm--Wintner--Conley defect} of $x$ is
\[
\operatorname{vdef}(x)\;:=\;\dim\big(\ker(S_x)/A_x\big)\;=\;n-k-1-\mathrm{bwc}(x).
\]
\end{definition}

\begin{proposition}[Vertical degeneracy is necessary at meetings of strata]\label{prop:wc-defect}
Let $x$ be a central configuration of dimension $k$. The following
are equivalent: \textup{(i)} $\mathrm{bwc}(x)<n-k-1$; \textup{(ii)} $A_x$ is
a proper subspace of $\ker(S_x)$; \textup{(iii)} $\operatorname{vdef}(x)>0$.
Moreover, if $[x]$ is the limit, in the quotient of collision-free
configurations by isometries and homotheties with fixed masses, of classes
of central configurations of dimension $k'\ge k+1$, then
\[
\operatorname{vdef}(x)\;\ge\;k'-k\;\ge\;1.
\]
The converse is not asserted: a positive defect is necessary for $x$ to be
such a limit, not sufficient.
\end{proposition}

\begin{proof}
The equivalences follow from $\dim A_x=k+1$ and
$\dim\ker(S_x)=n-\mathrm{bwc}(x)$. For the limit statement, write $[x]=\lim_j[x_j]$ with $\delta(x_j)=k'$,
and choose representatives normalized by $r_0=1$, which is possible within each class
since $r_0$ scales linearly under homotheties
(Remark~\ref{rem:wc-invariance}). The multiplier, and hence $r_0$, depends
continuously on a collision-free configuration; thus if $y_j\to y$ are
convergent representatives, then $r_0(y_j)\to r_0(y)>0$, the rescaled
representatives $y_j/r_0(y_j)$ converge to $y/r_0(y)$, and
$s_{ij}=r_{ij}^{2a}-1$ converges entrywise: $S_{x_j}\to S_x$. A nonzero minor of $S_x$ remains nonzero for all nearby
matrices, so $\mathrm{bwc}(x)\le\mathrm{bwc}(x_j)\le n-k'-1$ for all large
$j$, and therefore
$\operatorname{vdef}(x)=n-k-1-\mathrm{bwc}(x)\ge k'-k\ge1$. A sequence of
varying dimensions reduces to this case along a subsequence of constant
dimension, since the dimensions take finitely many values; the bound then
holds with $k'$ any dimension attained infinitely often.
\end{proof}

\begin{remark}[Relation to the work of Albouy and Fernandes]\label{rem:af-comparison}
The limit statement above is the quantitative, fixed-mass form of
\cite[Proposition~1.4]{albouy2024limit}, which is proved there in greater
generality: the masses are allowed to converge along the sequence, and the
rank-semicontinuity argument applies verbatim (it is the first of the two
proofs given there). The second proof in \cite{albouy2024limit} is
constructive: parametrizing the degenerating family by a Taylor series, in
the sense of Hampton and Moeckel \cite{hamptonmoeckel2006}, exhibits an
explicit kernel vector, namely the leading coefficient of the vertical
coordinates along the branch \cite[Propositions~3.1
and~3.2]{albouy2024limit}; this is information that the semicontinuity
argument does not provide. The decisive companion result at $n=5$ is
\cite[Theorem~1.2]{albouy2024limit}: a planar five-body central
configuration has no nontrivial vertical degeneracy, that is,
$\operatorname{vdef}(x)=0$. Combined with the necessity above, a planar
five-body central configuration is not a limit of spatial ones
\cite[Theorem~1.1]{albouy2024limit}.
\end{remark}

One of the main difficulties in proceeding with this approach lies in the fact that, for a central configuration $x$ of fixed dimension, the rank of the matrix $S$ is generally unknown. A comprehensive account of the existing results on this topic is provided in~\cite{albouy2024limit}. In particular, Albouy and Fernandes prove in that work that a planar central configuration of five bodies satisfies $\mathrm{rank}(S)=2$. To proceed, we therefore fix the rank of $S$.

Let $\mathfrak{X}_{n,k,t}$ be the set of equivalence classes of central configurations of $n$ bodies with dimension $k$ and Brehm--Wintner--Conley dimension $t$, modulo symmetries and homotheties.

\begin{definition}\label{def:wc-grassmannian-map}
Let $[x] \in \mathfrak{X}_{n,k,t}$. The \emph{BWC Grassmannian Map} is defined as
\[
\Phi: \mathfrak{X}_{n,k,t} \to \mathrm{Gr}(n-t,n), \qquad [x] \mapsto \ker(S_x).
\]

The Pl\"ucker coordinates of the subspace $\ker(S_x)$, representing a point in $\mathbb{P}\big(\bigwedge^{n-t}\mathbb{R}^n\big)$, are well defined for the class $[x]$ and are called the \emph{Pl\"ucker--BWC coordinates} of the configuration, denoted $\Delta^{\mathrm{bwc}}_K$.

Since $\ker(S_x) = \ker(S_{\tilde x})$ for equivalent configurations, $\Phi$ is strictly well defined.
\end{definition}

\begin{remark}\label{rem:wc-grassmannian-remarks}
Three further remarks on Definition~\ref{def:wc-grassmannian-map}.
\begin{enumerate}
\item[(i)] The target of $\Phi$ is forced, not chosen: by Definition~\ref{def:wc-dimension} and the rank--nullity theorem, $\dim\ker(S_x)=n-t$ automatically, so $\Phi$ always lands in $\mathrm{Gr}(n-t,n)$ without any further hypothesis.
\item[(ii)] Because $\mathrm{rank}(S)$ can only \emph{drop} under specialization (the locus $\{S : \mathrm{rank}(S)\le \tau\}$ is Zariski closed for every $\tau$), the maximal-rank stratum $\mathfrak{X}_{n,k,n-k-1}$ is open among dimension-$k$ classes, while for each $\tau<n-k-1$ the union of the strata with $\mathrm{bwc}(x)\le\tau$ is closed; each stratum is therefore locally closed. On every irreducible component of an algebraic family of dimension-$k$ classes that is not contained in a lower-rank locus, the maximal-rank stratum is a nonempty Zariski-open, hence dense, subset; no density assertion is intended for arbitrary connected components, on which a nonempty open set need not be dense. This is why $\Phi$ can only be defined stratum by stratum: its own target $\mathrm{Gr}(n-t,n)$ changes dimension with $t$, so no single map is defined uniformly over all dimension-$k$ classes at once.
\item[(iii)] Exactly as with the classical $\Delta_I$, the coordinates $\Delta^{\mathrm{bwc}}_K$ are canonical only up to one overall nonzero scalar (the choice of basis $Y$ for $\ker(S_x)$), which is precisely why Definition~\ref{def:wc-grassmannian-map} phrases them projectively, as a point of $\mathbb{P}(\bigwedge^{n-t}\mathbb{R}^n)$, rather than as numbers. From here on, once a representative $Y$ has been fixed, we treat $\Delta^{\mathrm{bwc}}_K$ as an honest scalar; every identity stated below is a ratio or a product of an \emph{even} number of such scalars (Propositions~\ref{prop:6.2-wc} and~\ref{prop:7.2-wc}), hence independent of this choice, exactly as in the classical case.
\end{enumerate}
\end{remark}

We can now formulate these relations directly in terms of the equivalence class $[x]$ and its associated algebraic coordinates.

In the maximal rank scenario where $\mathrm{bwc}(x) = n-k-1$ (e.g.\ planar $5$-body configurations, where $t=2$), the right kernel of $S$ is precisely spanned by the columns of the mass-weighted configuration matrix $\mu X^T$. Consequently, the algebraic Pl\"ucker--BWC coordinates $\Delta^{\mathrm{bwc}}_K$ naturally factorize into the purely geometric Pl\"ucker coordinates (the oriented volumes $\Delta_K$ of the subconfigurations), scaled by the intrinsic masses of the bodies involved.

Hence, for any index set $K$ with $|K|=k+1$, we obtain the relation
\begin{equation}\label{eq:wc-generic-factorization}
\Delta^{\mathrm{bwc}}_K = \Big(\prod_{q\in K} m_q\Big)\, \Delta_K.
\end{equation}

\subsection{Equations for Central Configurations with Fixed Brehm--Wintner--Conley Dimension}

In this subsection we obtain determinantal identities for the Pl\"ucker--BWC coordinates $\Delta^{\mathrm{bwc}}$, generalizing Lemma~\ref{lem:eqg}, Proposition~\ref{prop:geometric_S}, Proposition~\ref{Wilgen-prop}, Proposition~\ref{dziowill} and Proposition~\ref{prop:kappa-formula} to an arbitrary, possibly non-generic, Brehm--Wintner--Conley dimension $t$.

Throughout the rest of this subsection, fix $[x] \in \mathfrak{X}_{n,k,t}$ and let $Y$ be an $(n-t)\times n$ matrix whose rows form a basis of $\ker(S_x)$, so that $\Delta^{\mathrm{bwc}}_K = \det(Y_{\cdot,K})$ for $K \in I_{n-t}$; write $y_q \in \mathbb{R}^{n-t}$ for the $q$-th column of $Y$.

\begin{lemma}[Analogue of Lemma~\ref{lem:eqg}]\label{lem:3.1-wc}
With notation as above,
\[
\sum_{q=1}^n s_{iq}\, y_q = 0 \in \mathbb{R}^{n-t}, \qquad i=1,\dots,n.
\]
\end{lemma}

\begin{proof}
Each row of $Y$, regarded as a vector in $\mathbb{R}^n$, lies in $\ker(S)$ by construction, so $SY^T=0$. The $i$-th row of this matrix identity reads $\sum_{q} s_{iq}y_q=0$.
\end{proof}

\begin{remark}
In the maximal-rank case $t=n-k-1$, taking $Y=X\mu$ (whose rows are the columns of $\mu X^T$, shown above to lie in $\ker(S)$) gives $y_q=m_qg_q$, and Lemma~\ref{lem:3.1-wc} reduces to Lemma~\ref{lem:eqg} (after relabeling $i\leftrightarrow q$ and using $s_{iq}=s_{qi}$).
\end{remark}

\begin{proposition}[Analogue of Proposition~\ref{prop:geometric_S}]\label{prop:3.2-wc}
For every $H\in I_{n-t-1}$ and every $i\in\{1,\dots,n\}$,
\[
\sum_{l\notin H}\epsilon_{l,H}\, s_{il}\,\Delta^{\mathrm{bwc}}_{H\cup\{l\}} = 0.
\]
\end{proposition}

\begin{proof}
Fix $i$ and take the wedge product of both sides of Lemma~\ref{lem:3.1-wc} with $w_H := y_{h_1}\wedge\cdots\wedge y_{h_{n-t-1}}$, where $H=\{h_1<\cdots<h_{n-t-1}\}$:
\[
\sum_{q=1}^n s_{iq}\,(y_q\wedge w_H) = 0 \in \bigwedge^{n-t}\mathbb{R}^{n-t}.
\]
If $q \in H$, then $y_q \wedge w_H = 0$. Otherwise, exactly as in the proof of Proposition~\ref{prop:geometric_S}, $y_q\wedge w_H = \epsilon_{q,H}\,\Delta^{\mathrm{bwc}}_{H\cup\{q\}}\, e$, where $e = e_1\wedge\cdots\wedge e_{n-t}$. Substituting yields the result.
\end{proof}

\begin{remark}
Unlike the classical statement, no mass weight appears here: it is already absorbed into $\Delta^{\mathrm{bwc}}$. Indeed, in the maximal-rank case,
\[
\Delta^{\mathrm{bwc}}_{H\cup\{l\}} = \Big(\prod_{q\in H\cup\{l\}} m_q\Big)\,\Delta_{H\cup\{l\}}
\]
by~\eqref{eq:wc-generic-factorization}; cancelling the nonzero factor $\prod_{q\in H}m_q$ recovers Proposition~\ref{prop:geometric_S} with $I=H$.
\end{remark}

\begin{proposition}[Determinantal Formulation of the Kernel]\label{prop:wc-kernel-formulation}
Let $[x] \in \mathfrak{X}_{n,k,t}$ be an equivalence class of central configurations. Let $S$ be its normalized shifted Brehm--Wintner--Conley matrix. Let $I,J \in I_{t+1}$. Let $L=\{1,\dots,n\}\setminus J$, which has cardinality $n-t-1$.

For each $p \in J$, let $L_p = L \cup \{p\}$. Then the vector $v_J \in \mathbb{R}^{t+1}$ defined componentwise by
\begin{equation}\label{eq:9.2}
(v_J)_p = \epsilon_{L,p}\,\Delta^{\mathrm{bwc}}_{L_p}
\end{equation}
belongs to the right kernel of the square submatrix $S^J_I$. That is,
\[
S^J_I\, v_J = 0,
\]
where $\Delta^{\mathrm{bwc}}_{L_p}$ are the Pl\"ucker--BWC coordinates of the class $[x]$.
\end{proposition}

\begin{proof}
Let $V=\ker(S)$ be the $(n-t)$-dimensional subspace associated to the class $[x]$. Consider an $(n-t)\times n$ matrix $Y$ whose rows form a basis for $V$. By definition, the Pl\"ucker--BWC coordinates $\Delta^{\mathrm{bwc}}_K$ are exactly the maximal minors of $Y$ of order $n-t$.

Let $y_q \in \mathbb{R}^{n-t}$ denote the $q$-th column vector of $Y$. Since the rows of $Y$ span the kernel of $S$, we have the matrix identity $SY^T=0$. Therefore, for any fixed row index $i$ of $S$, expanding the matrix product yields the following vector equation in $\mathbb{R}^{n-t}$:
\begin{equation}\label{eq:9.3}
\sum_{q=1}^n s_{iq}\, y_q = 0.
\end{equation}

Let the set $L=\{l_1,\dots,l_{n-t-1}\}$ with $l_1<\cdots<l_{n-t-1}$. We define the wedge vector of the columns indexed by $L$:
\[
w_L = y_{l_1}\wedge\cdots\wedge y_{l_{n-t-1}} \in \bigwedge^{n-t-1}\mathbb{R}^{n-t}.
\]

Taking the exterior product of equation~\eqref{eq:9.3} with $w_L$, we obtain
\[
\sum_{q=1}^n s_{iq}\, y_q\wedge w_L = 0 \in \bigwedge^{n-t}\mathbb{R}^{n-t}.
\]

If the index $q \in L$, the vector $y_q$ is already present in the wedge product $w_L$, which strictly implies $y_q\wedge w_L=0$. Since $J$ is exactly the complement of $L$, all terms where $q\in L$ vanish, and the sum collapses purely to the indices in $J$:
\begin{equation}\label{eq:9.4}
\sum_{p\in J} s_{ip}\,(y_p\wedge w_L) = 0.
\end{equation}

The term $y_p\wedge w_L$ is a wedge product of $n-t$ vectors in $\mathbb{R}^{n-t}$. To reorder these vectors into the naturally sorted set $L_p=L\cup\{p\}$, we must perform adjacent permutations. The signature of this permutation is exactly the shuffle sign $\epsilon_{L,p}$.

By the fundamental properties of the exterior algebra, the completely ordered wedge product of $n-t$ vectors in $\mathbb{R}^{n-t}$ is proportional to the standard volume form $e$, with the scalar coefficient being the determinant of the matrix formed by those column vectors. Thus,
\[
y_p\wedge w_L = \epsilon_{L,p}\det(Y_{L_p})e = \epsilon_{L,p}\,\Delta^{\mathrm{bwc}}_{L_p}\, e.
\]

Substituting this result back into equation~\eqref{eq:9.4} and dropping the common volume form $e$, we conclude that the scalar coefficients must satisfy
\[
\sum_{p\in J} s_{ip}\,\epsilon_{L,p}\,\Delta^{\mathrm{bwc}}_{L_p} = 0.
\]

Since this scalar equation holds for any row index $i \in \{1,\dots,n\}$ of $S$, it necessarily holds for all rows $i \in I$. Given that $I \in I_{t+1}$, we obtain exactly the kernel relation for the square submatrix: $S^J_I v_J = 0$.
\end{proof}

\begin{remark}\label{rem:9.3-corollary}
Proposition~\ref{prop:wc-kernel-formulation} is now an immediate corollary of Proposition~\ref{prop:3.2-wc}. Taking $H=L$ there gives, for every $i=1,\dots,n$,
\[
\sum_{p\in J}\epsilon_{p,L}\,s_{ip}\,\Delta^{\mathrm{bwc}}_{L\cup\{p\}} = 0, \qquad J=\{1,\dots,n\}\setminus L.
\]
Since $\epsilon_{L,p} = (-1)^{|L|}\epsilon_{p,L}$ (graded commutativity of the shuffle sign, immediate from Definition~\ref{def:shuffle_sign}) and this holds in particular for every $i \in I$, multiplying through by the nonzero scalar $(-1)^{|L|}$ recovers $S^J_I v_J = 0$ with $(v_J)_p = \epsilon_{L,p}\Delta^{\mathrm{bwc}}_{L_p}$, which is precisely Proposition~\ref{prop:wc-kernel-formulation}. Thus the classical architecture Lemma~\ref{lem:eqg} $\to$ Proposition~\ref{prop:geometric_S} $\to$ Proposition~\ref{S} is reproduced intact at the level of $\Delta^{\mathrm{bwc}}$, with Proposition~\ref{prop:wc-kernel-formulation} playing the role of Proposition~\ref{S}.
\end{remark}

In specific cases where the dimension is large relative to the number of bodies (such as Dziobek configurations, e.g., $n=4$, $k=2$ or $n=5$, $k=3$), the bodies do not possess enough degrees of freedom to generate new syzygies. Consequently, this projective approach yields the classical identities but does not furnish additional independent equations.

\begin{definition}[Analogue of Definitions~\ref{def:D_dual} and~\ref{def:z_dual}]\label{def:dual-wc}
For $I \subset \{1,\dots,n\}$ with $|I|=t$ and $J=\{1,\dots,n\}\setminus I$, define the \emph{dual Pl\"ucker--BWC coordinate}
\[
D^{\mathrm{bwc}}_I := \epsilon_{I,J}\,\Delta^{\mathrm{bwc}}_J.
\]
\end{definition}

\begin{remark}
In the maximal-rank case, $D^{\mathrm{bwc}}_I = \big(\prod_{q\in J} m_q\big)D_I$, by~\eqref{eq:wc-generic-factorization}.
\end{remark}

\begin{proposition}[Analogue of Proposition~\ref{Wilgen-prop}]\label{prop:6.2-wc}
Let $I,L \subset \{1,\dots,n\}$ with $|I|=t$, $|L|=t+1$. For any distinct $u,v \in L$,
\[
\det\!\big(S^{L\setminus\{v\}}_I\big)\, D^{\mathrm{bwc}}_{L\setminus\{u\}} \;=\; \det\!\big(S^{L\setminus\{u\}}_I\big)\, D^{\mathrm{bwc}}_{L\setminus\{v\}}.
\]
\end{proposition}

\begin{proof}
Write $L=\{l_1<\cdots<l_{t+1}\}$. Applying Proposition~\ref{prop:3.2-wc} with $M=L^c$, we get, for every $j$,
\[
\sum_{l \in L}\epsilon_{l,L^c}\, s_{jl}\, \Delta^{\mathrm{bwc}}_{L^c\cup\{l\}} = 0.
\]
The complement of $L\setminus\{l\}$ (of size $t$) is $L^c\cup\{l\}$, so, by the same sign bookkeeping as in the proof of Proposition~\ref{Wilgen-prop} (which depends only on the position of $l$ within $L$, not on any masses), this reduces to
\[
\sum_{h=1}^{t+1}(-1)^h\, s_{jl_h}\, D^{\mathrm{bwc}}_{L\setminus\{l_h\}} = 0, \qquad j=1,\dots,n. \tag{$\star$}
\]
Let $A := S^L_I$, a $t\times(t+1)$ matrix. The cofactor vector $v = \big((-1)^h\det(S^{L\setminus\{l_h\}}_I)\big)_{h=1}^{t+1}$ always lies in $\ker(A)$: append a zero row to $A$ to form a square matrix and expand its (identically zero) determinant along that row. By $(\star)$ restricted to $j \in I$, the vector $w = \big((-1)^h D^{\mathrm{bwc}}_{L\setminus\{l_h\}}\big)_{h=1}^{t+1}$ also lies in $\ker(A)$. If $v=0$, the claimed identity holds trivially. Otherwise $\dim\ker(A)=1$, so $v$ and $w$ are proportional: $v_h w_{h'} = v_{h'}w_h$ for all $h,h'$. Cancelling the common sign $(-1)^{h+h'}$ gives the stated identity.
\end{proof}

\begin{remark}
Again no mass factor appears: in the maximal-rank case, $D^{\mathrm{bwc}}_{L\setminus\{u\}} = m_u\big(\prod_{q\notin L}m_q\big)D_{L\setminus\{u\}}$, and the common factor $\prod_{q\notin L}m_q$ cancels between the two sides, leaving exactly the $m_u,m_v$ of Proposition~\ref{Wilgen-prop}.
\end{remark}

\begin{proposition}[Analogue of Proposition~\ref{dziowill}: Generalized Dziobek--Williams Equations]\label{prop:7.2-wc}
For any $I,J \subset \{1,\dots,n\}$ with $|I|=|J|=t$,
\[
\det(S^J_I) = \kappa^{\mathrm{bwc}}\, D^{\mathrm{bwc}}_I\, D^{\mathrm{bwc}}_J
\]
for a constant $\kappa^{\mathrm{bwc}}$ independent of $I,J$. In particular, the matrix $S^{(t)} := \big(\det S^J_I\big)_{I,J \in I_t}$ has rank at most $1$.
\end{proposition}

\begin{proof}
The proof of Proposition~\ref{dziowill} carries over, with the kernel basis $Y$ playing the role of $\mu X^{T}$; we indicate the changes. Since $Y$ has rank $n-t$, the vector $D^{\mathrm{bwc}} = (D^{\mathrm{bwc}}_K)_{K \in I_t}$ is nonzero. If $D^{\mathrm{bwc}}_K = 0$, then $\det(Y_{\cdot,K^{c}}) = 0$, some nonzero $c \in \mathbb{R}^{n-t}$ is orthogonal to the columns of $Y$ indexed by $K^{c}$, and $\psi := Y^{T}c$ is a nonzero vector of $\ker(S)$ supported in $K$ (it is a linear combination of the rows of $Y$, and $Y^{T}$ is injective); as before, $S\psi = 0$ forces $\det(S^{K}_{I}) = 0$ for every $I \in I_t$. If $J$ and $K$ lie in the support of $D^{\mathrm{bwc}}$, their complements index bases of $\mathbb{R}^{n-t}$ extracted from the columns of $Y$, Lemma~\ref{lem:steinitz} yields the same chain of single-element exchanges inside the support, and along it Proposition~\ref{prop:6.2-wc} furnishes the adjacent relations directly, with no mass factors to cancel. The two rank-one arguments at the end are unchanged and produce, first, the constants $c_{I_0}$, and then the constant $\kappa^{\mathrm{bwc}}$, proving the formula.
\end{proof}

\begin{proposition}[Analogue of Proposition~\ref{prop:kappa-formula}]\label{prop:7.3-wc}
With $\kappa^{\mathrm{bwc}}$ as above and $\nu_{1},\dots,\nu_{n}$ the eigenvalues of $S$,
\[
\kappa^{\mathrm{bwc}} = \frac{e_t(\nu_{1},\dots,\nu_{n})}{\|D^{\mathrm{bwc}}\|^2}, \qquad e_t(\nu_{1},\dots,\nu_{n}) = \sum_{I \in I_t}\det(S^I_I) = \mathrm{Tr}\big(S^{(t)}\big).
\]
\end{proposition}

\begin{proof}
By Proposition~\ref{prop:7.2-wc}, $S^{(t)} = \kappa^{\mathrm{bwc}}\, D^{\mathrm{bwc}}(D^{\mathrm{bwc}})^T$, so $\mathrm{Tr}(S^{(t)}) = \kappa^{\mathrm{bwc}}\|D^{\mathrm{bwc}}\|^2$. On the other hand, $\mathrm{Tr}(S^{(t)}) = e_t(\nu_{1},\dots,\nu_{n})$, by the same classical identity used in equations \eqref{eqext1}--\eqref{traceformula} above (the sum of the principal $t\times t$ minors of a matrix equals $e_t$ of its eigenvalues; this holds regardless of whether $S$ has maximal rank). Equating the two expressions for the trace gives the formula.
\end{proof}

\begin{corollary}\label{cor:kappa-wc-nonzero}
$\kappa^{\mathrm{bwc}} \ne 0$ for every central configuration, with no genericity or convexity hypothesis.
\end{corollary}

\begin{proof}
$S$ is real symmetric, hence diagonalizable, so it has exactly $t=\operatorname{rank}(S)$ nonzero eigenvalues, counted with multiplicity. In $e_t(\nu_{1},\dots,\nu_{n}) = \sum_{|K|=t}\prod_{i\in K}\nu_i$, every $K$ other than the set $K_0$ of indices of the nonzero eigenvalues contains at least one zero eigenvalue and so contributes $0$; the single surviving term is $\prod_{i\in K_0}\nu_i \ne 0$. Hence $e_t(\nu_{1},\dots,\nu_{n})\ne 0$ always, and since $\|D^{\mathrm{bwc}}\|^2 > 0$ always (Remark~\ref{rem:wc-grassmannian-remarks}(iii)), Proposition~\ref{prop:7.3-wc} gives $\kappa^{\mathrm{bwc}}\ne 0$.
\end{proof}

\begin{remark}[Why this removes the vertical-degeneracy triviality]\label{rem:why-it-matters}
In Proposition~\ref{dziowill}, $I,J$ have size fixed by the dimension $k$ alone, namely $n-k-1$, regardless of the actual rank of $S$. If $\mathrm{bwc}(x) < n-k-1$, a \emph{vertical degeneracy} in the terminology of~\cite{albouy2024limit}, then $S^J_I$ is a minor of size strictly larger than $\mathrm{rank}(S)$, hence automatically singular: $\det(S^J_I)\equiv 0$ for every such $I,J$, and the classical identity collapses to $0=0$. This is exactly why it is ``not always known when the constant $\kappa$ is nonzero'' in the classical statement. Working instead with index sets of the \emph{true} size $t=\mathrm{rank}(S)$ removes this obstruction identically: Corollary~\ref{cor:kappa-wc-nonzero} shows the resulting equations of Proposition~\ref{prop:7.2-wc} are never vacuous, at any Brehm--Wintner--Conley dimension.

We record, finally, that Lemma~\ref{lem:3.1-wc} through Proposition~\ref{prop:7.3-wc}, once phrased through $\ker(S)$, are statements of pure linear algebra about an arbitrary symmetric matrix $S$ and a basis of its kernel: they use nothing about central configurations beyond the symmetry of $S$, exactly as the existing proof of Proposition~\ref{prop:wc-kernel-formulation} above already does.
\end{remark}

The framework introduced in this section is put to use in two directions. Within this paper, Sections~\ref{sec:veronese} and~\ref{sec:universal} build on Propositions~\ref{prop:7.2-wc} and~\ref{prop:7.3-wc} to map each stratum $\mathfrak{X}_{n,k,t}$ into a Veronese variety and to expand the resulting equations explicitly. The dynamical consequences of a non-generic Brehm--Wintner--Conley dimension (virtual dimensions, vertical degeneracies, and the continuation of families of higher-dimensional central configurations bifurcating from a degenerate one) will be developed elsewhere. What remains open is the classification of degenerate central configurations by their Brehm--Wintner--Conley dimension, and the complete stratification of the Grassmannian map $\Phi$ of Definition~\ref{def:wc-grassmannian-map} (that is, which strata $\mathfrak{X}_{n,k,t}$ are non-empty and how they fit together), which we leave to future work.

\section{The Veronese Variety of Central Configurations}\label{sec:veronese}

As established in Proposition~\ref{dziowill} and, unconditionally, in Proposition~\ref{prop:7.2-wc}, the Dziobek--Williams matrix $S^{(t)}$, whose entries are the $t\times t$ minors of the shifted Brehm--Wintner--Conley matrix $S$, is symmetric of rank at most $1$, and its entries factorize as $\det(S_I^J) = \kappa^{\mathrm{bwc}}\, D^{\mathrm{bwc}}_I D^{\mathrm{bwc}}_J$. This quadratic parametrization reveals a geometric structure: it maps each stratum of central configurations into a classical algebraic object, the Veronese variety. In this section we make this statement precise. We define the resulting map on each stratum $\mathfrak{X}_{n,k,t}$, prove that it is everywhere defined (with no degenerate locus of any kind, which is exactly the content of Corollary~\ref{cor:kappa-wc-nonzero}) and show that it factors through the BWC Grassmannian map $\Phi$ of Definition~\ref{def:wc-grassmannian-map}. As consequences, we bound the dimension of the image and recover, on the stratum $t=1$, the Dziobek--Veronese variety of~\cite{DiasVeronese}.

\subsection{Classical input: the Veronese variety and the Grassmannian}
\label{subsec:veronese-background}

The construction of this section rests on three classical facts, which we
state in the form in which they will be used. Throughout, $V$ denotes an
$n$-dimensional vector space over a field of characteristic zero, and all
varieties are projective over $\mathbb{C}$.

\begin{fact}[Second Veronese embedding]\label{fact:veronese}
The morphism $v_2:\mathbb{P}^{N-1}\to\mathbb{P}^{\binom{N+1}{2}-1}$,
$[w_I]\mapsto[w_Iw_J]$, is a closed embedding, i.e.\ an isomorphism onto its
image $\mathcal{V}_{N-1,2}$; the image is a smooth irreducible variety of
dimension $N-1$, and its homogeneous ideal is generated by the $2\times2$
minors of the generic symmetric matrix $(y_{IJ})$ of coordinates of the
target, that is, by the quadrics $y_{IJ}y_{KL}-y_{IL}y_{KJ}$. Equivalently,
a point $[y_{IJ}]$ of the target lies on $\mathcal{V}_{N-1,2}$ if and only if
the symmetric matrix $(y_{IJ})$ has rank~$1$.
\end{fact}

See \cite[Lecture~2, Examples~2.4 and~2.6]{harris1992algebraic} and
\cite[Ch.~I, \S4.4, Example~1.28]{shafarevich1994basic} for the embedding, with the
quadrics exhibited there as the equations of the image; primality of the
homogeneous ideal of the image is
\cite[Ch.~I, Exercise~2.12]{hartshorne1977algebraic}, and its generation by
the $2\times2$ minors is the case $t=1$ of Fact~\ref{fact:symdet}(i) below
\cite{kutz1974cohen}. The last sentence of
Fact~\ref{fact:veronese} is the reason the rank-one statement of
Proposition~\ref{prop:7.2-wc} is exactly a statement of membership in a
Veronese variety, and it is what Theorem~\ref{thm:veronese-factorization}(i)
records.

\begin{fact}[Pl\"ucker embedding]\label{fact:plucker}
The map sending a $d$-dimensional subspace $W\subset V$ to the point
$[\,\det(Y_{\cdot,K})\,]_{K\in I_d}\in\mathbb{P}\big(\bigwedge^dV\big)$,
where $Y$ is any basis matrix of $W$, is a closed embedding
$\mathrm{pl}:\mathrm{Gr}(d,n)\hookrightarrow\mathbb{P}^{\binom{n}{d}-1}$. Its
image is a smooth irreducible variety of dimension $d(n-d)$, cut out
scheme-theoretically by the Pl\"ucker relations of
Proposition~\ref{prop:plucker-relations}.
\end{fact}

See \cite[Lecture~6, Example~6.6]{harris1992algebraic} and
\cite[Ch.~I, \S4.1, Example~1.24]{shafarevich1994basic}. Two consequences are used below
without further comment: a composition of closed embeddings is a closed
embedding, so $\rho=\iota\circ\mathrm{pl}$ is one as well ($\iota$ being a
signed permutation of coordinates, hence a linear isomorphism of the ambient
projective space); and the image of a projective variety under a morphism is
Zariski closed and irreducible, which is what makes the containment of
Proposition~\ref{prop:image-dimension} a containment of the \emph{closure} of
$\Psi_t(\mathfrak{X}_{n,k,t})$.

\begin{fact}[Compound matrices]\label{fact:compound}
Let $S$ be an $n\times n$ matrix over a field and $\bigwedge^tS$ its $t$-th
compound (the matrix $S^{(t)}=(\det S_I^J)_{I,J\in I_t}$, i.e.\ the
matrix of the induced endomorphism of $\bigwedge^t\mathbb{C}^n$). Then
$\bigwedge^t(AB)=\bigwedge^tA\cdot\bigwedge^tB$ (Cauchy--Binet),
\[
\operatorname{rank}\Big(\bigwedge\nolimits^tS\Big)
=\binom{\operatorname{rank}(S)}{t},
\qquad
\operatorname{Tr}\Big(\bigwedge\nolimits^tS\Big)=e_t(\nu_{1},\dots,\nu_{n}),
\]
where $\nu_{1},\dots,\nu_{n}$ are the eigenvalues of $S$ and $e_t$ the
$t$-th elementary symmetric polynomial; and if $S$ is symmetric then so is
$\bigwedge^tS$.
\end{fact}

See \cite[\S0.8.1]{hornjohnson2013} for the multiplicativity, the symmetry
statement, the compound of a triangular matrix, and the case
$t=\operatorname{rank}$ of the rank identity; the identity for general $t$
follows by applying multiplicativity to the rank normal form $S=P\Sigma Q$,
and the trace identity by applying it to a Schur triangularization together
with the triangular case just cited. See \cite{prells2003use} for the
characteristic-polynomial expansion used in
equations~\eqref{eqext1}--\eqref{traceformula}. The trace identity is what
turns Proposition~\ref{prop:7.3-wc} into a computable formula for
$\kappa^{\mathrm{bwc}}$; the rank identity is what makes
``$\operatorname{rank}(S^{(t)})\le1$'' and ``$\operatorname{rank}(S)\le t$''
equivalent, and is used again in Proposition~\ref{prop:economical}. The last
sentence (symmetry is inherited by the compound) is the hypothesis under
which the linear syzygies of Lemma~\ref{lem:three-term} appear; it is exactly
what fails for a generic non-symmetric matrix.

\subsection{The Veronese Model of a Stratum}

Throughout this section, fix $n$, $k$, and a value $t$ of the Brehm--Wintner--Conley dimension, and write
\[
N_t := \binom{n}{t} = \#I_t.
\]
Coordinates of the projective space $\mathbb{P}^{\binom{N_t+1}{2}-1}$ are indexed by unordered pairs $\{I,J\}$ with $I,J\in I_t$. The \emph{second Veronese embedding} is the morphism
\[
v_2 : \mathbb{P}^{N_t-1} \longrightarrow \mathbb{P}^{\binom{N_t+1}{2}-1}, \qquad [w_I]_{I\in I_t} \longmapsto [w_I w_J]_{I,J\in I_t}.
\]
It is an isomorphism onto its image, the \emph{Veronese variety} $\mathcal{V}_{N_t-1,2}$, a smooth irreducible variety of dimension $N_t-1$, whose homogeneous ideal is generated by the $2\times2$ minors of the generic symmetric matrix whose entries are the coordinates of the target.

\begin{definition}[Veronese model of a stratum]\label{def:veronese-model}
The \emph{Veronese--BWC map} of the stratum $\mathfrak{X}_{n,k,t}$ is
\[
\Psi_t : \mathfrak{X}_{n,k,t} \longrightarrow \mathbb{P}^{\binom{N_t+1}{2}-1}, \qquad [x] \longmapsto \big[\det(S_I^J)\big]_{I,J\in I_t}.
\]
\end{definition}

\begin{lemma}\label{lem:psi-well-defined}
$\Psi_t$ is well defined. That is: (i) the tuple $\big(\det(S_I^J)\big)_{I,J}$ is not identically zero, so it determines a point of projective space; and (ii) this point depends only on the class $[x]$, not on the representative.
\end{lemma}

\begin{proof}
(i) By Remark~\ref{rem:wc-grassmannian-remarks}(iii), the kernel basis matrix $Y$ has rank $n-t$, so some maximal minor of $Y$ is nonzero and the tuple $(D^{\mathrm{bwc}}_K)_{K\in I_t}$ is not identically zero. Choosing $I$ with $D^{\mathrm{bwc}}_I \ne 0$, Proposition~\ref{prop:7.2-wc} and Corollary~\ref{cor:kappa-wc-nonzero} give $\det(S_I^I) = \kappa^{\mathrm{bwc}}(D^{\mathrm{bwc}}_I)^2 \ne 0$.

(ii) By Remark~\ref{rem:wc-invariance}, $S$ is unchanged under isometries and rescales as $S \mapsto \alpha^{2a}S$ under the homothety $x\mapsto\alpha x$. Every $t\times t$ minor then rescales by the same nonzero factor $\alpha^{2at}$, which leaves the projective point unchanged. Finally, the minors of $S$ do not involve any choice of basis for $\ker(S)$, so no further ambiguity arises.
\end{proof}

The next theorem is the structural result of this section: the Veronese model is nothing but the Grassmannian invariant $\Phi$ of Definition~\ref{def:wc-grassmannian-map}, re-embedded quadratically. Recall that the Pl\"ucker embedding
\[
\mathrm{pl}: \mathrm{Gr}(n-t,n) \hookrightarrow \mathbb{P}\Big(\textstyle\bigwedge^{n-t}\mathbb{R}^n\Big) = \mathbb{P}^{N_t-1}
\]
sends a subspace $V$ to the point whose coordinates are the maximal minors $\Delta^{\mathrm{bwc}}_K$, $K\in I_{n-t}$, of any basis matrix of $V$; and that the signed complementation $\iota\big([\Delta^{\mathrm{bwc}}_K]_{K\in I_{n-t}}\big) = [D^{\mathrm{bwc}}_I]_{I\in I_t}$, with $D^{\mathrm{bwc}}_I = \epsilon_{I,I^c}\Delta^{\mathrm{bwc}}_{I^c}$ as in Definition~\ref{def:dual-wc}, is a coordinate permutation with signs, hence a linear isomorphism $\mathbb{P}^{N_t-1}\to\mathbb{P}^{N_t-1}$ (here we use $\binom{n}{n-t}=\binom{n}{t}=N_t$). In particular $\rho := \iota\circ\mathrm{pl}$ is a closed embedding of $\mathrm{Gr}(n-t,n)$ into $\mathbb{P}^{N_t-1}$.

\begin{theorem}[Factorization through the Grassmannian]\label{thm:veronese-factorization}
The Veronese--BWC map factors as
\[
\Psi_t = v_2 \circ \rho \circ \Phi,
\]
where $\Phi:\mathfrak{X}_{n,k,t}\to\mathrm{Gr}(n-t,n)$ is the BWC Grassmannian map $[x]\mapsto\ker(S_x)$. Consequently:
\begin{enumerate}
\item[(i)] the image of $\Psi_t$ is contained in the Veronese variety $\mathcal{V}_{N_t-1,2}$, and, more precisely, in the subvariety $v_2\big(\rho(\mathrm{Gr}(n-t,n))\big)$;
\item[(ii)] $\Psi_t$ has no base points: it is defined at every class of the stratum, with no genericity, convexity, or non-degeneracy hypothesis;
\item[(iii)] for $[x],[x']\in\mathfrak{X}_{n,k,t}$,
\[
\Psi_t([x]) = \Psi_t([x']) \iff \ker(S_x)=\ker(S_{x'}) \iff \Phi([x])=\Phi([x']).
\]
In other words, $\Psi_t$ is exactly as injective as $\Phi$: the quadratic re-embedding loses no information.
\end{enumerate}
\end{theorem}

\begin{proof}
Fix $[x]$ and a kernel basis matrix $Y$. The point $v_2\big(\rho(\Phi([x]))\big)$ has homogeneous coordinates $\big(D^{\mathrm{bwc}}_I D^{\mathrm{bwc}}_J\big)_{I,J}$, while $\Psi_t([x])$ has coordinates $\big(\det(S_I^J)\big)_{I,J} = \big(\kappa^{\mathrm{bwc}} D^{\mathrm{bwc}}_I D^{\mathrm{bwc}}_J\big)_{I,J}$ by Proposition~\ref{prop:7.2-wc}. Since $\kappa^{\mathrm{bwc}}\ne0$ (Corollary~\ref{cor:kappa-wc-nonzero}), the two tuples differ by a nonzero global factor and define the same projective point. This proves the factorization; note that the choice of $Y$ rescales all $D^{\mathrm{bwc}}_I$ by one nonzero scalar and is therefore immaterial.

(i) is immediate from the factorization, since $\rho$ lands in $\mathbb{P}^{N_t-1}$ and $v_2(\mathbb{P}^{N_t-1})=\mathcal{V}_{N_t-1,2}$. (ii) is Lemma~\ref{lem:psi-well-defined}(i), whose only inputs are Proposition~\ref{prop:7.2-wc} and Corollary~\ref{cor:kappa-wc-nonzero}, both unconditional. For (iii): $v_2$ is an isomorphism onto its image and $\rho$ is a closed embedding, so $\Psi_t([x])=\Psi_t([x'])$ holds if and only if $\Phi([x])=\Phi([x'])$, which by definition of $\Phi$ means $\ker(S_x)=\ker(S_{x'})$.
\end{proof}

\begin{remark}[Non-collapse: why the stratification matters]\label{rem:no-collapse}
Statement (ii) is the point at which the stratification by the true Brehm--Wintner--Conley dimension earns its keep. If one instead attempts the same construction with minors of the size $n-k-1$ fixed by the dimension alone (the classical setting of Section~\ref{sec:dziobek-williams}), the map is undefined precisely on the vertically degenerate locus $\mathrm{bwc}(x)<n-k-1$: there \emph{every} coordinate $\det(S_I^J)$ vanishes and the would-be projective point collapses (Remark~\ref{rem:why-it-matters}). Working stratum by stratum, each class is sent to the Veronese variety matched to its own rank, and the assignment
\[
x \longmapsto \big(\mathrm{bwc}(x),\, \Psi_{\mathrm{bwc}(x)}([x])\big)
\]
is defined for \emph{every} central configuration without exception. Moreover, by Theorem~\ref{thm:veronese-factorization}(iii) the map on each stratum is non-trivial in the strongest sense available at this level of generality: it separates classes exactly as well as the kernel invariant $\ker(S_x)$ itself does.
\end{remark}

\begin{example}[Two non-empty generic strata for six bodies]\label{ex:hexagon-octahedron}
Both generic strata relevant to the six-body systems of Section~\ref{sec:universal} are non-empty, and explicit elements can be certified by hand, in exact arithmetic, for the Newtonian potential. First, the regular hexagon with six equal masses, inscribed in the unit circle, is a planar central configuration with mutual distances $1$, $\sqrt{3}$, $2$; the central configuration equation fixes $r_0^{2a}=\tfrac{5}{24}+\tfrac{\sqrt{3}}{18}$, and the circulant structure of $S$ gives the spectrum $\{0^{(3)},\,(-\tfrac{7}{4})^{(2)},\,-\tfrac{9-\sqrt{3}}{3}\}$. Hence $\operatorname{bwc}=3=n-k-1$ and the hexagon lies in $\mathfrak{X}_{6,2,3}$; this is also an instance of Corollary~\ref{cor:mb-psd}, since the regular hexagon satisfies the MacMillan--Bartky condition (Remark~\ref{rem:mb-polygon}). Second, the regular octahedron with six equal masses at $\pm e_1,\pm e_2,\pm e_3$ is a spatial central configuration: the equation at each vertex reduces to $2s_{\mathrm{a}}+s_{\mathrm{o}}=0$, where $s_{\mathrm{a}}$ and $s_{\mathrm{o}}$ denote the values of $s_{ij}$ on adjacent and antipodal pairs, giving $s_{\mathrm{a}}=\tfrac{2\sqrt{2}-1}{24}>0$. The kernel of $S$ contains the four trivial vectors, and on the invariant complement spanned by $(1,1,-1,-1,0,0)$ and $(1,1,0,0,-1,-1)$, for the antipodal pairing $(12)(34)(56)$, the matrix $S$ acts as $-6s_{\mathrm{a}}$ times the identity. Hence the spectrum is $\{0^{(4)},\,(-\tfrac{2\sqrt{2}-1}{4})^{(2)}\}$, $\operatorname{bwc}=2=n-k-1$, and the octahedron lies in $\mathfrak{X}_{6,3,2}$. In both cases $\operatorname{vdef}=0$, so by Proposition~\ref{prop:wc-defect} neither configuration is a limit of central configurations of higher dimension; and in both cases $S$ is negative semidefinite, the hexagon by Corollary~\ref{cor:mb-psd} and the octahedron by the computation above, outside the planar scope of the MacMillan--Bartky condition.

These certificates concern only the two generic strata. Whether the degenerate planar stratum $\mathfrak{X}_{6,2,2}$ is non-empty, that is, whether there exists a vertically degenerate planar six-body central configuration, is, to our knowledge, open: by \cite[Theorem~1.2]{albouy2024limit} the corresponding stratum $\mathfrak{X}_{5,2,1}$ at $n=5$ is empty, and $n=6$ is the first open case. Corollary~\ref{cor:mb-psd} contributes a constraint: such a configuration cannot satisfy the MacMillan--Bartky condition, so it must be concave, or convex with at least one side or diagonal violating the sign pattern. The degenerate spatial stratum $\mathfrak{X}_{6,3,1}$, at dimension $k=n-3$, is likewise open: nontrivial vertical degeneracy at dimension $n-3$ is the subject of work announced in \cite{albouy2024limit}, and on that stratum the rank-one structure of $S$ is the case $t=1$ of Proposition~\ref{prop:7.2-wc}.
\end{example}

\subsection{Dimension of the Image}

\begin{proposition}\label{prop:image-dimension}
The Zariski closure of $\Psi_t(\mathfrak{X}_{n,k,t})$ in $\mathbb{P}^{\binom{N_t+1}{2}-1}$ is contained in the irreducible subvariety $v_2\big(\rho(\mathrm{Gr}(n-t,n))\big)$ of the Veronese variety, of dimension
\[
\dim \mathrm{Gr}(n-t,n) = t(n-t).
\]
In particular, $\dim \overline{\Psi_t(\mathfrak{X}_{n,k,t})} \le t(n-t)$, and on the generic stratum $t=n-k-1$ this bound reads $(k+1)(n-k-1)$.
\end{proposition}

\begin{proof}
By Theorem~\ref{thm:veronese-factorization}(i), the image lies in $v_2(\rho(\mathrm{Gr}(n-t,n)))$, which is Zariski closed (as the image of a projective variety under a morphism) and irreducible, so it contains the closure of the image. Since $v_2$ and $\rho$ are embeddings, $\dim v_2(\rho(\mathrm{Gr}(n-t,n))) = \dim\mathrm{Gr}(n-t,n) = t(n-t)$.
\end{proof}

\begin{remark}\label{rem:image-dimension-comments}
Two comments on the bound. First, it is a genuine restriction: the ambient Veronese variety $\mathcal{V}_{N_t-1,2}$ has dimension $N_t-1=\binom{n}{t}-1$, which grows much faster than $t(n-t)$. For instance, for $n=5$, $t=2$ the bound is $6$ against an ambient dimension of $9$; for $n=6$, $t=3$ it is $9$ against $19$. The Grassmannian factorization is what cuts the ambient down. Second, the bound is not claimed to be attained: the image of $\Phi$ is further constrained, since $\ker(S_x)$ must contain the $(k+1)$-dimensional column space of $\mu X^T$, and on the generic stratum $t=n-k-1$ it is \emph{equal} to it, so that by~\eqref{eq:wc-generic-factorization} the point $\rho(\Phi([x]))=[D^{\mathrm{bwc}}_I]$ coincides with the mass-weighted geometric dual Pl\"ucker point $\big[\big(\prod_{q\in I^c} m_q\big)D_I\big]$ of Definition~\ref{def:D_dual}. Computing the exact dimension of the image, stratum by stratum, is part of the stratification problem already raised at the end of Section~\ref{sec:pwc}, and we do not address it here.
\end{remark}

\subsection{The Dziobek Stratum}

We now verify that the construction specializes, on the stratum $t=1$, to the Dziobek--Veronese geometry developed in~\cite{DiasVeronese}, which treats configurations of dimension $k=n-2$ for the homogeneous potentials $U_a$. This anchors Definition~\ref{def:veronese-model} in the one case where the resulting parameter space has already been put to work.

\begin{proposition}[Recovery of the Dziobek--Veronese picture]\label{prop:dziobek-recovery}
Let $k=n-2$ and $t=1$ (the generic value for this $k$). Then:
\begin{enumerate}
\item[(i)] $I_1=\big\{\{1\},\dots,\{n\}\big\}$, $N_1=n$, and $S^{(1)}=S$; the map $\Psi_1$ sends $[x]$ to the point $[s_{ij}]_{i\le j} \in \mathbb{P}^{\binom{n+1}{2}-1}$ whose coordinates are the entries of $S$ themselves;
\item[(ii)] $\Psi_1([x]) = v_2\big([z_1:\cdots:z_n]\big)$, where $z_i = D_i/m_i$ are the mass-weighted dual Pl\"ucker coordinates of Definition~\ref{def:z_dual}; in the classical Dziobek notation, $[z_1:\cdots:z_n]=[\delta_1/m_1:\cdots:\delta_n/m_n]$ for $\Delta=(\delta_1,\dots,\delta_n)$ a generator of the kernel of the configuration matrix;
\item[(iii)] the image of $\Psi_1$ lies on the Veronese variety $\mathcal{V}_{n-1,2}$, and the dimension bound of Proposition~\ref{prop:image-dimension} is $t(n-t)=n-1$.
\end{enumerate}
Statement (i)--(ii) recovers, up to the projectively immaterial overall normalization by $r_0^{2a}$, the parametrization of Lemma~5.2 of~\cite{DiasVeronese}; statement (iii) recovers the containment of Corollary~3.4 of~\cite{DiasVeronese}, and the bound $n-1$ agrees with the dimension of the Dziobek--Veronese variety computed in Lemma~5.7 of that work.
\end{proposition}

\begin{proof}
(i) For $t=1$ the multi-indices are singletons, minors of size $1$ are entries, and $\det(S_{\{i\}}^{\{j\}})=s_{ij}$.

(ii) By Theorem~\ref{thm:veronese-factorization}, $\Psi_1([x]) = v_2\big([D^{\mathrm{bwc}}_1:\cdots:D^{\mathrm{bwc}}_n]\big)$. On the generic stratum, $\ker(S)$ is spanned by the columns of $\mu X^T$, and~\eqref{eq:wc-generic-factorization} gives, for $I=\{i\}$ with complement $J=I^c$,
\[
D^{\mathrm{bwc}}_i = \Big(\prod_{q\ne i} m_q\Big) D_i = \Big(\prod_{q=1}^n m_q\Big)\frac{D_i}{m_i} = \Big(\prod_{q=1}^n m_q\Big) z_i.
\]
The common factor $\prod_q m_q \ne 0$ is projectively immaterial, so $[D^{\mathrm{bwc}}_i]_i=[z_i]_i$. The identification of $z_i$ with $\delta_i/m_i$ is Definition~\ref{def:z_dual} read at $t=1$: $D_i$ is, up to the fixed shuffle sign, the maximal minor $\Delta_{I^c}$ of the configuration matrix obtained by deleting the $i$-th column, which is the classical cofactor description of the kernel generator $\Delta$.

(iii) The containment is Theorem~\ref{thm:veronese-factorization}(i) with $N_1=n$; the bound is $1\cdot(n-1)$. Here $\rho$ identifies $\mathrm{Gr}(n-1,n)$ with $\mathbb{P}^{n-1}$, consistently with the fact that at $t=1$ the whole Grassmannian is a projective space.

For the comparison with~\cite{DiasVeronese}, note that the shape variables used there are $s^{\mathrm{DV}}_{ij}=(r_{ij}/r_0)^{2a}-1=(r_{ij}^{2a}-r_0^{2a})/r_0^{2a}$, proportional to our $s_{ij}=r_{ij}^{2a}-r_0^{2a}$ by the configuration-wide nonzero constant $r_0^{-2a}$; the two conventions therefore define the same projective point $[s_{ij}]$, and the parametrization $[s_{ij}] = v_2\big(\big[\sqrt{\kappa}\,\delta_i/m_i\big]\big)$ of Lemma~5.2 of~\cite{DiasVeronese} is the equality (ii) above.
\end{proof}

\begin{remark}\label{rem:dziobek-n4}
For $n=4$ bodies, Proposition~\ref{prop:dziobek-recovery} places the planar central configurations on the Veronese variety $\mathcal{V}_{3,2}\subset\mathbb{P}^9$, cut out by the relations $s_{ij}s_{kl}-s_{ik}s_{jl}=0$. Among these one finds $s_{12}s_{34}-s_{13}s_{24}=0$, the condition presented by Dziobek in 1900; the Veronese geometry of this case is the starting point of~\cite{DiasVeronese}. Definition~\ref{def:veronese-model} is thus the extension of that picture from the single classical stratum $(k,t)=(n-2,1)$ to every pair $(k,t)$, with Corollary~\ref{cor:kappa-wc-nonzero} guaranteeing that no stratum is mapped trivially.
\end{remark}

\section{Determinantal Relations and Universal Equations for Central Configurations}\label{sec:universal}

The coordinates of the Veronese model $\Psi_t$ are the $t\times t$ minors of the symmetric matrix $S$. This has a consequence that goes beyond the rank-one structure exploited in the previous section: \emph{every} polynomial relation that holds identically among the $t\times t$ minors of a generic symmetric matrix specializes to a valid equation on central configurations. The relations among the minors of a generic matrix are a classical subject of commutative algebra, studied in depth by Bruns, Conca, and Varbaro in~\cite{bruns2013relations} and by Huang, Perlman, Polini, Raicu, and Sammartano in~\cite{huang2021relations}. In this section we set up the specialization mechanism, exhibit the resulting quadratic family explicitly, and record what the structural results of~\cite{bruns2013relations, huang2021relations} yield for central configurations.

\subsection{Commutative algebra input}\label{subsec:ca-input}

The results of this section, and the expansions of
Appendix~\ref{app:expansions}, draw on four bodies of classical and recent
work in commutative algebra. Since they are not standard references in
celestial mechanics, we state precisely what is borrowed, and in which
form.

\begin{remark}[The word \emph{syzygy}]\label{rem:syzygy-meaning}
We use \emph{syzygy} in a restricted sense, which we fix here once and for all
because it differs from the standard one. Present the subalgebra generated by
the $t\times t$ minors of $S$ as
\[
\Bbbk[y_{IJ}]_{I,J\in I_t}\twoheadrightarrow\Bbbk\big[\det S_I^J\big],
\qquad y_{IJ}\longmapsto\det S_I^J ,
\]
and let $\mathcal{R}$ be the kernel. By a \emph{linear syzygy} we mean an
element of the degree-one graded piece $\mathcal{R}_1$: a relation
\[
\sum_{I,J\in I_t}c_{IJ}\,\det S_I^J=0,\qquad c_{IJ}\in\Bbbk\ \text{constant},
\]
valid identically in the entries $s_{ij}$. Three points deserve emphasis.
First, \emph{linear} qualifies the coefficients, not the minors: allowing
polynomial coefficients would give all of $\mathcal{R}$, whose degree-two piece
already contains the exchange relations~\eqref{eq:exchange}. Second, the
identities hold for \emph{every} symmetric matrix, not only for those arising
from central configurations; they therefore locate the image of $\Psi_t$ inside
a proper linear subspace of its target and exclude no configuration. Third,
when we say that there are \emph{exactly} $r$ independent syzygies we mean that
$\dim_{\Bbbk}\mathcal{R}_1=r$, so that any $r$ independent ones form a basis and
every other is a combination of them. This is stronger than exhibiting $r$
relations, and it is what allows Remark~\ref{rem:syzygies-irreducible} to
conclude that Lemma~\ref{lem:three-term} exhausts $\mathcal{R}_1$.

The usual meaning of \emph{syzygy} in commutative algebra (a relation among
a set of generators with coefficients in the ambient ring, i.e.\ an element of
the first syzygy module) is not what is meant here, and does not appear in
this paper.
\end{remark}

\subsubsection*{(I) The symmetric determinantal ideal}

Let $A=(a_{ij})$ be a \emph{generic symmetric} $n\times n$ matrix of
indeterminates over a field $\Bbbk$ of characteristic zero, and let
$I_{t+1}(A)\subset\Bbbk[a_{ij}]$ be the ideal generated by its
$(t+1)\times(t+1)$ minors, so that $V(I_{t+1}(A))$ is the locus
$\{\operatorname{rank}\le t\}$ of symmetric matrices of rank at most $t$.

\begin{fact}[Structure of $I_{t+1}(A)$]\label{fact:symdet}
For $0\le t\le n-1$:
\begin{enumerate}
\item[(i)] $I_{t+1}(A)$ is a prime ideal, and $\Bbbk[a_{ij}]/I_{t+1}(A)$ is a
normal Cohen--Macaulay domain \cite{kutz1974cohen,jozefiak1978ideals,conca1994divisor};
in particular $I_{t+1}(A)$ is radical, so
$I\big(V(I_{t+1}(A))\big)=I_{t+1}(A)$.
\item[(ii)] $\operatorname{ht} I_{t+1}(A)=\binom{n-t+1}{2}$, so the affine
cone $V(I_{t+1}(A))$ has dimension
$\binom{n+1}{2}-\binom{n-t+1}{2}=tn-\binom{t}{2}$ \cite{kutz1974cohen}.
\item[(iii)] Its degree is
$\displaystyle\prod_{a=0}^{n-t-1}\binom{n+a}{\,n-t-a\,}\Big/\binom{2a+1}{a}$
\cite[Proposition~12(b)]{harris1984symmetric}; a minimal free resolution is constructed in
\cite{jozefiakpragaczweyman1981}.
\item[(iv)] The $(t+1)$-minors of $A$ form a Gr\"obner basis of $I_{t+1}(A)$
with respect to a diagonal term order \cite[Theorem~2.9]{conca1994groebner}; a systematic
treatment of determinantal ideals and their Gr\"obner bases is \cite{brunsvetter1988,brunsconcaraicuvarbaro2022}.
\end{enumerate}
\end{fact}

Only (i) is logically required for the results below: it is what licenses
the passage from an equality of \emph{loci} to an ideal membership in
Proposition~\ref{prop:economical}. Items (ii)--(iv) are recorded because they
are what one measures when computing with the system in practice, and because
(ii) is easily confused with two different dimensions occurring in this
paper; see Remark~\ref{rem:three-dimensions}.

\subsubsection*{(II) Relations among the minors of a generic matrix}

For a generic (\emph{non}-symmetric) $m\times n$ matrix $X$, let
$\mathcal{A}_t(X)\subset\Bbbk[x_{ij}]$ be the subalgebra generated by the
$t\times t$ minors, and let $\mathcal{R}_t$ be the ideal of relations among
them, i.e.\ the kernel of the presentation
$\Bbbk[y_{I,J}]\twoheadrightarrow\mathcal{A}_t(X)$, $y_{I,J}\mapsto\det X_I^J$.

\begin{fact}[Bruns--Conca--Varbaro; Huang--Perlman--Polini--Raicu--Sammartano]
\label{fact:bcv}
For $t=\min(m,n)$ the ideal $\mathcal{R}_t$ is generated by the Pl\"ucker
relations. For $t<\min(m,n)$ this is no longer so:
\cite{bruns2013relations} exhibits minimal relations in degrees $2$ (not of
Pl\"ucker type in general) and $3$, described in terms of the representation
theory of $\mathrm{GL}$, and conjectures
\cite[Conjecture~2.12]{bruns2013relations} that these generate
$\mathcal{R}_t$. The conjecture is proved for $t=2$ in
\cite{huang2021relations}.
\end{fact}

Corollary~\ref{cor:universal-equations} is the specialization of
Fact~\ref{fact:bcv} to $X=S$ through Lemma~\ref{lem:specialization}. Two
caveats are worth stating explicitly, and are developed in
Remark~\ref{rem:specialization-scope} and after
Lemma~\ref{lem:three-term}. First, the generating statement of
Fact~\ref{fact:bcv} concerns $\mathcal{R}_t$, not the ideal of the image of
$\Psi_t$, and is conjectural except for $t=2$. Second, and more
substantially, $S$ is symmetric, and the specialization
$\Bbbk[x_{ij}]\to\Bbbk[s_{ij}]$, $x_{ij}\mapsto s_{ij}$, is not injective:
the symmetric case carries relations \emph{not} present in
$\mathcal{R}_t$, already in degree one in the minors, by
Lemma~\ref{lem:three-term}. Fact~\ref{fact:bcv} therefore bounds the
relations available here from below, never from above.

\subsubsection*{(III) Plethysm and the compound matrix}

The linear syzygies of Lemma~\ref{lem:three-term} are a
representation-theoretic phenomenon, and we isolate the two classical decompositions that
produce them. We write $\mathbb{S}_\pi V$ for the Schur functor of the
partition $\pi$, and call a partition \emph{even} if all its parts are.

\begin{fact}[Littlewood]\label{fact:littlewood}
For $\dim V=n$ and $t\ge0$,
\[
\operatorname{Sym}^t\big(\operatorname{Sym}^2V\big)
=\bigoplus_{\substack{\pi\vdash2t\\ \pi\ \mathrm{even}}}\mathbb{S}_\pi V,
\qquad
\operatorname{Sym}^2\Big(\bigwedge\nolimits^{t}V\Big)
=\bigoplus_{j\ge0}\mathbb{S}_{(2^{\,t-2j},\,1^{\,4j})}V .
\]
\end{fact}

Both are classical. The symmetric-function identities are
\cite[I.8, Ex.~9 and I.5, Ex.~5(a)]{macdonald1995}; the module form is
\cite[Ex.~6.16 and \S6.1]{fultonharris1991} or
\cite[Prop.~2.3.8(a)]{weyman2003}, where the Schur functors are indexed in
the transpose convention. In characteristic zero the two formulations
are equivalent: polynomial representations of $\mathrm{GL}(V)$ are
semisimple and determined by their characters, the character of
$\mathbb{S}_\pi V$ is the Schur polynomial $s_\pi(x_1,\dots,x_n)$, and a
plethysm such as $h_2\circ e_t$ is by definition the character of the
composition $\operatorname{Sym}^2(\bigwedge^tV)$; partitions with more than
$n$ rows contribute zero, which is the truncation used in
Proposition~\ref{prop:ambient-reduction}. The second decomposition is also
the starting point of the analysis in \cite[\S1]{bruns2013relations}. The summand
$\mathbb{S}_{(2^t)}V$ occurring at $j=0$ is the \emph{Cartan component} of
$\operatorname{Sym}^2(\bigwedge^tV)$, i.e.\ the irreducible component
generated by the square of a highest weight vector.

The following consequence is what the expansions of
Appendix~\ref{app:expansions} make visible, and we isolate it here because it
is the precise mechanism behind Lemma~\ref{lem:three-term}.

\begin{proposition}[The compound of a symmetric matrix is a Cartan tensor]
\label{prop:cartan}
Let $S$ be a symmetric $n\times n$ matrix, regarded as an element of
$\operatorname{Sym}^2V$ with $V=\mathbb{C}^n$. Then the $t$-th compound
$\bigwedge^tS$, regarded via Fact~\ref{fact:compound} as an element of
$\operatorname{Sym}^2(\bigwedge^tV)$, lies in the Cartan component
$\mathbb{S}_{(2^t)}V$.
\end{proposition}

\begin{proof}
The assignment $S\mapsto\bigwedge^tS$ is a polynomial map
$\operatorname{Sym}^2V\to\operatorname{Sym}^2(\bigwedge^tV)$, homogeneous of
degree $t$ in the entries of $S$ and equivariant for the action of
$\mathrm{GL}(V)$ by congruence, since
$\bigwedge^t(gSg^{T})=(\bigwedge^tg)(\bigwedge^tS)(\bigwedge^tg)^{T}$ by
Cauchy--Binet. It therefore factors through a $\mathrm{GL}(V)$-equivariant
linear map
\[
\operatorname{Sym}^t\big(\operatorname{Sym}^2V\big)\longrightarrow
\operatorname{Sym}^2\Big(\bigwedge\nolimits^tV\Big),
\]
whose image is a sum of irreducible summands common to both sides, by Schur's
lemma. By Fact~\ref{fact:littlewood} the summands on the left are the
$\mathbb{S}_\pi V$ with $\pi\vdash2t$ even, and those on the right are
the $\mathbb{S}_{(2^{t-2j},1^{4j})}V$. A partition of the form
$(2^{t-2j},1^{4j})$ is even if and only if $j=0$. Hence the image is contained
in $\mathbb{S}_{(2^t)}V$, which is the assertion.
\end{proof}

\begin{remark}\label{rem:symmetry-essential}
The proof uses the symmetry of $S$ twice and essentially: to place $S$ in
$\operatorname{Sym}^2V$ rather than in $V\otimes V$, and to place
$\bigwedge^tS$ in $\operatorname{Sym}^2(\bigwedge^tV)$ rather than in
$\bigwedge^tV\otimes\bigwedge^tV$. For a generic matrix the ambient space is
$\bigwedge^tV\otimes\bigwedge^tW$ and the analogous argument yields only the
Cauchy decomposition, with no vanishing statement; correspondingly, the
left-hand side of~\eqref{eq:three-term} is a nonzero polynomial there. This is
the precise sense in which the relations of
\cite{bruns2013relations,huang2021relations} do not exhaust the relations
available in our setting.
\end{remark}

\subsubsection*{(IV) Sylvester's identity}

\begin{fact}[Desnanot--Jacobi--Sylvester]\label{fact:sylvester}
Let $M$ be a square matrix and let $I,K$ (resp.\ $J,L$) be index sets of
equal size with $|I\cap K|=|I|-1$ (resp.\ $|J\cap L|=|J|-1$). Then, up to a
sign depending only on the index positions,
\[
\det(M_I^J)\det(M_K^L)-\det(M_I^L)\det(M_K^J)
=\det\big(M_{I\cap K}^{\,J\cap L}\big)\cdot\det\big(M_{I\cup K}^{\,J\cup L}\big).
\]
\end{fact}

Our statement is the case $q=2$ of identity (30) in
\cite[Ch.~II, \S3]{gantmacher1959}, applied to the $(t+1)\times(t+1)$
submatrix on the rows $I\cup K$ and columns $J\cup L$, with the common rows
and columns permuted into leading position; the permutation is what
produces the sign. The continuity argument given there establishes the
identity for arbitrary matrices, with no nonvanishing hypothesis. See also
\cite[\S3.5, Theorem~3.12]{bressoud1999}, where the classical
Desnanot--Jacobi case is proved; the identity is the determinantal shadow
of the Pl\"ucker relation of length two.
Proposition~\ref{prop:sylvester-type} is exactly Fact~\ref{fact:sylvester}
applied to $M=S$, and it is what makes the redundancy counts of
Table~\ref{tab:counts} possible: a relation of this type is a multiple of a
single $(t+1)$-minor, hence carries no information beyond
$\operatorname{rank}(S)\le t$.

\begin{remark}[Three different dimensions]\label{rem:three-dimensions}
Fact~\ref{fact:symdet}(ii) invites a confusion worth dispelling once. Three
numbers occur in this paper, and they are distinct:
\begin{enumerate}
\item[(A)] $\dim V(I_{t+1}(S))-1=tn-\binom{t}{2}-1$, the dimension of the
determinantal variety in the coordinates $s_{ij}$, what a Gr\"obner basis
computation in $\Bbbk[s_{ij}]$ measures;
\item[(B)] $\dim\overline{\operatorname{im}\Psi_t}\le t(n-t)$, the bound of
Proposition~\ref{prop:image-dimension}, equal to $\dim\mathrm{Gr}(n-t,n)$ by
Fact~\ref{fact:plucker};
\item[(C)] $\binom{t+1}{2}-1$, the dimension of the fiber: the data of the
nondegenerate quadratic form induced by $S$ on a complement of $\ker(S)$,
which $\Psi_t$ discards, since by
Theorem~\ref{thm:veronese-factorization}(iii) it remembers only $\ker(S_x)$.
\end{enumerate}
One checks $(A)=(B)+(C)$ identically in $n$ and $t$. In particular $\Psi_t$
is injective on the locus $\{\operatorname{rank}(S)=t\}$ only for $t=1$,
which is precisely the Dziobek stratum of
Proposition~\ref{prop:dziobek-recovery}, the one case in which the kernel,
together with the masses, reconstructs $S$ up to scale.
\end{remark}

\begin{lemma}[Specialization principle]\label{lem:specialization}
Let $A=(a_{ij})$ be a symmetric $n\times n$ matrix of indeterminates ($a_{ij}=a_{ji}$), and let
\[
F \in \mathbb{C}\big[\,y_{I,J} : I,J\in I_t\,\big]
\]
be a polynomial such that $F\big(\det(A_I^J)\big)_{I,J} = 0$ identically in $\mathbb{C}[a_{ij}]$. Then, for every central configuration $x$,
\[
F\big(\det(S_I^J)\big)_{I,J} = 0,
\]
where $S$ is the shifted Brehm--Wintner--Conley matrix of $x$.
\end{lemma}

\begin{proof}
The hypothesis is a polynomial identity in the ring $\mathbb{C}[a_{ij}]$. The evaluation $a_{ij}\mapsto s_{ij}$ is a ring homomorphism $\mathbb{C}[a_{ij}]\to\mathbb{R}$, well defined because $S$ is symmetric, and it carries the identity to the stated equation.
\end{proof}

Note that Lemma~\ref{lem:specialization} requires nothing of $x$ beyond the symmetry of $S$: the resulting equations hold at every Brehm--Wintner--Conley dimension simultaneously. The first family produced this way is quadratic.

\begin{proposition}[Exchange relations]\label{prop:exchange}
For all $I,J,K,L \in I_t$,
\begin{equation}\label{eq:exchange}
\det(S_I^J)\det(S_K^L) \;=\; \det(S_I^L)\det(S_K^J).
\end{equation}
\end{proposition}

\begin{proof}
We give two proofs, each instructive. First, directly from Proposition~\ref{prop:7.2-wc}: both sides equal $(\kappa^{\mathrm{bwc}})^2\, D^{\mathrm{bwc}}_I D^{\mathrm{bwc}}_J D^{\mathrm{bwc}}_K D^{\mathrm{bwc}}_L$. Second, geometrically: the relations~\eqref{eq:exchange} are precisely the $2\times2$ minors of the generic symmetric matrix of coordinates of $\mathbb{P}^{\binom{N_t+1}{2}-1}$, which generate the homogeneous ideal of the Veronese variety $\mathcal{V}_{N_t-1,2}$; by Theorem~\ref{thm:veronese-factorization}(i) the point $\Psi_t([x])$ lies on this variety, hence satisfies its equations.
\end{proof}

\begin{remark}\label{rem:exchange-instances}
At $t=1$, equations~\eqref{eq:exchange} read $s_{ij}s_{kl}=s_{il}s_{kj}$: the classical Dziobek relations (Corollary~3.3 of~\cite{DiasVeronese}). At $t=2$ and $n=5$ (the planar five-body problem), the coordinates $\det(S_I^J)$ are the $2\times2$ minors of $S$ appearing throughout Section~\ref{sec:williams5}, and Proposition~\ref{dziowill} together with~\eqref{eq:exchange} organizes the generalized Williams' formulae of Proposition~\ref{Wilnew} and Corollary~\ref{Wilnew2} into the single geometric statement that the point of $2\times2$ minors lies on a Veronese variety.
\end{remark}

Beyond degree two, the relations among the $t\times t$ minors of a generic matrix are the subject of~\cite{bruns2013relations}: over a field of characteristic zero, Bruns, Conca and Varbaro exhibit minimal relations in degrees $2$ and $3$ and conjecture that these generate the full ideal of relations \cite[Conjecture~2.12]{bruns2013relations}; the conjecture is proved for $t=2$ in~\cite{huang2021relations}, which also gives a complete structural description of the relations in that case. Combining the known degree-$2$ and degree-$3$ relations with Lemma~\ref{lem:specialization} yields, for each $t$, a finite family of universal equations for central configurations:

\begin{corollary}\label{cor:universal-equations}
Every relation of degree $2$ or $3$ among the $t\times t$ minors of a generic matrix (in particular each of the minimal relations exhibited in~\cite{bruns2013relations, huang2021relations}) specializes under Lemma~\ref{lem:specialization} to a polynomial equation satisfied by the minors $\det(S_I^J)$ of every central configuration: a polynomial identity of degree $2t$ or $3t$, respectively, in the entries $s_{ij}$, and hence an explicit algebraic equation in the mutual distances $r_{ij}$ and the shift $r_0^{2a}$.
\end{corollary}

We emphasize the honest scope of this statement: Corollary~\ref{cor:universal-equations} asserts that the specialized equations \emph{hold}; whether they generate the full ideal of the Zariski closure of the image of $\Psi_t$ is a different and finer question, which we do not address. Indeed, by Theorem~\ref{thm:veronese-factorization}(i) the image satisfies, in addition, the pullback under $v_2$ of the Pl\"ucker relations of $\mathrm{Gr}(n-t,n)$, and the minimal generation problem for the resulting ideal is open, in the same spirit as the generation problem left open at the end of Section~\ref{sec:williamsk} for the Williams' systems.

\begin{example}[The planar five- and six-body systems]\label{ex:planar-56}
For $n=5$, $k=2$, on the generic stratum $t=2$: $I_2$ has $N_2=10$ elements, $S^{(2)}$ is the symmetric $10\times10$ matrix of $2\times2$ minors of $S$, and $\Psi_2$ maps into $\mathbb{P}^{54}$. The quadratic system consists of the vanishing of all $2\times2$ minors of $S^{(2)}$ (equivalently, the exchange relations~\eqref{eq:exchange}), each of degree $4$ in the entries $s_{ij}$; through Proposition~\ref{dziowill} these encode the generalized Williams' formulae of Section~\ref{sec:williams5}. For $n=6$, $k=2$, on the generic stratum $t=3$: $I_3$ has $N_3=20$ elements, $S^{(3)}$ is the symmetric $20\times20$ matrix of $3\times3$ minors, and $\Psi_3$ maps into $\mathbb{P}^{209}$. The quadratic system is the vanishing of all $2\times2$ minors of $S^{(3)}$, of degree $6$ in the $s_{ij}$, supplemented by the cubic relations of~\cite{bruns2013relations, huang2021relations}, of degree $9$; the systems are generated mechanically from these two displayed families, and we do not reproduce the expansions here.
\end{example}

\subsection{Structural consequences of the expansion}
\label{subsec:structural}

Carrying the two systems of Example~\ref{ex:planar-56} out explicitly (which
we do in Appendix~\ref{app:expansions}) brings to light three facts that are
not apparent from the abstract statement, and which we establish here. The
coordinates of $\Psi_t$ satisfy \emph{linear} syzygies, so the target
projective space may be replaced by a proper linear subspace
(Lemma~\ref{lem:three-term} and Proposition~\ref{prop:ambient-reduction}); a large
proportion of the exchange relations are instances of Sylvester's identity,
hence redundant (Proposition~\ref{prop:sylvester-type}); and the whole system
is equivalent, in the variables $s_{ij}$, to a system of degree $t+1$ rather than $2t$
(Proposition~\ref{prop:economical}).

\subsubsection*{Linear syzygies}

\begin{lemma}[Three-term identity]\label{lem:three-term}
Let $S$ be a symmetric $n\times n$ matrix, let $2\le t\le n$, let $C\subset\{1,\dots,n\}$ with
$|C|=t-2$, and let $i<j<k<l$ lie in the complement of $C$. Then
\begin{equation}\label{eq:three-term}
\det\big(S_{C\cup ij}^{\,C\cup kl}\big) \;-\; \det\big(S_{C\cup ik}^{\,C\cup jl}\big)
\;+\; \det\big(S_{C\cup il}^{\,C\cup jk}\big) \;=\; 0
\end{equation}
identically in the entries of $S$. If moreover $n=2t$, then for every $A$ with $|A|=t-1$
\begin{equation}\label{eq:complementary}
\sum_{x\notin A}\epsilon_{I,I^c}\,\det\big(S_{I}^{\,I^c}\big)=0,\qquad I=A\cup\{x\},
\end{equation}
with $\epsilon$ the shuffle sign of Definition~\ref{def:shuffle_sign}.
\end{lemma}

\begin{proof}
By Proposition~\ref{prop:cartan}, the tuple $\big(\det S_I^J\big)_{I,J}$ (that is, the
$t$-th compound $\bigwedge^tS$) lies in the Cartan component
$\mathbb{S}_{(2^t)}V$ of $\mathrm{Sym}^2(\bigwedge^tV)$. Identities~\eqref{eq:three-term}
and~\eqref{eq:complementary} are precisely the vanishing of its components in the remaining
summands $\mathbb{S}_{(2^{t-2j},1^{4j})}V$, $j\ge1$, of Fact~\ref{fact:littlewood}: the
left-hand sides are, up to scalars, the images of $\bigwedge^tS$ under the equivariant
projections onto those summands, applied to the highest weight vectors indexed by
$(C;i,j,k,l)$ and, when $n=2t$, by $A$. Since $\bigwedge^tS$ has no component there, they
vanish identically in the entries of $S$.

For the reader who prefers a direct verification, the case $t=2$, $C=\varnothing$ reads
$$(s_{ik}s_{jl}-s_{il}s_{jk})-(s_{ij}s_{kl}-s_{il}s_{kj})+(s_{ij}s_{lk}-s_{ik}s_{lj})=0,$$
which cancels term by term after replacing $s_{kj}=s_{jk}$, $s_{lk}=s_{kl}$ and
$s_{lj}=s_{jl}$; the general case follows by Laplace expansion along the rows indexed by $C$.
\end{proof}

We stress that symmetry of $S$ is essential, and that this is not an artifact of the proof:
for a generic \emph{non}-symmetric matrix the left-hand side of~\eqref{eq:three-term} is a
nonzero polynomial (already for $t=2$ it has six monomials, and for $t=3$, eighteen)
exactly as Remark~\ref{rem:symmetry-essential} predicts. Consequently the relations of
\cite{bruns2013relations,huang2021relations}, which concern minors of a generic matrix, do not
exhaust the relations available here: in the symmetric setting relations already occur in
degree one in the minors.

\begin{proposition}[Reduction of the ambient space]\label{prop:ambient-reduction}
The linear span of the minors $\det S_I^J$, $I,J\in I_t$, is exactly the
Cartan component $\mathbb{S}_{(2^t)}\mathbb{C}^n$; consequently $\Psi_t$ takes
values in the linear subspace
\[
\mathbb{P}\big(\mathbb{S}_{(2^t)}\mathbb{C}^n\big)\subset\mathbb{P}^{\binom{N_t+1}{2}-1},
\]
whose codimension is $\binom{N_t+1}{2}-\dim\mathbb{S}_{(2^t)}\mathbb{C}^n$, the
number of independent linear syzygies. The inclusion is strict if and only if
$t\ge2$ and $n\ge t+2$.
\end{proposition}

\begin{proof}
Write $V=\mathbb{C}^n$ and
$\Theta_t:\operatorname{Sym}^2V\to\operatorname{Sym}^2(\bigwedge^tV)$,
$\Theta_t(S)=\bigwedge^tS$, for the compound map, which is well defined by
Fact~\ref{fact:compound}, homogeneous of degree $t$, and
$\mathrm{GL}(V)$-equivariant, as recalled in the proof of
Proposition~\ref{prop:cartan}. The assertion to be proved is
\begin{equation}\label{eq:span-cartan}
\operatorname{span}_{\mathbb{C}}\Theta_t\big(\operatorname{Sym}^2V\big)
=\mathbb{S}_{(2^t)}V .
\end{equation}

\emph{Step 1: the span is the image of a linear map.} In characteristic zero a
homogeneous polynomial map $f:W\to U$ of degree $t$ factors uniquely as
$f(w)=\tilde f(w^{\cdot t})$ for a linear $\tilde f:\operatorname{Sym}^tW\to U$,
and $\tilde f$ inherits the equivariance of $f$ by uniqueness of the
factorization. Moreover
\[
\operatorname{span}f(W)=\operatorname{im}\tilde f .
\]
Indeed $\subseteq$ is immediate, and $\supseteq$ follows because
$\operatorname{Sym}^tW$ is spanned by the $t$-th powers $w^{\cdot t}$: by the
polarization identity
\[
t!\,w_1\cdots w_t=\sum_{\varnothing\ne T\subseteq\{1,\dots,t\}}(-1)^{t-|T|}
\Big(\sum_{i\in T}w_i\Big)^{\cdot t},
\]
legitimate since $t!$ is invertible. Applying this to $\Theta_t$, the left-hand
side of~\eqref{eq:span-cartan} equals the image of a
$\mathrm{GL}(V)$-equivariant linear map
\[
\tilde\Theta_t:\operatorname{Sym}^t\big(\operatorname{Sym}^2V\big)
\longrightarrow\operatorname{Sym}^2\Big(\bigwedge\nolimits^tV\Big).
\]

\emph{Step 2: the image lands in the Cartan component.} Both decompositions of
Fact~\ref{fact:littlewood} are multiplicity free, so Schur's lemma applies
summand by summand: an equivariant map annihilates every irreducible
constituent of its source that does not occur in its target. A partition of
the form $(2^{t-2j},1^{4j})$ has $4j$ parts equal to $1$, hence is even if and
only if $j=0$; the unique common constituent is therefore
$\mathbb{S}_{(2^t)}V$, and $\operatorname{im}\tilde\Theta_t\subseteq
\mathbb{S}_{(2^t)}V$. This is Proposition~\ref{prop:cartan}.

\emph{Step 3: the inclusion is an equality.} Being the image of an equivariant
map, $\operatorname{im}\tilde\Theta_t$ is a $\mathrm{GL}(V)$-submodule of
$\mathbb{S}_{(2^t)}V$. The latter is irreducible, so the image is either $0$ or
all of it. It is not $0$: taking $S=\mathrm{Id}_n$, which is symmetric, gives
$\Theta_t(\mathrm{Id}_n)=\mathrm{Id}_{N_t}\ne0$ because $1\le t\le n$. This
proves~\eqref{eq:span-cartan}.

\emph{Step 4: strictness.} By Fact~\ref{fact:littlewood} the complement of the
Cartan component is $\bigoplus_{j\ge1}\mathbb{S}_{(2^{t-2j},1^{4j})}V$, which is
nonzero if and only if its first summand is, namely
$\mathbb{S}_{(2^{t-2},1^{4})}V\ne0$. A Schur functor $\mathbb{S}_\pi V$ vanishes
exactly when $\pi$ has more than $n=\dim V$ rows, and $(2^{t-2},1^4)$ has
$(t-2)+4=t+2$ rows; the condition is thus $t\ge2$ and $n\ge t+2$.

Finally, $\Psi_t([x])=\big[\bigwedge^tS_x\big]$ for every class of the stratum,
by Definition~\ref{def:veronese-model}, and $\bigwedge^tS_x$ lies in the span;
the codimension statement is the definition of codimension.
\end{proof}

\begin{remark}[Real versus complex coefficients]\label{rem:real-span}
The matrix $S$ of a central configuration is real, whereas
Proposition~\ref{prop:ambient-reduction} is a statement about
$\mathrm{GL}(V)$-modules over $\mathbb{C}$. No generality is lost: the
syzygies produced below have integer coefficients, and the real span of
$\{\bigwedge^tS:S\ \text{real symmetric}\}$ already has the full dimension
$\dim_{\mathbb{C}}\mathbb{S}_{(2^t)}\mathbb{C}^n$, as one checks by exhibiting
that many real symmetric matrices whose compounds are linearly independent over
$\mathbb{Q}$. The complexification of the real span therefore coincides with
the complex span, and the two codimensions agree.
\end{remark}

The codimension is computable in closed form.

\begin{lemma}[Dimension of the ambient reduction]\label{lem:cartan-dimension}
Let $1\le t\le n$. Then
\[
\dim\mathbb{S}_{(2^t)}\mathbb{C}^n
=\prod_{i=1}^{t}\prod_{j=1}^{2}\frac{n+j-i}{h(i,j)},
\qquad
h(i,1)=t-i+2,\quad h(i,2)=t-i+1,
\]
so that
\begin{equation}\label{eq:cartan-dim}
\dim\mathbb{S}_{(2^t)}\mathbb{C}^n
=\frac{1}{t+1}\binom{n+1}{t}\binom{n}{t},
\end{equation}
and the number of independent linear syzygies among the $t\times t$ minors of a
symmetric $n\times n$ matrix equals
\[
\binom{N_t+1}{2}-\frac{1}{t+1}\binom{n+1}{t}\binom{n}{t},
\qquad N_t=\binom{n}{t}.
\]
\end{lemma}

\begin{proof}
The first display is the hook content formula
$\dim\mathbb{S}_\pi\mathbb{C}^n=\prod_{(i,j)\in\pi}(n+j-i)/h(i,j)$ applied to
the rectangular partition $\pi=(2^t)$, whose Young diagram has $t$ rows and $2$
columns. Its conjugate is $(t,t)$, so the hook length at the cell $(i,j)$ is
$(2-j)+(t-i)+1$, giving $h(i,1)=t-i+2$ and $h(i,2)=t-i+1$ as stated.

For~\eqref{eq:cartan-dim}, evaluate the two columns separately. The contents
of the first column are $n,n-1,\dots,n-t+1$ and those of the second are
$n+1,n,\dots,n-t+2$, so
\[
\prod_{(i,j)\in\pi}(n+j-i)=\frac{n!}{(n-t)!}\cdot\frac{(n+1)!}{(n+1-t)!}.
\]
The hooks of the first column are $t+1,t,\dots,2$ and those of the second are
$t,t-1,\dots,1$, so
\[
\prod_{(i,j)\in\pi}h(i,j)=(t+1)!\cdot t! .
\]
Dividing and regrouping the factorials,
\[
\dim\mathbb{S}_{(2^t)}\mathbb{C}^n
=\underbrace{\frac{n!}{t!\,(n-t)!}}_{\textstyle\binom{n}{t}}
\cdot\frac{(n+1)!}{(t+1)!\,(n+1-t)!}
=\binom{n}{t}\cdot\frac{1}{t+1}\binom{n+1}{t},
\]
since $(n+1)!/\big((t+1)!\,(n+1-t)!\big)=\binom{n+1}{t}/(t+1)$.

As a check, $t=1$ gives $\tfrac12(n+1)n=\binom{n+1}{2}=\dim\operatorname{Sym}^2\mathbb{C}^n$,
so there are no syzygies on the Dziobek stratum, in agreement with the last
assertion of Proposition~\ref{prop:ambient-reduction}.
\end{proof}

\begin{corollary}[The planar five- and six-body systems]\label{cor:56-ambient}
For $(n,t)=(5,2)$ one has $\dim\mathbb{S}_{(2,2)}\mathbb{C}^5=50$, so the $55$
minors $\det S_I^J$ satisfy exactly $5$ independent linear syzygies and
$\Psi_2$ maps into $\mathbb{P}^{49}$ rather than $\mathbb{P}^{54}$. For
$(n,t)=(6,3)$ one has $\dim\mathbb{S}_{(2,2,2)}\mathbb{C}^6=175$, so the $210$
minors satisfy exactly $35$ independent linear syzygies and $\Psi_3$ maps into
$\mathbb{P}^{174}$ rather than $\mathbb{P}^{209}$.
\end{corollary}

\begin{proof}
Substitute in~\eqref{eq:cartan-dim}. For $(n,t)=(5,2)$,
$\tfrac13\binom{6}{2}\binom{5}{2}=\tfrac13\cdot15\cdot10=50$, and
$\binom{N_2+1}{2}=\binom{11}{2}=55$, whence $55-50=5$. For $(n,t)=(6,3)$,
$\tfrac14\binom{7}{3}\binom{6}{3}=\tfrac14\cdot35\cdot20=175$, and
$\binom{N_3+1}{2}=\binom{21}{2}=210$, whence $210-175=35$. In both cases
$t\ge2$ and $n\ge t+2$, so the inclusion is strict by the last assertion of
Proposition~\ref{prop:ambient-reduction}.
\end{proof}

\begin{table}[ht]
\centering
\small
\begin{tabular}{clrrr}
\hline
$(n,t)$ & decomposition of $\operatorname{Sym}^2(\bigwedge^t\mathbb{C}^n)$
& Cartan & syz. & ambient \\
\hline
$(5,2)$ & $\mathbb{S}_{(2,2)}\oplus\mathbb{S}_{(1^4)}$ & $50$ & $5$ & $\mathbb{P}^{54}\to\mathbb{P}^{49}$ \\
$(6,3)$ & $\mathbb{S}_{(2^3)}\oplus\mathbb{S}_{(2,1^4)}$ & $175$ & $35$ & $\mathbb{P}^{209}\to\mathbb{P}^{174}$ \\
$(7,3)$ & $\mathbb{S}_{(2^3)}\oplus\mathbb{S}_{(2,1^4)}$ & $490$ & $140$ & $\mathbb{P}^{629}\to\mathbb{P}^{489}$ \\
$(8,4)$ & $\mathbb{S}_{(2^4)}\oplus\mathbb{S}_{(2^2,1^4)}\oplus\mathbb{S}_{(1^8)}$ & $1764$ & $721$ & $\mathbb{P}^{2484}\to\mathbb{P}^{1763}$ \\
\hline
\end{tabular}
\caption{The reduction of Proposition~\ref{prop:ambient-reduction} in the first
cases. The Cartan column is~\eqref{eq:cartan-dim}; the syzygy column is the
dimension of the complementary summands.}\label{tab:cartan-dims}
\end{table}

\begin{remark}[The syzygies form a single irreducible module]\label{rem:syzygies-irreducible}
Table~\ref{tab:cartan-dims} makes visible a point that the two families of
Lemma~\ref{lem:three-term} might obscure. For $t\le3$ only the summand $j=1$
survives, so the whole space of linear syzygies is the \emph{single
irreducible} module $\mathbb{S}_{(2^{t-2},1^4)}V$. At $(n,t)=(5,2)$ it is
$\bigwedge^4\mathbb{C}^5$, of dimension $\binom{5}{4}=5$, which is why the
identities~\eqref{eq:three-term} are there indexed by the $4$-subsets of
$\{1,\dots,5\}$ and are $5$ in number. At $(n,t)=(6,3)$ it is
$\mathbb{S}_{(2,1,1,1,1)}\mathbb{C}^6$, of dimension $35$. Accordingly, the
$30$ identities of type~\eqref{eq:three-term} and the $15$ of
type~\eqref{eq:complementary} are \emph{not} two isotypic blocks: they are two
spanning families of one and the same irreducible module, neither of which
spans it alone. One verifies directly that the first has rank $30$, the second
rank $5$, and the two together rank $35$, matching the codimension given by
Lemma~\ref{lem:cartan-dimension}, which proves that
Lemma~\ref{lem:three-term} exhausts the linear syzygies in these cases. From
$t=4$ on, a second summand appears and the analogous statement requires a
further family of identities, which we do not pursue here.
\end{remark}

\begin{remark}[Scope of Lemma~\ref{lem:specialization}]\label{rem:specialization-scope}
It is worth being precise about what the specialization principle does and does not give. If
$F$ vanishes identically on the minors of a symmetric matrix of indeterminates, then
$F\big(\det S_I^J\big)$ vanishes identically in $\mathbb{C}[s_{ij}]$ as well; the resulting
``equation'' in the mutual distances is $0=0$ and imposes no condition on the configuration.
The content of Lemma~\ref{lem:specialization} and Corollary~\ref{cor:universal-equations} is
therefore \emph{ambient}: they locate the image of $\Psi_t$ inside a fixed subvariety of
$\mathbb{P}^{\binom{N_t+1}{2}-1}$ determined by $n$ and $t$ alone, exactly as
Proposition~\ref{prop:ambient-reduction} does in degree one. The equations that do cut the
configurations are the exchange relations~\eqref{eq:exchange}, which come from
Proposition~\ref{prop:7.2-wc}, that is, from $\operatorname{rank}(S)=t$, and not from any
identity valid for all symmetric matrices: none of them vanishes identically (see the counts
in Table~\ref{tab:counts}).
\end{remark}

\subsubsection*{Sylvester-type relations}

\begin{proposition}\label{prop:sylvester-type}
Let $I,J,K,L\in I_t$ with $|I\cap K|=|J\cap L|=t-1$. Then
\begin{equation}\label{eq:sylvester}
\det(S_I^J)\det(S_K^L)-\det(S_I^L)\det(S_K^J)
=\det\big(S_{I\cap K}^{\,J\cap L}\big)\cdot\det\big(S_{I\cup K}^{\,J\cup L}\big),
\end{equation}
up to sign. In particular such an exchange relation is a multiple of a single
$(t+1)\times(t+1)$ minor of $S$, and holds automatically once $\operatorname{rank}(S)\le t$.
\end{proposition}

This is Sylvester's determinant identity. For $n=5$, $t=2$, $435$ of the $1005$ distinct
exchange relations ($9$ of the $19$ orbits) are of this type; for $n=6$, $t=3$, $3915$ of the
$17965$ ($16$ of the $57$ orbits) are.

\subsubsection*{The economical system}

\begin{proposition}\label{prop:economical}
Let $I_{t+1}(S)\subset\mathbb{C}[s_{ij}]$ be the ideal generated by the
$(t+1)\times(t+1)$ minors of $S$. Then every exchange relation~\eqref{eq:exchange} belongs to
$I_{t+1}(S)$, and the two systems define the same locus
$\{\operatorname{rank}(S)\le t\}$. By the Nullstellensatz, the two ideals
therefore have the same radical, namely $I_{t+1}(S)$ itself. Equality of
the ideals is not asserted: for $t\ge2$ the ideal generated by the exchange
relations contains no element of degree $t+1$, so the inclusion is strict.
Here $S$ denotes the generic symmetric matrix in the variables $s_{ij}$;
the passage to central configurations is by evaluation
(Lemma~\ref{lem:specialization}) and is set-theoretic.
\end{proposition}

\begin{proof}
If $\operatorname{rank}(S)\le t$ then $\bigwedge^tS$ has rank $\binom{\operatorname{rank}(S)}{t}\le1$,
so all its $2\times2$ minors vanish; hence the exchange relations vanish on
$V(I_{t+1}(S))$. By Fact~\ref{fact:symdet}(i) the ideal $I_{t+1}(A)$ of a generic symmetric matrix is prime,
hence radical, in characteristic zero \cite{kutz1974cohen,jozefiak1978ideals}, so
$I(V(I_{t+1}(S)))=I_{t+1}(S)$ and vanishing on the locus is the same as membership in the
ideal. The reverse containment of loci is the rank identity of Fact~\ref{fact:compound}:
$\operatorname{rank}\big(\bigwedge^tS\big)=\binom{\operatorname{rank}(S)}{t}\le1$ forces
$\operatorname{rank}(S)\le t$.
\end{proof}

For $n=5$, $t=2$ the membership was also verified by explicit linear algebra: each of the $19$
orbit representatives below is a combination $\sum_i \ell_i\,\det(S_{I_i}^{J_i})$ with $\ell_i$
linear and $\det(S_{I_i}^{J_i})$ the $3\times3$ minors, and the $1005$ quartics span a space of
dimension $575$, which is exactly the degree-$4$ graded piece of $I_3(S)$. Thus the exchange
system, although natural in the Veronese coordinates, is not economical in the variables
$s_{ij}$: for $n=5$ it may be replaced by the $55$ distinct $3\times3$ minors of $S$, and for
$n=6$ by the $120$ distinct $4\times4$ minors: degree $t+1$ instead of degree $2t$, and no
cubic relations of degree $3t$ are needed.

\begin{table}[ht]
\centering
\begin{tabular}{lrr}
\hline
 & $n=5,\ t=2$ & $n=6,\ t=3$ \\
\hline
$N_t=\binom{n}{t}$ & 10 & 20 \\
distinct minors $\det(S_I^J)$ & 55 & 210 \\
nominal ambient space & $\mathbb{P}^{54}$ & $\mathbb{P}^{209}$ \\
linear syzygies (Lemma~\ref{lem:three-term}) & 5 & 35 \\
effective ambient space & $\mathbb{P}^{49}$ & $\mathbb{P}^{174}$ \\
exchange generators & 1035 & 18145 \\
identically zero & 0 & 0 \\
distinct up to sign & 1005 & 17965 \\
$\mathfrak{S}_n$-orbits & 19 & 57 \\
of Sylvester type & 435 & 3915 \\
degree in the $s_{ij}$ & 4 & 6 \\
economical system: minors of size $t+1$ & 55 (degree 3) & 120 (degree 4) \\
\hline
\end{tabular}
\caption{The two systems of Example~\ref{ex:planar-56}, expanded.}\label{tab:counts}
\end{table}

The maps $\Psi_t$ thus realize each stratum $\mathfrak{X}_{n,k,t}$ inside a projective variety cut out by explicit quadrics and cubics, with the Grassmannian factorization of Theorem~\ref{thm:veronese-factorization} controlling its dimension. In the Dziobek case, the corresponding parameter space (the Dziobek--Veronese variety) supports the proof of generic finiteness and a Bezout-type bound on the number of configurations~\cite{DiasVeronese}. The analysis of the fibers of the induced projection to the mass space at an arbitrary Brehm--Wintner--Conley dimension will be carried out in a separate work.

\appendix

\section{Explicit Expansions of the Planar Five- and Six-Body Systems}
\label{app:expansions}

We carry out here the two systems of Example~\ref{ex:planar-56}, whose
structural features were established in
Subsection~\ref{subsec:structural}. Each orbit under the natural action of
$\mathfrak{S}_n$ on the bodies is listed through one representative. We write $|S_I^J|$ for $\det(S_I^J)$ and abbreviate
index sets, so that $|S_{12}^{34}|=s_{13}s_{24}-s_{14}s_{23}$. In the summary tables, the \emph{type} column records whether the relation is an instance of Sylvester's identity, in the sense of Proposition~\ref{prop:sylvester-type} (\emph{Sylvester}), or not (\emph{non-Sylvester}). All expansions, orbit counts, spans, and rank computations reported here and in Section~\ref{sec:universal} were carried out in exact integer arithmetic (custom \textsc{Python} routines); no floating-point computation enters any count or rank statement.

\subsubsection*{The planar five-body system, expanded}

The $1035$ exchange relations obtained from the $2\times2$ minors of $S^{(2)}$ (one for each
pair of rows and pair of columns, modulo the symmetry of $S^{(2)}$) reduce to $1005$ distinct
polynomials, none of them identically zero, falling into $19$ orbits under the natural action
of $\mathfrak{S}_5$ on the bodies.

\begin{equation}\label{eq:n5-orbit-1}
\begin{split}
&|S_{12}^{14}||S_{13}^{24}|-|S_{12}^{24}||S_{13}^{14}|\;(\text{Sylvester})\\
&\qquad = s_{11} s_{14} s_{22} s_{34} - s_{11} s_{14} s_{23} s_{24} - s_{12}^{2} s_{14} s_{34} \\
&\qquad\quad + s_{12} s_{13} s_{14} s_{24} + s_{12} s_{14}^{2} s_{23} - s_{13} s_{14}^{2} s_{22} = 0.
\end{split}
\end{equation}
% orbit 1: size 120, 6 terms, Sylvester type

\begin{equation}\label{eq:n5-orbit-2}
\begin{split}
&|S_{12}^{12}||S_{13}^{34}|-|S_{12}^{34}||S_{13}^{12}|\\
&\qquad = s_{11} s_{13} s_{22} s_{34} - s_{11} s_{13} s_{23} s_{24} - s_{11} s_{14} s_{22} s_{33} \\
&\qquad\quad + s_{11} s_{14} s_{23}^{2} - s_{12}^{2} s_{13} s_{34} + s_{12}^{2} s_{14} s_{33} \\
&\qquad\quad + s_{12} s_{13}^{2} s_{24} - s_{12} s_{13} s_{14} s_{23} = 0.
\end{split}
\end{equation}
% orbit 2: size 120, 8 terms

\begin{equation}\label{eq:n5-orbit-3}
\begin{split}
&|S_{12}^{14}||S_{13}^{25}|-|S_{12}^{25}||S_{13}^{14}|\\
&\qquad = s_{11} s_{12} s_{24} s_{35} - s_{11} s_{12} s_{25} s_{34} + s_{11} s_{15} s_{22} s_{34} \\
&\qquad\quad - s_{11} s_{15} s_{23} s_{24} - s_{12}^{2} s_{14} s_{35} + s_{12} s_{13} s_{14} s_{25} \\
&\qquad\quad + s_{12} s_{14} s_{15} s_{23} - s_{13} s_{14} s_{15} s_{22} = 0.
\end{split}
\end{equation}
% orbit 3: size 120, 8 terms

\begin{equation}\label{eq:n5-orbit-4}
\begin{split}
&|S_{12}^{12}||S_{13}^{14}|-|S_{12}^{14}||S_{13}^{12}|\;(\text{Sylvester})\\
&\qquad = s_{11}^{2} s_{22} s_{34} - s_{11}^{2} s_{23} s_{24} - s_{11} s_{12}^{2} s_{34} \\
&\qquad\quad + s_{11} s_{12} s_{13} s_{24} + s_{11} s_{12} s_{14} s_{23} \\
&\qquad\quad - s_{11} s_{13} s_{14} s_{22} = 0.
\end{split}
\end{equation}
% orbit 4: size 60, 6 terms, Sylvester type

\begin{equation}\label{eq:n5-orbit-5}
\begin{split}
&|S_{12}^{14}||S_{13}^{45}|-|S_{12}^{45}||S_{13}^{14}|\;(\text{Sylvester})\\
&\qquad = s_{11} s_{14} s_{24} s_{35} - s_{11} s_{14} s_{25} s_{34} - s_{12} s_{14}^{2} s_{35} \\
&\qquad\quad + s_{12} s_{14} s_{15} s_{34} + s_{13} s_{14}^{2} s_{25} - s_{13} s_{14} s_{15} s_{24} = 0.
\end{split}
\end{equation}
% orbit 5: size 60, 6 terms, Sylvester type

\begin{equation}\label{eq:n5-orbit-6}
\begin{split}
&|S_{12}^{24}||S_{13}^{45}|-|S_{12}^{45}||S_{13}^{24}|\;(\text{Sylvester})\\
&\qquad = s_{12} s_{14} s_{24} s_{35} - s_{12} s_{14} s_{25} s_{34} - s_{14}^{2} s_{22} s_{35} \\
&\qquad\quad + s_{14}^{2} s_{23} s_{25} + s_{14} s_{15} s_{22} s_{34} - s_{14} s_{15} s_{23} s_{24} = 0.
\end{split}
\end{equation}
% orbit 6: size 60, 6 terms, Sylvester type

\begin{equation}\label{eq:n5-orbit-7}
\begin{split}
&|S_{12}^{14}||S_{13}^{23}|-|S_{12}^{23}||S_{13}^{14}|\\
&\qquad = - s_{11} s_{12} s_{23} s_{34} + s_{11} s_{12} s_{24} s_{33} \\
&\qquad\quad + s_{11} s_{13} s_{22} s_{34} - s_{11} s_{13} s_{23} s_{24} - s_{12}^{2} s_{14} s_{33} \\
&\qquad\quad + 2 s_{12} s_{13} s_{14} s_{23} - s_{13}^{2} s_{14} s_{22} = 0.
\end{split}
\end{equation}
% orbit 7: size 60, 7 terms

\begin{equation}\label{eq:n5-orbit-8}
\begin{split}
&|S_{12}^{12}||S_{13}^{45}|-|S_{12}^{45}||S_{13}^{12}|\\
&\qquad = s_{11} s_{14} s_{22} s_{35} - s_{11} s_{14} s_{23} s_{25} - s_{11} s_{15} s_{22} s_{34} \\
&\qquad\quad + s_{11} s_{15} s_{23} s_{24} - s_{12}^{2} s_{14} s_{35} + s_{12}^{2} s_{15} s_{34} \\
&\qquad\quad + s_{12} s_{13} s_{14} s_{25} - s_{12} s_{13} s_{15} s_{24} = 0.
\end{split}
\end{equation}
% orbit 8: size 60, 8 terms

\begin{equation}\label{eq:n5-orbit-9}
\begin{split}
&|S_{12}^{24}||S_{13}^{35}|-|S_{12}^{35}||S_{13}^{24}|\\
&\qquad = s_{12} s_{13} s_{24} s_{35} - s_{12} s_{13} s_{25} s_{34} + s_{12} s_{15} s_{23} s_{34} \\
&\qquad\quad - s_{12} s_{15} s_{24} s_{33} - s_{13} s_{14} s_{22} s_{35} \\
&\qquad\quad + s_{13} s_{14} s_{23} s_{25} + s_{14} s_{15} s_{22} s_{33} - s_{14} s_{15} s_{23}^{2} = 0.
\end{split}
\end{equation}
% orbit 9: size 60, 8 terms

\begin{equation}\label{eq:n5-orbit-10}
\begin{split}
&|S_{12}^{13}||S_{34}^{25}|-|S_{12}^{25}||S_{34}^{13}|\\
&\qquad = s_{11} s_{23}^{2} s_{45} - s_{11} s_{23} s_{24} s_{35} - s_{12} s_{13} s_{23} s_{45} \\
&\qquad\quad + s_{12} s_{13} s_{24} s_{35} - s_{12} s_{13} s_{25} s_{34} \\
&\qquad\quad + s_{12} s_{14} s_{25} s_{33} + s_{13} s_{15} s_{22} s_{34} \\
&\qquad\quad - s_{14} s_{15} s_{22} s_{33} = 0.
\end{split}
\end{equation}
% orbit 10: size 60, 8 terms

\begin{equation}\label{eq:n5-orbit-11}
\begin{split}
&|S_{12}^{12}||S_{13}^{13}|-|S_{12}^{13}||S_{13}^{12}|\;(\text{Sylvester})\\
&\qquad = s_{11}^{2} s_{22} s_{33} - s_{11}^{2} s_{23}^{2} - s_{11} s_{12}^{2} s_{33} \\
&\qquad\quad + 2 s_{11} s_{12} s_{13} s_{23} - s_{11} s_{13}^{2} s_{22} = 0.
\end{split}
\end{equation}
% orbit 11: size 30, 5 terms, Sylvester type

\begin{equation}\label{eq:n5-orbit-12}
\begin{split}
&|S_{12}^{12}||S_{13}^{23}|-|S_{12}^{23}||S_{13}^{12}|\;(\text{Sylvester})\\
&\qquad = s_{11} s_{12} s_{22} s_{33} - s_{11} s_{12} s_{23}^{2} - s_{12}^{3} s_{33} \\
&\qquad\quad + 2 s_{12}^{2} s_{13} s_{23} - s_{12} s_{13}^{2} s_{22} = 0.
\end{split}
\end{equation}
% orbit 12: size 30, 5 terms, Sylvester type

\begin{equation}\label{eq:n5-orbit-13}
\begin{split}
&|S_{12}^{12}||S_{13}^{24}|-|S_{12}^{24}||S_{13}^{12}|\;(\text{Sylvester})\\
&\qquad = s_{11} s_{12} s_{22} s_{34} - s_{11} s_{12} s_{23} s_{24} - s_{12}^{3} s_{34} \\
&\qquad\quad + s_{12}^{2} s_{13} s_{24} + s_{12}^{2} s_{14} s_{23} - s_{12} s_{13} s_{14} s_{22} = 0.
\end{split}
\end{equation}
% orbit 13: size 30, 6 terms, Sylvester type

\begin{equation}\label{eq:n5-orbit-14}
\begin{split}
&|S_{12}^{24}||S_{13}^{34}|-|S_{12}^{34}||S_{13}^{24}|\;(\text{Sylvester})\\
&\qquad = s_{12} s_{14} s_{23} s_{34} - s_{12} s_{14} s_{24} s_{33} - s_{13} s_{14} s_{22} s_{34} \\
&\qquad\quad + s_{13} s_{14} s_{23} s_{24} + s_{14}^{2} s_{22} s_{33} - s_{14}^{2} s_{23}^{2} = 0.
\end{split}
\end{equation}
% orbit 14: size 30, 6 terms, Sylvester type

\begin{equation}\label{eq:n5-orbit-15}
\begin{split}
&|S_{12}^{23}||S_{13}^{45}|-|S_{12}^{45}||S_{13}^{23}|\\
&\qquad = s_{12} s_{14} s_{23} s_{35} - s_{12} s_{14} s_{25} s_{33} - s_{12} s_{15} s_{23} s_{34} \\
&\qquad\quad + s_{12} s_{15} s_{24} s_{33} - s_{13} s_{14} s_{22} s_{35} \\
&\qquad\quad + s_{13} s_{14} s_{23} s_{25} + s_{13} s_{15} s_{22} s_{34} \\
&\qquad\quad - s_{13} s_{15} s_{23} s_{24} = 0.
\end{split}
\end{equation}
% orbit 15: size 30, 8 terms

\begin{equation}\label{eq:n5-orbit-16}
\begin{split}
&|S_{12}^{12}||S_{34}^{35}|-|S_{12}^{35}||S_{34}^{12}|\\
&\qquad = s_{11} s_{22} s_{33} s_{45} - s_{11} s_{22} s_{34} s_{35} - s_{12}^{2} s_{33} s_{45} \\
&\qquad\quad + s_{12}^{2} s_{34} s_{35} - s_{13}^{2} s_{24} s_{25} + s_{13} s_{14} s_{23} s_{25} \\
&\qquad\quad + s_{13} s_{15} s_{23} s_{24} - s_{14} s_{15} s_{23}^{2} = 0.
\end{split}
\end{equation}
% orbit 16: size 30, 8 terms

\begin{equation}\label{eq:n5-orbit-17}
\begin{split}
&|S_{12}^{14}||S_{13}^{15}|-|S_{12}^{15}||S_{13}^{14}|\;(\text{Sylvester})\\
&\qquad = s_{11}^{2} s_{24} s_{35} - s_{11}^{2} s_{25} s_{34} - s_{11} s_{12} s_{14} s_{35} \\
&\qquad\quad + s_{11} s_{12} s_{15} s_{34} + s_{11} s_{13} s_{14} s_{25} \\
&\qquad\quad - s_{11} s_{13} s_{15} s_{24} = 0.
\end{split}
\end{equation}
% orbit 17: size 15, 6 terms, Sylvester type

\begin{equation}\label{eq:n5-orbit-18}
\begin{split}
&|S_{12}^{13}||S_{34}^{24}|-|S_{12}^{24}||S_{34}^{13}|\\
&\qquad = s_{11} s_{23}^{2} s_{44} - s_{11} s_{23} s_{24} s_{34} - s_{12} s_{13} s_{23} s_{44} \\
&\qquad\quad + s_{12} s_{14} s_{24} s_{33} + s_{13} s_{14} s_{22} s_{34} - s_{14}^{2} s_{22} s_{33} = 0.
\end{split}
\end{equation}
% orbit 18: size 15, 6 terms

\begin{equation}\label{eq:n5-orbit-19}
\begin{split}
&|S_{12}^{12}||S_{34}^{34}|-|S_{12}^{34}||S_{34}^{12}|\\
&\qquad = s_{11} s_{22} s_{33} s_{44} - s_{11} s_{22} s_{34}^{2} - s_{12}^{2} s_{33} s_{44} \\
&\qquad\quad + s_{12}^{2} s_{34}^{2} - s_{13}^{2} s_{24}^{2} + 2 s_{13} s_{14} s_{23} s_{24} \\
&\qquad\quad - s_{14}^{2} s_{23}^{2} = 0.
\end{split}
\end{equation}
% orbit 19: size 15, 7 terms

\begin{table}[ht]
\centering
\begin{tabular}{rlrrl}
\hline
\# & representative $(I,J,K,L)$ & orbit size & terms & type \\
\hline
1 & $(12,14,13,24)$ & 120 & 6 & Sylvester \\
2 & $(12,12,13,34)$ & 120 & 8 & non-Sylvester \\
3 & $(12,14,13,25)$ & 120 & 8 & non-Sylvester \\
4 & $(12,12,13,14)$ & 60 & 6 & Sylvester \\
5 & $(12,14,13,45)$ & 60 & 6 & Sylvester \\
6 & $(12,24,13,45)$ & 60 & 6 & Sylvester \\
7 & $(12,14,13,23)$ & 60 & 7 & non-Sylvester \\
8 & $(12,12,13,45)$ & 60 & 8 & non-Sylvester \\
9 & $(12,24,13,35)$ & 60 & 8 & non-Sylvester \\
10 & $(12,13,34,25)$ & 60 & 8 & non-Sylvester \\
11 & $(12,12,13,13)$ & 30 & 5 & Sylvester \\
12 & $(12,12,13,23)$ & 30 & 5 & Sylvester \\
13 & $(12,12,13,24)$ & 30 & 6 & Sylvester \\
14 & $(12,24,13,34)$ & 30 & 6 & Sylvester \\
15 & $(12,23,13,45)$ & 30 & 8 & non-Sylvester \\
16 & $(12,12,34,35)$ & 30 & 8 & non-Sylvester \\
17 & $(12,14,13,15)$ & 15 & 6 & Sylvester \\
18 & $(12,13,34,24)$ & 15 & 6 & non-Sylvester \\
19 & $(12,12,34,34)$ & 15 & 7 & non-Sylvester \\
\hline
\multicolumn{2}{l}{total} & 1005 & & \\
\hline
\end{tabular}
\caption{The 19 $\mathfrak{S}_5$-orbits of exchange relations for $n=5$, $t=2$.}\label{tab:n5-orbits}
\end{table}

\subsubsection*{The planar six-body system, expanded}

For $n=6$ and $t=3$ the same construction produces 18145 generators, none identically zero,
17965 of them distinct up to sign, in 57 orbits under $\mathfrak{S}_6$. Each relation has
degree $6$ in the $s_{ij}$ and between 30 and 71 monomials, so we display
only the two shortest orbit representatives.

\begingroup\allowdisplaybreaks
\begin{align}
&|S_{123}^{134}||S_{124}^{234}|-|S_{123}^{234}||S_{124}^{134}|\;(\text{Sylvester})\nonumber\\
&\qquad = s_{11} s_{13} s_{22} s_{24} s_{33} s_{44} - s_{11} s_{13} s_{22} s_{24} s_{34}^{2} \nonumber\\
&\qquad\quad - s_{11} s_{13} s_{23}^{2} s_{24} s_{44} + 2 s_{11} s_{13} s_{23} s_{24}^{2} s_{34} \nonumber\\
&\qquad\quad - s_{11} s_{13} s_{24}^{3} s_{33} - s_{11} s_{14} s_{22} s_{23} s_{33} s_{44} \nonumber\\
&\qquad\quad + s_{11} s_{14} s_{22} s_{23} s_{34}^{2} + s_{11} s_{14} s_{23}^{3} s_{44} \nonumber\\
&\qquad\quad - 2 s_{11} s_{14} s_{23}^{2} s_{24} s_{34} + s_{11} s_{14} s_{23} s_{24}^{2} s_{33} \nonumber\\
&\qquad\quad - s_{12}^{2} s_{13} s_{24} s_{33} s_{44} + s_{12}^{2} s_{13} s_{24} s_{34}^{2} \nonumber\\
&\qquad\quad + s_{12}^{2} s_{14} s_{23} s_{33} s_{44} - s_{12}^{2} s_{14} s_{23} s_{34}^{2} \nonumber\\
&\qquad\quad + 2 s_{12} s_{13}^{2} s_{23} s_{24} s_{44} - 2 s_{12} s_{13}^{2} s_{24}^{2} s_{34} \nonumber\\
&\qquad\quad - 2 s_{12} s_{13} s_{14} s_{23}^{2} s_{44} + 2 s_{12} s_{13} s_{14} s_{24}^{2} s_{33} \nonumber\\
&\qquad\quad + 2 s_{12} s_{14}^{2} s_{23}^{2} s_{34} - 2 s_{12} s_{14}^{2} s_{23} s_{24} s_{33} \nonumber\\
&\qquad\quad - s_{13}^{3} s_{22} s_{24} s_{44} + s_{13}^{3} s_{24}^{3} \nonumber\\
&\qquad\quad + s_{13}^{2} s_{14} s_{22} s_{23} s_{44} + 2 s_{13}^{2} s_{14} s_{22} s_{24} s_{34} \nonumber\\
&\qquad\quad - 3 s_{13}^{2} s_{14} s_{23} s_{24}^{2} - 2 s_{13} s_{14}^{2} s_{22} s_{23} s_{34} \nonumber\\
&\qquad\quad - s_{13} s_{14}^{2} s_{22} s_{24} s_{33} + 3 s_{13} s_{14}^{2} s_{23}^{2} s_{24} \nonumber\\
&\qquad\quad + s_{14}^{3} s_{22} s_{23} s_{33} - s_{14}^{3} s_{23}^{3} = 0.
\end{align}
\endgroup
% orbit size 45, 30 terms

\begingroup\allowdisplaybreaks
\begin{align}
&|S_{123}^{123}||S_{124}^{134}|-|S_{123}^{134}||S_{124}^{123}|\;(\text{Sylvester})\nonumber\\
&\qquad = s_{11}^{2} s_{22} s_{23} s_{33} s_{44} - s_{11}^{2} s_{22} s_{23} s_{34}^{2} \nonumber\\
&\qquad\quad - s_{11}^{2} s_{23}^{3} s_{44} + 2 s_{11}^{2} s_{23}^{2} s_{24} s_{34} \nonumber\\
&\qquad\quad - s_{11}^{2} s_{23} s_{24}^{2} s_{33} - s_{11} s_{12}^{2} s_{23} s_{33} s_{44} \nonumber\\
&\qquad\quad + s_{11} s_{12}^{2} s_{23} s_{34}^{2} - s_{11} s_{12} s_{13} s_{22} s_{33} s_{44} \nonumber\\
&\qquad\quad + s_{11} s_{12} s_{13} s_{22} s_{34}^{2} + 3 s_{11} s_{12} s_{13} s_{23}^{2} s_{44} \nonumber\\
&\qquad\quad - 4 s_{11} s_{12} s_{13} s_{23} s_{24} s_{34} + s_{11} s_{12} s_{13} s_{24}^{2} s_{33} \nonumber\\
&\qquad\quad - 2 s_{11} s_{12} s_{14} s_{23}^{2} s_{34} + 2 s_{11} s_{12} s_{14} s_{23} s_{24} s_{33} \nonumber\\
&\qquad\quad - s_{11} s_{13}^{2} s_{22} s_{23} s_{44} + s_{11} s_{13}^{2} s_{23} s_{24}^{2} \nonumber\\
&\qquad\quad + 2 s_{11} s_{13} s_{14} s_{22} s_{23} s_{34} - 2 s_{11} s_{13} s_{14} s_{23}^{2} s_{24} \nonumber\\
&\qquad\quad - s_{11} s_{14}^{2} s_{22} s_{23} s_{33} + s_{11} s_{14}^{2} s_{23}^{3} \nonumber\\
&\qquad\quad + s_{12}^{3} s_{13} s_{33} s_{44} - s_{12}^{3} s_{13} s_{34}^{2} \nonumber\\
&\qquad\quad - 2 s_{12}^{2} s_{13}^{2} s_{23} s_{44} + 2 s_{12}^{2} s_{13}^{2} s_{24} s_{34} \nonumber\\
&\qquad\quad + 2 s_{12}^{2} s_{13} s_{14} s_{23} s_{34} - 2 s_{12}^{2} s_{13} s_{14} s_{24} s_{33} \nonumber\\
&\qquad\quad + s_{12} s_{13}^{3} s_{22} s_{44} - s_{12} s_{13}^{3} s_{24}^{2} \nonumber\\
&\qquad\quad - 2 s_{12} s_{13}^{2} s_{14} s_{22} s_{34} + 2 s_{12} s_{13}^{2} s_{14} s_{23} s_{24} \nonumber\\
&\qquad\quad + s_{12} s_{13} s_{14}^{2} s_{22} s_{33} - s_{12} s_{13} s_{14}^{2} s_{23}^{2} = 0.
\end{align}
\endgroup
% orbit size 180, 32 terms

\begin{table}[!b]
\centering\small
\begin{tabular}{rlrrl}
\hline
\# & representative $(I,J,K,L)$ & orbit size & terms & type \\
\hline
1 & $(123,134,124,234)$ & 45 & 30 & Sylvester \\
2 & $(123,123,124,134)$ & 180 & 32 & Sylvester \\
3 & $(123,123,124,124)$ & 90 & 32 & Sylvester \\
4 & $(123,134,124,345)$ & 360 & 42 & Sylvester \\
5 & $(123,123,124,135)$ & 180 & 43 & Sylvester \\
6 & $(123,125,124,135)$ & 720 & 44 & Sylvester \\
7 & $(123,135,124,345)$ & 360 & 44 & Sylvester \\
8 & $(123,123,124,125)$ & 180 & 44 & Sylvester \\
9 & $(123,135,124,145)$ & 180 & 44 & Sylvester \\
10 & $(123,125,124,156)$ & 360 & 46 & Sylvester \\
11 & $(123,135,124,156)$ & 360 & 46 & Sylvester \\
12 & $(123,135,124,356)$ & 360 & 46 & Sylvester \\
13 & $(123,156,124,356)$ & 360 & 46 & Sylvester \\
14 & $(123,156,124,256)$ & 90 & 46 & Sylvester \\
15 & $(123,123,145,145)$ & 90 & 46 & non-Sylvester \\
16 & $(123,125,124,126)$ & 45 & 46 & Sylvester \\
\hline
\end{tabular}
\caption{The $57$ $\mathfrak{S}_6$-orbits of exchange relations for $n=6$,
$t=3$: orbits $1$--$16$.}\label{tab:n6-orbits}
\end{table}

\begin{table}[p]
\centering\small
\begin{tabular}{rlrrl}
\hline
\# & representative $(I,J,K,L)$ & orbit size & terms & type \\
\hline
17 & $(123,356,124,456)$ & 45 & 46 & Sylvester \\
18 & $(123,123,456,456)$ & 10 & 46 & non-Sylvester \\
19 & $(123,123,124,345)$ & 360 & 52 & non-Sylvester \\
20 & $(123,125,124,345)$ & 180 & 52 & non-Sylvester \\
21 & $(123,125,124,134)$ & 360 & 54 & non-Sylvester \\
22 & $(123,134,124,235)$ & 720 & 56 & non-Sylvester \\
23 & $(123,124,145,135)$ & 90 & 56 & non-Sylvester \\
24 & $(123,123,124,145)$ & 720 & 57 & non-Sylvester \\
25 & $(123,123,124,356)$ & 180 & 58 & non-Sylvester \\
26 & $(123,135,124,245)$ & 360 & 59 & non-Sylvester \\
27 & $(123,123,124,156)$ & 360 & 60 & non-Sylvester \\
28 & $(123,124,145,235)$ & 360 & 61 & non-Sylvester \\
29 & $(123,135,124,346)$ & 720 & 62 & non-Sylvester \\
30 & $(123,125,124,356)$ & 360 & 62 & non-Sylvester \\
31 & $(123,135,124,236)$ & 360 & 62 & non-Sylvester \\
32 & $(123,156,124,345)$ & 360 & 62 & non-Sylvester \\
33 & $(123,134,124,156)$ & 180 & 62 & non-Sylvester \\
34 & $(123,125,124,136)$ & 720 & 63 & non-Sylvester \\
35 & $(123,135,124,256)$ & 720 & 63 & non-Sylvester \\
36 & $(123,135,124,146)$ & 360 & 63 & non-Sylvester \\
37 & $(123,123,145,146)$ & 180 & 63 & non-Sylvester \\
38 & $(123,123,145,245)$ & 180 & 63 & non-Sylvester \\
39 & $(123,123,145,456)$ & 180 & 63 & non-Sylvester \\
40 & $(123,135,124,456)$ & 720 & 64 & non-Sylvester \\
41 & $(123,134,124,356)$ & 360 & 64 & non-Sylvester \\
42 & $(123,123,145,246)$ & 360 & 66 & non-Sylvester \\
43 & $(123,123,124,456)$ & 180 & 66 & non-Sylvester \\
44 & $(123,124,145,356)$ & 360 & 68 & non-Sylvester \\
45 & $(123,125,124,346)$ & 180 & 68 & non-Sylvester \\
46 & $(123,134,124,256)$ & 180 & 68 & non-Sylvester \\
47 & $(123,126,145,346)$ & 720 & 69 & non-Sylvester \\
48 & $(123,124,145,256)$ & 720 & 70 & non-Sylvester \\
49 & $(123,246,145,356)$ & 90 & 70 & non-Sylvester \\
50 & $(123,135,124,246)$ & 360 & 71 & non-Sylvester \\
51 & $(123,124,145,136)$ & 360 & 71 & non-Sylvester \\
52 & $(123,124,145,236)$ & 360 & 71 & non-Sylvester \\
53 & $(123,126,145,345)$ & 360 & 71 & non-Sylvester \\
54 & $(123,126,145,456)$ & 360 & 71 & non-Sylvester \\
55 & $(123,126,145,245)$ & 180 & 71 & non-Sylvester \\
56 & $(123,236,145,456)$ & 45 & 71 & non-Sylvester \\
57 & $(123,124,456,356)$ & 45 & 71 & non-Sylvester \\
\hline
\multicolumn{2}{l}{total} & 17965 & & \\
\hline
\end{tabular}
\caption{Orbits $17$--$57$ of the $n=6$, $t=3$ system (continuation of
Table~\ref{tab:n6-orbits}), with the total count.}
\end{table}

\clearpage

\begin{remark}
The pattern is uniform in $t$: the exchange relations are the equations of the Veronese variety
read in the coordinates $\det(S_I^J)$, and their content in the variables $s_{ij}$ is exactly
$\operatorname{rank}(S)\le t$. In particular, since
$\operatorname{bwc}(x)\le n-\delta(x)-1$, the two systems expanded here, both
with $t=n-3$, are satisfied by \emph{every} central configuration of
dimension $\delta(x)\ge2$: Dziobek, spatial, and planar configurations,
vertically degenerate or not, as well as the regular simplex. Only collinear
configurations of maximal Brehm--Wintner--Conley dimension $n-2$ escape, while
vertically degenerate collinear ones with $\operatorname{bwc}\le n-3$ are
again included. The expansion is therefore useful less as a system to be solved
than as a dictionary between the ambient Veronese picture, where the geometry is transparent,
and the determinantal picture, where the equations are of minimal degree.
\end{remark}

\clearpage

\section*{Acknowledgments}

The author thanks Alain Albouy, for inspiring conversations in Itajubá and Recife, and Ernesto Pérez-Chavela, for equally inspiring conversations in Rio de Janeiro and Mexico, on this subject and on much else; Eduardo Leandro, for the invitation to give a series of talks on this subject at the Mathematical Physics Seminar of the Department of Mathematics at UFPE; and Rafael Holanda and Alan Muniz, organizers of the Sagui Seminar on Commutative Algebra and Algebraic Geometry, for the invitation to present the ideas that eventually led to this work. Finally, this paper is written in memory of Bob, the author's dog and faithful companion throughout much of its writing.

\bibliographystyle{amsplain}

\bibliography{refs}

\end{document}